\documentclass[a4paper, 12pt, reqno]{guthesis}

\usepackage{stmaryrd} 
\usepackage{float}
\usepackage{graphicx}
\usepackage{url}
\usepackage{amssymb,amsmath,amsthm,verbatim,longtable,latexsym, mathrsfs}
\usepackage{mathabx}
\usepackage[all, dvips, 2cell]{xy}
\UseTwocells
\usepackage{color}
\usepackage{hyperref}
\usepackage{times}
\usepackage{bbm}
\usepackage{tikz}

\def\quotient#1#2{%
    \raise1ex\hbox{$#1$}\Big/\lower1ex\hbox{$#2$}%
}

\newcommand{\smxy}[1]{{\text{\scriptsize$#1$}}}
\newcommand{\smmxy}[1]{{\text{\small$#1$}}}

\newcommand{\Decr}{\operatorname{Dec}^{\operatorname{r}}}
\newcommand{\Decl}{\operatorname{Dec}^{\operatorname{l}}}

\newcommand{\cG}{\mathcal G}
\newcommand{\cK}{\mathcal K}
\newcommand{\HH}{\operatorname{H}}
\newcommand{\C}{\operatorname{C}}
\newcommand{\D}{\operatorname{D}}
\newcommand{\M}{\operatorname{M}}

\newcommand{\HC}{\operatorname{HC}}
\newcommand{\lto}{\longrightarrow}

\newcommand{\Hom}{\operatorname{Hom}}
\newcommand{\cA}{\mathcal{A}}
\newcommand{\cB}{\mathcal{B}}
\newcommand{\cC}{\mathcal{C}}
\newcommand{\cD}{\mathcal{D}}
\newcommand{\cE}{\mathcal{E}}
\newcommand{\cH}{\mathcal{H}}
\newcommand{\cx}{\mathcal{M}}
\newcommand{\cX}{\mathcal{Y}}
\newcommand{\cY}{\mathcal{Z}}
\newcommand{\cV}{\mathcal{V}}

\newcommand{\cM}{\mathcal{M}}

\newcommand{\sC}{\mathscr{C}}
\newcommand{\sD}{\mathscr{D}}
\newcommand{\sR}{\mathscr{R}}
\newcommand{\sM}{\mathscr{M}}
\newcommand{\sN}{\mathscr{N}}
\newcommand{\sL}{\mathscr{L}}
\newcommand{\sF}{\mathscr{F}}
\newcommand{\sU}{\mathscr{U}}

\newcommand{\rBB}{\mathrm B} 
\newcommand{\rC}{\mathrm C}
\newcommand{\rCC}{\mathrm C} 

\newcommand{\rF}{\mathrm F}

\newcommand{\rH}{\mathrm H}
\newcommand{\rM}{\mathrm M}
\newcommand{\rN}{\mathrm N}

\newcommand{\rT}{\mathrm T}

\newcommand{\rV}{\mathrm V}

\newcommand{\bG}{\mathbb G}

\newcommand{\cMix}{\operatorname{Mix}} 
\newcommand{\cDist}{\operatorname{Dist}} 
\newcommand{\cSet}{\mathbf{Set}} 
\newcommand{\cMod}{\mbox{-}\mathbf{Mod}}
\newcommand{\cComod}{\mbox{-}\mathbf{Comod}}
\newcommand{\mix}{\mathbf{Mix}}
\newcommand{\dist}{\mathbf{Dist}}
\newcommand{\kmod}{k\mbox{-}\mathbf{Mod}}
\newcommand{\amoda}{A\mbox{-}\mathbf{Mod}\mbox{-}A}
\newcommand{\AmodA}{\mathcal A \mbox{-}\mathbf{Mod}\mbox{-}\mathcal A}
\newcommand{\amod}{A\mbox{-}\mathbf{Mod}}

\newcommand{\aemod}{\Ae\mbox{-}\mathbf{Mod}}
\newcommand{\modae}{\mathbf{Mod}\mbox{-}\Ae}

\newcommand{\cmd}{\operatorname{Cmd}}
\newcommand{\mnd}{\operatorname{Mnd}}

\newcommand{\set}{\mathbf{Set}}
\newcommand{\cat}{\mathbf{Cat}}
\newcommand{\twocat}{\operatorname{2-\mathbf{Cat}}}

\newcommand{\sq}{\mathrm{Sq}}

\newcommand{\Tor}{{\rm Tor}}

\newcommand{\Aop}{{A^*}}
\newcommand{\Ae}{{A^\mathrm{e}}}

\renewcommand\epsilon{\varepsilon}
\renewcommand\phi{\varphi} 
\newcommand{\lmapsto}{\longmapsto}
\newcommand{\colim}{\operatorname{colim}}
\newcommand{\im}{\operatorname{im}}

\newcommand{\ev}{\mathrm{ev}}

\newcommand{\lact}{\smalltriangleright}
\newcommand{\ract}{\smalltriangleleft}
\newcommand{\blact}{\blacktriangleright}
\newcommand{\bract}{\blacktriangleleft}

\numberwithin{equation}{section}

\newcommand{\STRINGDIAGRAM}{\xymatrix@R=10pt@C=10pt@H=0pt@W=0pt@M=0pt}

\newcommand{\bprod}{\ar@{-}[dd] &
\ar@{-}@(d,r)@/^3mm/[dl]}
\newcommand{\bkopr}{\ar@{-}[uu]
&\ar@{-}@(u,l)@/_3mm/[ul]}

\newcommand{\bmult}{\ar@{-}@(d,r)@/^2mm/}
\newcommand{\bkomu}{\ar@{-}@(u,r)@/_1.7mm/}
\newcommand{\bbigkomu}{\ar@{-}@(u,r)@/_4mm/}

\newcommand{\FERMION}{\ar@{-}@(d,u)}
\newcommand{\FERMIONddl}{\ar@{-}@(d,u)[ddl]}
\newcommand{\FERMIONddddr}{\ar@{-}@(d,u)[ddddr]}
\newcommand{\FERMIONddddrr}{\ar@{-}@(d,u)[ddddrr]}
\newcommand{\FERMIONddddl}{\ar@{-}@(d,u)[ddddl]}
\newcommand{\FERMIONddddll}{\ar@{-}@(d,u)[ddddll]}
\newcommand{\FERMIONddddlll}{\ar@{-}@(d,u)[ddddlll]}
\newcommand{\BOSON}{\ar@{~}}

\newtheorem{thm}{Theorem}[section]
\newtheorem{prop}[thm]{Proposition}
\newtheorem{lem}[thm]{Lemma}
\newtheorem{cor}[thm]{Corollary}

\theoremstyle{definition}
\newtheorem{exa}[thm]{Example}
\newtheorem{defn}[thm]{Definition}
\newtheorem{ques}{Question}

\newtheorem{rem}[thm]{Remark}

\CompileMatrices

\title{2-categories and cyclic homology}
\author{Paul Slevin \\ BSc (Hons), MASt}
\submissionmonth{May}
\submissionyear{2016}

\begin{document}
\maketitle

\newlength\longest

\clearpage

\thispagestyle{empty}
\null\vfill

\settowidth\longest{\LARGE\itshape ``Das Mittelding, das Wahre in allen Sachen}
\begin{center}

  \Large{to}\par%
  \Large{Anthony Thomas McMahon\\
  1933--1968}

\end{center}
\vfill\vfill

\clearpage

\thispagestyle{empty}
\null\vfill

\settowidth\longest{\LARGE\itshape ``Das Mittelding, das Wahre in allen Sachen}
\begin{center}
\parbox{\longest}{%
  \raggedright{\LARGE\itshape%
   Das Mittelding, das Wahre in allen Sachen kennt und sch\"atzt man jetzt nimmer; um Beifall zu erhalten mu{\ss}
   man Sachen schreiben, die so verst\"andlich sind, da{\ss} es ein Fiacre nachsingen k\"onnte, oder so unverst\"andlich, da{\ss} es ihnen,
   eben weil es kein vern\"unftiger Mensch verstehen kann, gerade eben deswegen gef\"allt.\par\bigskip
  }
  \raggedleft\Large\textsc{Wolfgang Amadeus Mozart}\par%
  \raggedleft\normalsize{December 28 1782}
}
\end{center}
\vfill\vfill

\clearpage

\chapter*{Statement}
I certify that all material in this thesis that is not my own work has been identified and that no material has previously been submitted and approved for the award of a degree by the University of Glasgow or any other institution.

\clearpage

\chapter*{Abstract}
The topic of this thesis is the application of distributive laws between comonads to the theory of cyclic homology. The work herein is based on the three papers~\cite{2, 1, woohoo}, to which the current author has contributed. Explicitly, our main aims are:
\begin{itemize}
\item To study
how the cyclic homology of associative algebras and
of Hopf algebras in the original sense of Connes and
Moscovici arises from a distributive law, and to clarify
the r\^ole of different notions of bimonad in this
generalisation.

\item To extend the procedure of twisting the cyclic
homology of a unital associative algebra to any duplicial object
defined by a distributive law.

\item To study the universality of B{\"o}hm and {\c S}tefan's approach
to constructing duplicial objects, which we do in terms of a 2-categorical
generalisation of Hochschild (co)homology.

\item To characterise those categories whose nerve admits a duplicial structure.
\end{itemize}
%
%

\chapter*{Acknowledgements}
Firstly, I would like to thank my family and friends for the immense amount of love and support offered to me whilst undertaking this venture. I would not have been able to produce this document if not for your patience and understanding in helping to see things clearly when maybe I could not.

Secondly, I must thank my second family at the University of Glasgow for providing the greatest learning environment to me for the last ten years. I will not forget the dedication and passion of my lecturers in their commitment to my learning in my undergraduate years, nor will I forget the insane level of support from my PhD supervisors, Ulrich Kr\"ahmer and Tara Brendle.

Thirdly, thank you to my friends and colleagues in Australia for the incredible hospitality. I learned a great deal of category theory in my two visits to Macquarie University, as well as the fact that it is possible to get sick of hot weather (but seriously, thank you for allowing me to experience two summers this year).

Finally, thank you to those in the musical part of my life. Art is the most crucial aspect of our humanity, and I have been lucky enough to be able to practise two of them.

I gratefully acknowledge the support of the Engineering and Physical Sciences Research Council (grants K503058/1 and P505534/1) in providing the funding necessary to complete this PhD. Thank you also to the Centre of Australian Category Theory and the University of Glasgow College of Science and Engineering for providing support for my research project in Australia.
\tableofcontents
\chapter{Introduction}\label{INTRO}
We begin by giving some context and background for the thesis in terms of the aims presented in the abstract, followed by the conventions that we use throughout.


\section{Background and aims}

The Dold-Kan correspondence generalises chain complexes
in abelian categories to general simplicial objects,
and thus homological algebra to homotopical algebra.
The classical homology theories
defined by an augmented algebra (such as group, Lie
algebra, Hochschild, de Rham and Poisson homology)
become expressed as the homology of suitable comonads
$T$, defined via simplicial objects
$\rCC_T(N,M)$ obtained from the bar
construction (see e.g.~\cite[\S6.5]{MR1269324}). Here $M,N$ are suitable functors
providing homology coefficients.

Distributive laws between monads were originally defined by Beck in~\cite{MR0241502}
and correspond to monad structures on the composite of the two underlying endofunctors. They
have found many applications in mathematics as well as computer science; see
e.g.~\cite{MR2520969, MR2504663, MR1692751, MR2220892, MR2784770}.

The study of monads and comonads arose from homological algebra, which prompts the question: can we go back and apply distributive laws to homological algebra? The answer is yes. Connes' cyclic homology created a new paradigm
of homology theories defined in terms of mixed
complexes \cite{MR883882,MR826872}. The homotopical
counterparts are cyclic
\cite{MR777584} or more generally duplicial
objects \cite{MR826872,MR885102}. B\"ohm and \c Stefan \cite{MR2415479} showed how
$\rCC_T(N,M)$ becomes duplicial in the
presence of a second comonad $S$ together with a distributive law between $T$ and $S$, which is compatible in
a suitable sense with $N$ and $M$.

The paradigmatic example of such a cyclic homology theory is the cyclic homology $\HC(A)$ of
a unital associative algebra $A$~\cite{MR823176, MR695483}. This leads us to our first aim:
\begin{itemize}
\item To study
how the cyclic homology of associative algebras and
of Hopf algebras in the original sense of Connes and
Moscovici \cite{MR1657389}
fits into the monadic formalism  of B{\" o}hm and {\c S}tefan, extending
the construction from \cite{MR2803876}, and to clarify
the r\^ole of different notions of bimonad in this
generalisation (Chapters~\ref{DISTRIBUTIVE},~\ref{CYCLIC},~\ref{EXAMPLES}).
\end{itemize}

It was observed by Kustermans, Murphy,
and Tuset~\cite{MR1943179} that the functor $\HC$ can be twisted by automorphisms of $A$. In fact, this twisted cyclic homology
occurs as an instance of B{\" o}hm and {\c S}tefan's construction. Thus, our second aim:
\begin{itemize}
\item To extend the procedure of twisting the cyclic
homology of a unital associative algebra to any duplicial object
defined by a distributive law (Chapters~\ref{DISTRIBUTIVE},~\ref{CYCLIC},~\ref{EXAMPLES}).
\end{itemize}

The construction of simplicial objects via the bar resolution is universal in the sense that comonads on a category $\cB$ correspond to strict monoidal functors $\Delta_+^* \to [\cB, \cB]$ where
$\Delta_+$  denotes the augmented simplicial category (cf.~Definition~\ref{simpldef}), and $*$ denotes the opposite category. Using the symmetric monoidal closed structure of $\cat$, we obtain the bar resolution of the corresponding comonad as a functor $\cB \to [\Delta_+^*, \cB]$. What has been missing in the literature so far is a similar universal description of the situation for duplicial objects. Our third main aim is then:
\begin{itemize}
\item To study the universality of B{\"o}hm and {\c S}tefan's approach
to constructing duplicial objects, which we do in terms of a 2-categorical
generalisation of Hochschild (co)homology (Chapter~\ref{AUSTRALIA}).
\end{itemize}

The nerve functor $N \colon \cat \to [\Delta^*, \set]$ is full and faithful, embedding categories into simplicial sets. Thus, it makes sense to discuss
a simplicial or duplicial structure on a category. This leads to our final main aim:

\begin{itemize}
\item To characterise those categories whose nerve admits a duplicial structure (Chapter~\ref{AUSTRALIA}).
\end{itemize}
In addition to our main goals, we have some subsidiary aims:
\begin{itemize}
\item To develop some 2-category theory to assist with the above aims (Chapter~\ref{MONAD}).

\item To give a wide variety of examples (Chapters~\ref{DISTRIBUTIVE},~\ref{EXAMPLES}).

\item To pose some questions that the author was not able to answer (Chapter~\ref{GRANDFINALE}).
\end{itemize}

\section{Conventions}

We assume that the reader has familiarity with ordinary category theory (as, for example, in~\cite{MR1712872, MR1291599}), as well as the very basics of the theory of (co)modules over (co)algebras (see e.g.~\cite{MR2455920, MR1269324, MR2012570}). When natural transformations (and similar notions) appear, by abuse of notation we often write the same symbol for a natural transformation and its components, e.g.\ $\alpha \colon F \Rightarrow G$ and $\alpha \colon FX \to GX$. When a commutative diagram is given with unlabelled variables, we rather mean the collection of commutative diagrams where the variables are objects in the relevant category of interest. We completely ignore issues of size throughout (for more information on size considerations, see~\cite[p.~viii]{MR2178101} or~\cite[Ch.~I]{MR1712872} for a more serious discussion).

Each chapter contains relevant background material as well as original work of the author.
Chapters~2,~3,~4 and 6 contain joint work of the author with Ulrich Kr\"ahmer and Niels Kowalzig, although here some of the results of the relevant papers~\cite{1,2} are developed further and more background detail is given.
Chapter~5 contains joint work carried out by the author, Richard Garner and Steve Lack in~\cite{woohoo}.
At the beginning of each chapter it is made precise which parts of the material consist of original work.

\chapter{Monads and comonads}\label{MONAD}

In this chapter, we review all the 2-category theory needed for the thesis. After giving basic definitions we study monads internal to 2-categories and related concepts. Sections~\ref{2-categories} and~\ref{monadsandcomonads} contain the required definitions and preliminary results. Section~\ref{liftingthroughadjunctions}, concerning how one obtains distributive laws from certain types of squares, contains original work which is an expansion of~\cite[\S2.1--2.16]{1}. In Section~\ref{emem} we explicitly give the constructions of the previous sections in the 2-category $\cat$.

\section{Preliminaries}\label{2-categories}
We begin by recalling some fundamental notions of 2-category theory.
\subsection{2-categories}
Let $\mathbbm{1}$ denote the terminal category, containing one object and one (identity) morphism.

\begin{defn}
A 2-\emph{category} $\sC$ consists of
\begin{itemize}
\item a class $|\sC|$ whose elements we call 0-\emph{cells}
\item for any $\cA, \cB \in |\sC|$, a category $\sC(\cA, \cB)$ whose objects we call 1-\emph{cells}, whose morphisms
we call 2-\emph{cells}, and whose composition law we call \emph{vertical composition}
\item for each 0-cell $\cA$, a functor $u_\cA \colon \mathbbm{1} \to \sC(\cA, \cA)$, called the \emph{unit}
\item for any $\cA, \cB, \cC$ in $|\sC|$, a functor $\circ_{\cA, \cB, \cC} \colon \sC(\cB, \cC) \times \sC(\cA, \cB) \to \sC(\cA, \cC)$, called
\emph{horizontal composition}
\end{itemize}
satisfying associativity and unitality conditions, that is, commutativity of the two diagrams
$$
\xymatrix@C=3em{
\sC(\cC, \cD) \times \sC(\cB, \cC) \times \sC(\cA, \cB)\ar[d]_-{\circ_{\cB, \cC, \cD} \times 1} \ar[rr]^-{1 \times \circ_{\cA, \cB, \cC}} && \sC(\cC, \cD) \times \sC(\cA, \cC) \ar[d]^-{\circ_{\cA, \cC, \cD}} \\
\sC(\cB, \cD) \times \sC(\cA, \cB) \ar[rr]_-{\circ_{\cA, \cB, \cD} } && \sC(\cA, \cD)
}
$$
$$
\xymatrix@C=3em{
\sC(\cA, \cB) \ar@{=}[drr] \ar[rr]^-{u_\cA \times 1} \ar[d]_-{1 \times u_\cB} && \sC(\cB, \cB) \times \sC(\cA, \cB) \ar[d]^-{\circ_{\cA, \cB, \cB}} \\
\sC(\cA, \cB) \times \sC(\cA, \cA) \ar[rr]_-{\circ_{\cA, \cA, \cB} } & & \sC(\cA, \cB)
}
$$
for all $\cA, \cB, \cC, \cD \in |\sC|$.
\end{defn}
From this point onward, we omit the subscripts on the functors $\circ$ and $u$, much as we do for natural transformations.

We denote a 1-cell $F$ in $\sC(\cA, \cB)$ by $F \colon \cA \to \cB$, and a 2-cell $\alpha$ between $F,F'$ is denoted by $\alpha \colon F \Rightarrow F'$. We denote both horizontal and vertical composition by concatenation, or occasionally by the symbol $\circ$. It is always clear from the context to which type of composition we refer.

A 2-cell inside a diagram of 1-cells denotes a 2-cell between their horizontal composites, e.g.\
$$
\xymatrix{
\cA \ar[r]^-F \ar[d]_-H & \cB \ar[d]^-G \ar@{}[dl]^(.25){}="a"^(.75){}="b" \ar@{=>}_-\alpha "a";"b"\\
\cC \ar[r]_-K \ar[r]_-K & \cD
}
$$
means that $\alpha \colon GF \Rightarrow KH$ is a 2-cell. In diagrams consisting only of 2-cells, we usually abandon the double arrows $\Rightarrow$ in favour of regular arrows $\to$ for the sake of readability.

We use the symbol $1$ to denote both the identity 1-cell and 2-cell, i.e.\ the images of the unique morphism and object respectively, under the unit functor $u$. However, when we write a horizontal composite involving an identity 2-cell, we rather write the \emph{corresponding 1-cell}. We do this because it makes it easy to reference individual cells, while also making diagrams easier to interpret; for example, given a diagram
$$
\xymatrix{
\cA \rtwocell^F_{F'}{\ \alpha}  & \cB \ar[r]^-G & \cC
}
$$
we denote the horizontal composite with the identity by
$$
\xymatrix{
\cA \rrtwocell^{GF}_{GF'}{\ \ \ G \alpha} && \cC
}
$$

There are various duals one obtains by reversing some of the structure in a 2-category. We denote by $\sC^*$ the
2-category obtained by reversing all 1-cells in
$\sC$, and by $\sC_*$ we denote the
2-category obtained by reversing all 2-cells in
$\sC$. Of course, $(\sC_*)^* = (\sC^*)_*$ so there is no harm in writing ${\sC}^*_*$ to denote either of these duals.

We may view an ordinary category $\cC$ as a 2-category, where for two objects $A,B$ in $\cC$, the category $\cC(A,B)$ is the discrete category on the morphisms $A \to B$. Using the above notation, $\cC^*$ can thus be viewed as the dual (or opposite) category to $\cC$ in the ordinary sense.
\begin{exa}
The paradigmatic example of a 2-category is $\cat$. This is the 2-category whose 0-, 1- and 2-cells are categories, functors and natural transformations respectively. 
Most examples of 2-categories in this thesis appear as constructions based on $\cat$.
\end{exa}
\subsection{2-functors and 2-natural transformations}
\begin{defn}
Let $\sC, \sD$ be 2-categories. A 2-\emph{functor} $\Phi \colon \sC \to \sD$ consists of
\begin{itemize}
\item a function $\Phi \colon |\sC| \to |\sD|$
\item for each pair $\cA, \cB \in \sC$, a functor $\Phi_{\cA, \cB} \colon \sC(\cA, \cB) \to \sD(\Phi\cA, \Phi\cB)$
\end{itemize}
that are compatible with both horizontal and vertical composition, that is, the diagrams
$$
\xymatrix{
\sC(\cB, \cC) \times \sC(\cA, \cB) \ar[r]^-{\circ} \ar[d]_-{\Phi_{\cB, \cC} \times \Phi_{\cA, \cB}} & \sC(\cA, \cC) \ar[d]^-{\Phi_{\cA, \cC}} \\
\sD(\Phi\cB, \Phi\cC) \times \sD(\Phi\cA, \Phi\cB) \ar[r]_-\circ & \sD(\Phi\cA, \Phi\cC)
}
\qquad
\xymatrix{
\mathbbm{1} \ar[r]^-u \ar[dr]_-u & \sC(\cA,\cA) \ar[d]^-{\Phi_{\cA, \cA}} \\
& \sD(\Phi\cA, \Phi\cA)
}
$$
commute for all $\cA, \cB, \cC \in |\sC|$.
\end{defn}
Thus a 2-functor is simply a mapping of cells between 2-categories that preserves commutative diagrams of 1-cells and 2-cells, as well as identities.
We now omit the subscripts on a 2-functor $\Phi$, similar again to the case for the composition and unit of a 2-category, and natural transformations.
\begin{exa}\label{hom2functors}
For each 0-cell $\cA$ in $\sC$, there are \emph{hom 2-functors}
\begin{align*}
\sC(\cA, -) \colon \sC \to \cat \\
\sC(-, \cA) \colon \sC^* \to \cat
\end{align*}
defined in the obvious way.
\end{exa}

There is also a notion of morphism between 2-functors:

\begin{defn}
Let $\Phi, \Phi' \colon \sC \to \sD$ be 2-functors. A 2-\emph{natural transformation}, which we denote by $\nu \colon \Phi \Rightarrow \Phi'$, is given by, for each $\cA \in |\sC|$, a 1-cell $\nu_{\cA} \colon \Phi\cA \to \Phi'\cA$ such that the diagram
$$
\xymatrix{
\sC(\cA, \cB) \ar[rr]^-\Phi \ar[d]_-{\Phi '} && \sD(\Phi\cA, \Phi\cB)\ar[d]^-{\sD(\Phi\cA, \nu_{\cB})} \\
\sD(\Phi'\cA, \Phi'\cB) \ar[rr]_-{\sD(\nu_{\cA}, \Phi'\cB)} && \sD(\Phi\cA, \Phi'\cB)
}
$$
commutes for all $\cA, \cB \in |\sC|$.\end{defn}
Again, we usually omit the indices on the 1-cell components of 2-natural transformations.
\begin{exa}
There is a 2-category $\twocat$ whose 0-cells are 2-categories, whose 1-cells are 2-functors and whose 2-cells are 2-natural transformations.
\end{exa}

\subsection{Adjunctions}
One advantage of studying 2-categories as opposed to just ordinary categories is that it allows us to study certain phenomena as being \emph{internal}. For example, instead of thinking of an adjunction \emph{between} categories, we can think of an adjunction \emph{inside} a 2-category:
\begin{defn}
Let $F \colon \cA \to \cB$ and $U \colon \cB \to \cA$ be 1-cells in a 2-category $\sC$. We say that $F$ is \emph{left adjoint} to $U$ if there are 2-cells $\eta \colon 1 \Rightarrow UF$, $\epsilon \colon FU \Rightarrow 1$ called the \emph{unit} and \emph{counit} respectively, such that the two \emph{triangle identities} hold, i.e.\ the diagrams
$$
\xymatrix{
F \ar[r]^-{F \eta} \ar@{=}[dr] & FUF\ar[d]^-{\epsilon F} \\
& F
}\qquad
\xymatrix{
U \ar[r]^-{\eta U} \ar@{=}[dr] & UFU \ar[d]^-{U \epsilon} \\
& U
}
$$
commute. In this situation, we also say that $U$ is \emph{right adjoint to} $F$, and that $F\dashv U$ is an \emph{adjunction}.
\end{defn}
\begin{exa}
An adjunction in $\cat$ is just an ordinary adjunction between categories.
\end{exa}
\begin{exa}
An adjunction in $\twocat$ is called a \emph{2-adjunction}. A 2-adjunction
$$
\xymatrix{
\sC \ar@{}[rr]|-{\perp}\ar@/^0.5pc/[rr]^-\Phi & & \ar@/^0.5pc/[ll]^-\Theta \sD
}
$$
with unit $\nu \colon 1 \Rightarrow \Theta\Phi$ and counit $\xi \colon \Phi \Theta \Rightarrow 1$ can be described equivalently as having 2-natural isomorphisms
$$
\sC(\cA, \Theta \cE) \cong \sD(\Phi\cA, \cE).
$$
explicitly given by
\begin{align*}
\xymatrix{
\cA \rrtwocell^X_{X'}{\alpha} & & \Theta\cE & \lmapsto & \Phi\cA\rrtwocell^{\Phi X}_{\Phi X'}{\ \ \Phi \alpha} & & \Phi \Theta \cE \ar[r]^-{\xi \cE} & \cE
} \\
\xymatrix{
\Phi \cA \rrtwocell^Z_{Z'}{\beta} & & \cE & \lmapsto & \cA \ar[r]^-{\nu \cA} & \Theta\Phi\cA \rrtwocell^{\Theta Z}_{\Theta Z'}{\ \ \Theta \beta} & &  \Theta \cE
}
\end{align*}
\end{exa}

Often 2-functors preserve these internal properties. Indeed:
\begin{prop}
All 2-functors preserve adjunctions.
\end{prop}
\begin{proof}
If $F \dashv U$ is an adjunction with unit $\eta$ and counit $\epsilon$, and $\Phi$ is a 2-functor, then the triangle identities for $\Phi\eta$ and $\Phi \epsilon$ are satisfied since 2-functors preserve identities and all forms of composition. Hence $\Phi F \dashv \Phi U$.
\end{proof}

\section{Monads and comonads in 2-categories}\label{monadsandcomonads}
Throughout the remainder of this chapter, let $\sC$ be a 2-category. In subsequent chapters, the 2-categorical constructions presented here
are considered only in the case  $\sC = \cat$, and are explicitly described in that way in Section~\ref{emem}, so if the reader is not entirely comfortable with the language of 2-categories, they may replace appearances of 0-, 1- and 2-cells in $\sC$ with the words `category', `functor' and `natural transformation' respectively.

We restate, in our terminology, some of the definitions and results involving monads in \cite{MR0299653}. We also emphasise results for comonads since they and their interplay with monads are an important topic in later chapters.

\subsection{The 2-category of monads}
\begin{defn}
Let $\cA$ be a 0-cell in $\sC$. A \emph{monad on $\cA$ in $\sC$} is a triple $(B, \mu, \eta)$ consisting of a 1-cell $B \colon \cA \to \cA$ together with 2-cells $\eta \colon 1 \Rightarrow B$, $\mu \colon BB \Rightarrow B$, called the \emph{unit} and \emph{multiplication} respectively, such that the following two diagrams commute:
$$
\xymatrix{
BBB \ar[r]^-{B \mu} \ar[d]_-{\mu B} & BB \ar[d]^-\mu \\
BB \ar[r]_-\mu & B
}\qquad
\xymatrix{
B \ar@{=}[dr] \ar[r]^-{B \eta} \ar[d]_-{\eta B} & BB \ar[d]^-{\mu} \\
BB \ar[r]_-{\mu} & B
}
$$
\end{defn}
By abuse of notation, we refer to a monad by its underlying 1-cell, and we always use the symbols $\eta$ and $\mu$ to refer to the unit and multiplication of an arbitrary monad.

Monads in $\sC$ constitute the 0-cells of a 2-category $\mnd(\sC)$, defined as follows. The 1-cells $(\cA,B) \to (\cD,A )$ consist of pairs $(\Sigma, \sigma)$ where $\Sigma \colon \cA \to \cD$ is a 1-cell in $\sC$ and $\sigma \colon A\Sigma \Rightarrow \Sigma B$ is a 2-cell in $\sC$, subject to the commutativity conditions
$$
\xymatrix{
AA\Sigma \ar[d]_-{\mu \Sigma} \ar[r]^-{A \sigma} & A \Sigma B \ar[r]^-{\sigma B} & \Sigma BB \ar[d]^-{\Sigma \mu} \\
A \Sigma \ar[rr]_-{\sigma} & & \Sigma B
}\qquad
\xymatrix{
\Sigma \ar[dr]_-{\Sigma \eta} \ar[r]^-{\eta \Sigma} & A \Sigma \ar[d]^-\sigma \\
& \Sigma B
}
$$
We call these 1-cells \emph{morphisms of monads}.
A 2-cell $\alpha \colon (\Sigma, \sigma) \Rightarrow (\Sigma', \sigma')$ is a 2-cell $\alpha \colon \Sigma \Rightarrow \Sigma'$ in $\sC$ such that the diagram
$$
\xymatrix{
A \Sigma \ar[d]_-{A \alpha} \ar[r]^-\sigma & \Sigma B \ar[d]^-{\alpha B} \\
A \Sigma ' \ar[r]_-{\sigma '} & \Sigma' B
}
$$
commutes.

Dually, we define a \emph{comonad} $(T, \delta, \epsilon)$ in $\sC$ to be a monad in $\sC_*$. We define the 2-category of comonads in $\sC$ as $\cmd(\sC):= \mnd(\sC_*)_*$. The 1-cells herein are called \emph{morphisms of comonads}. We say \emph{opmorphism of (co)monads} to mean a morphism of (co)monads in $\sC^*$.

\begin{exa}\label{monadj}
Suppose that we have an adjunction
$$
\xymatrix{
\cA \ar@{}[rr]|-{\perp}\ar@/^0.5pc/[rr]^-F & & \ar@/^0.5pc/[ll]^-U \cB
}
$$
in $\sC$ with unit $\eta$ and counit $\epsilon$.
The 1-cell $UF$ becomes a monad on $\cA$, and dually $FU$ becomes a comonad on $\cB$, and we say that these are \emph{generated by the adjunction}. The
(co)units and (co)multiplications are given by
\begin{align*}
  \xymatrix{
   1 \ar[r]^-{\eta} & UF} & & & \xymatrix{FU \ar[r]^-\epsilon & 1} \\
 \xymatrix{UFUF \ar[r]^-{U \epsilon F} & UF}& & &\xymatrix{FU \ar[r]^-{F \eta U} & FUFU}
\end{align*}
\end{exa}

\begin{rem}
Since the taking of duals is confusing, let us make explicit that if
$$\xymatrix{
 (\cA, B) \ar[rr]^-{(\Sigma, \sigma)} && (\cD, A),
}\qquad
\xymatrix{
(\cB, T) \ar[rr]^-{(\Gamma, \gamma)} && (\cE, G)
}
$$
  are morphisms of monads and comonads respectively, then the underlying 2-cells are given by $\sigma \colon A\Sigma \Rightarrow \Sigma B$ and $\gamma \colon \Gamma T \Rightarrow G \Gamma$. Note the differing positions of $\Gamma, \Sigma$ in each case, telling us that these morphisms are not of the same \emph{variance}. If we take opmorphisms instead, these 2-cells would reverse direction.
\end{rem}
\begin{exa}\label{trivialdist}
Consider a monad $(\cA, B)$ in $\sC$. The 2-cell
$$
\xymatrix{
\tau \colon BB \ar[rr]^-{BB\eta} && BBB \ar[r]^-{\mu B} & BB
}
$$
defines a morphism of monads
$$
\xymatrix{
(\cA, B) \ar[rr]^-{(B, \tau)} && (\cA, B).
}
$$
Furthermore, any morphism of monads $(\Sigma, \sigma) \colon (\cA, B) \to (\cD, A)$ induces a monad 2-cell
$$
\xymatrix{
(\cA, B) \ar[rr]^-{(\Sigma, \sigma)} \ar[dd]_-{(B, \tau)} && (\cD,A) \ar[dd]^-{(A, \tau)} \ar@{}[ddll]^(.25){}="a"^(.75){}="b" \ar@{=>}_-{\sigma} "a";"b" \\
\\
(\cA, B) \ar[rr]_-{(\Sigma, \sigma)} && (\cD, A)
}$$
\end{exa}
As in the case of adjunctions, 2-functors preserve the internal property of being a monad. Thus, a 2-functor
$\Phi \colon \sC \to \sD$ restricts to a 2-functor
$$\xymatrix{  \mnd(\sC) \ar[rr]^-{\mnd(\Phi)} && \mnd(\sD),}$$ defined diagramatically by
$$
\xymatrix{
(\cA, B ) \ddtwocell<9>_{(\Sigma, \sigma)\ \ \ \ \ }^{\ \ \ \ \ \ \  (\Sigma', \sigma')}{^\alpha}  & & & &  (\Phi\cA, \Phi B) \ddtwocell<9>_{(\Phi\Sigma, \Phi\sigma)\ \ \ \ \ \  \ \ \ }^{\ \ \ \ \ \ \ \ \ \ (\Phi\Sigma', \Phi\sigma')}{^\Phi\alpha} \\
 &  & \longmapsto \\
(\cD,  A) & & & & (\Phi\cD, \Phi A)
}
$$
A 2-natural transformation $\nu \colon \Phi \Rightarrow \Phi'$ clearly induces another 2-natural transformation $\mnd(\nu)$. We thus obtain the following:
\begin{prop}
The above assignment defines a 2-functor $\mnd \colon \twocat \to \twocat$.
\end{prop}
Dually, there is a 2-functor $\cmd \colon \twocat \to \twocat$.

\subsection{Distributive laws}\label{finallyoversection}
\begin{defn}\label{finallyover}
A \emph{distributive law of comonads} is a comonad in $\cmd(\sC)$. Explicitly, a distributive law between comonads $T,S$ on the same 1-cell $\cB$ is a 2-cell $\chi \colon TS \Rightarrow ST$  such that the four diagrams
\begin{align*}
\xymatrix{
TS \ar[d]_-{\delta S} \ar[rr]^-\chi & & ST \ar[d]^-{S \delta} \\
TTS \ar[r]_-{T \chi} & TST \ar[r]_-{\chi T} & STT
}\qquad
\xymatrix{
TS \ar[dr]_-{\epsilon S} \ar[r]^-\chi & ST \ar[d]^-{S \epsilon} \\
& S
} \\
\xymatrix{
TS \ar[d]_-{T\delta} \ar[rr]^-\chi & & ST \ar[d]^-{\delta T} \\
TSS \ar[r]_-{\chi S} & STS \ar[r]_-{ S\chi } & SST
}
\qquad
\xymatrix{
TS \ar[dr]_-{T \epsilon} \ar[r]^-\chi & ST \ar[d]^-{\epsilon T} \\
& T
}
\end{align*}
commute.
\end{defn}
We denote by $\cDist(\sC)$ the 2-category $\cmd(\cmd(\sC)^*)^*$. Thus, explicitly,
\begin{itemize}
\item $0$-cells are quadruples $(\cB, \chi, T, S)$,
where $\chi \colon TS \Rightarrow ST$ is a
comonad distributive law on $\cB$,
\item $1$-cells $$(\cB, \chi, T, S) \rightarrow (\cD,
\tau, G, C)$$ are triples $(\Sigma,
\sigma, \gamma)$, where $(\Sigma, \sigma) \colon (\cB, T) \rightarrow (\cD, G)$ is
an opmorphism of comonads, and $(\Sigma, \gamma)
\colon (\cB , S) \rightarrow (\cD, C)$ is a
morphism of comonads satisfying the Yang-Baxter
equation, i.e.\
$$
	\xymatrix@R=0.5em{ & \Sigma TS \ar[r]^-{\Sigma
\chi} & \Sigma ST \ar[dr]^-{\gamma  T} & \\
G\Sigma S \ar[dr]_-{G \gamma} \ar[ur]^-{\sigma
S} & & & C \Sigma T \\ & GC\Sigma
\ar[r]_-{\tau \Sigma} & CG \Sigma \ar[ur]_-{C
\sigma} & }
$$
commutes, and
\item $2$-cells $(\Sigma,
\sigma, \gamma) \Rightarrow (\Sigma', \sigma', \gamma')$
are 2-cells $\alpha \colon \!\Sigma \Rightarrow
\Sigma'$ in $\sC$ for which the diagrams
$$
	\xymatrix{
	G\Sigma \ar[d]_-\sigma \ar[r]^-{G \alpha}
	& G \Sigma '
	\ar[d]^-{\sigma '} \\ \Sigma T \ar[r]_-{\alpha T } &
	\Sigma' T }
	\quad \quad \quad
	\xymatrix{ \Sigma S
	\ar[r]^-{\alpha S} \ar[d]_-\gamma & \Sigma' S
	\ar[d]^-{\gamma '} \\ C \Sigma \ar[r]_-{C \alpha}& C
	\Sigma' }
$$
commute.
\end{itemize}

Similarly, we define the $2$-category of \emph{mixed distributive laws} in $\sC$ as $$\cMix(\sC):=\mnd(\cmd(\sC)).$$
Let us unpack this definition.
Consider an arbitrary 1-cell in the 2-category $\cMix(\sC)$:
$$
\xymatrix{
((\cA, C), (B, \theta) ) \ar[rr]^-{((\Sigma, \gamma), \sigma)} && ((\cD, D), (A, \tau) ).
}
$$
We have that:
\begin{itemize}
\item $(\cA, C)$ and $(\cD, D)$ are comonads in $\sC$;
\item $(B, \theta)$ and $(A, \tau)$ are monads in $\cmd(\sC)$, meaning that $$\theta \colon BC \Rightarrow CB, \qquad \tau \colon AD \Rightarrow DA$$ are 2-cells in $\sC$, compatible with the appropriate monad and comonad structures;
\item $(\Sigma, \gamma)$ is a morphism of comonads, so in particular $\Sigma \colon \cA \to \cD$ is a 1-cell in $\sC$, and $\gamma \colon \Sigma C \Rightarrow D \Sigma$ is a 2-cell in $\sC$ compatible with the comonad structures of $C$ and $D$;
\item $\sigma \colon (A, \tau) \circ (\Sigma, \gamma) \Rightarrow (\Sigma, \gamma) \circ (B, \theta)$ is a 2-cell in $\cmd(\sC)$ which corresponds to a Yang-Baxter-esque commutative hexagon;
\item $\sigma \colon A\Sigma \Rightarrow \Sigma B$ is a 2-cell compatible with the monad structures of $A$ and $B$.
\end{itemize}
The above data exactly defines a $1$-cell
$$
\xymatrix{
((\cA, B), (C, \theta) ) \ar[rr]^-{((\Sigma, \sigma), \gamma)} & & ((\cD, A), (D, \tau))
}
$$
in the 2-category $\cmd(\mnd(\sC))$. Comparing the 2-cells in a similar way shows that:
\begin{lem}\label{mndcmdiso}
There is a 2-isomorphism $$\mnd(\cmd(\sC)) \cong \cmd(\mnd(\sC)).$$
\end{lem}
We represent the 1-cells of $\cMix(\sC)$ hereafter as
$$
\xymatrix{
(\cA, \theta, B, C ) \ar[rr]^-{(\Sigma, \sigma, \gamma)} & & (\cD, \tau, A, D ),
}
$$
so our notation aligns with that of the category $\cDist(\sC)$.
\subsection{Eilenberg-Moore constructions}\label{EM2cat}

Let $\cA$ be a 0-cell in $\sC$. The identity 1-cell $1 \colon \cA \to \cA$ is part of a monad with multiplication and unit both given by $1$. This defines an inclusion 2-functor $I \colon \sC \to \mnd(\sC)$.

\begin{defn}
We say that $\sC$ admits \emph{Eilenberg-Moore constructions for monads} if the inclusion 2-functor $I$ has a right 2-adjoint $J \colon \mnd(\sC) \to \sC$.
\end{defn}

If this 2-functor exists, its action on a morphism of monads $(\Sigma, \sigma) \colon (\cA, B) \to (\cD, A)$ is denoted by $\Sigma^\sigma \colon \cA^B \to \cD^A$, and its action on a 2-cell $\alpha \colon (\Sigma, \sigma) \Rightarrow (\Sigma', \sigma')$ is denoted by $\tilde \alpha \colon \Sigma^\sigma \Rightarrow  \Sigma'^{\sigma'}$. The 2-adjunction part of the definition explicitly means that there are 2-natural isomorphisms
$$
\mnd(\sC) ( (\cA, 1), (\cD, A) ) \cong \sC\left(\cA, \cD^A\right)
$$
of hom-categories.

Let $\nu \colon 1 \Rightarrow JI$ and $\xi \colon IJ \Rightarrow 1$ denote the unit and counit of $I\dashv J$, respectively. For each monad $(\cA, B)$ the counit has a component $(\cA^B, 1) \to (\cA, B)$ . This is a morphism of monads, and we denote it by $(U^B, \kappa)$, where $\kappa \colon BU^B \Rightarrow U^B$ is a 2-cell in $\sC$ such that the diagrams
$$
\xymatrix{
U^B \ar[r]^-{\eta U^B}\ar@{=}[dr] & BU^B \ar[d]^-\kappa \\
& U^B
}\qquad
\xymatrix{
BBU^B \ar[r]^-{B\kappa} \ar[d]_-{\mu U^B} & BU^B \ar[d]^-\kappa \\
BU^B \ar[r]_-\kappa & B
}
$$
commute. The monad axioms tell us that $(B, \mu) \colon (\cA, 1) \to (\cA, B)$ is a morphism of monads, and this corresponds under the 2-adjunction $I\dashv J$ to a 1-cell $F^B\colon \cA \to \cA^B$ in $\sC$, unique such that
\begin{equation}\label{EMmonad}
\begin{array}{c}
\xymatrix{
(\cA, 1) \ar[r]^-{(F^B, 1)} \ar[dr]_-{(B, \mu)} & (\cA^B, 1) \ar[d]^-{(U^B,\kappa)} \\
& (\cA, B)
}
\end{array}
\end{equation}
commutes in $\mnd(\sC)$. In particular this tells us that $U^BF^B = B$ and $\kappa F^B = \mu$. We have that
$$
\xymatrix{
(\cA^B, 1) \rrtwocell^{<1.5>(U^BFU^B,~\mu U^B)}_{<1.5>(U^B,~\kappa)}{\kappa}&& (\cA, B)
}
$$
is a 2-cell in $\mnd(\sC)$ corresponding under the 2-adjunction to a 2-cell
$$
\xymatrix{
\cA^B \rrtwocell^{<1.5>F^BU^B}_{1}{\epsilon} && \cA^B
}
$$
in $\sC$, unique such that $U^B \epsilon = \kappa$. It turns out that $\eta$ is the unit and $\epsilon$ is the counit of an adjunction $F^B \dashv U^B$; see~\cite[p.~152]{MR0299653} for the full details of this construction.

As explained in Example~\ref{monadj}, the adjunction $F^B \dashv U^B$ generates the monad $B$ as well a comonad $F^BU^B \colon \cA^B \to \cA^B$ with counit $\epsilon$ and comultiplication $F^B \eta U^B$. We write $\tilde B$ to denote this comonad.

\begin{prop}
Let $B = 1$ be the identity monad on $\cA$. Then $F^B \colon \cA \to \cA^B$ is the unit of $I \dashv J$ evaluated at $\cA$ and is an isomorphism. Furthermore, $\tilde B = 1$.
\end{prop}
\begin{proof}
By diagram~\ref{EMmonad}, we have that $U^BF^B = 1$ and $\kappa = 1$. It follows that $\epsilon = 1$ since it is the unique 2-cell such that $U^B \epsilon = 1$. Hence $\tilde B = F^BU^B = 1$ and $F^B,U^B$ are isomorphisms. By the triangle identities for $I\dashv J$, Diagram~\ref{EMmonad} commutes for the choice $F^B = \nu$, so this must be the only choice by uniqueness.
\end{proof}
We henceforth identify $JI$ with 1.
 If $(\Sigma, \sigma) \colon (\cA, 1)\to (\cD, A)$ is an arbitrary morphism of monads, then it corresponds to
$$
\xymatrix{
\cA \ar[r]^-\nu_-\cong & \cA^1 \ar[r]^-{\Sigma^\sigma} & J(\cD, A)
}
$$
under the 2-adjunction, which we write as $\Sigma^\sigma \colon \cA \to \cD^A$. In a similar fashion, a morphism of monads $(\Sigma, \sigma) \colon (\cA, B) \to (\cD, 1)$ corresponds under the 2-adjunction to $\Sigma^\sigma \colon \cA^B \to \cD$.
\begin{prop}\label{cmdEM}
If $\sC$ admits Eilenberg-Moore constructions for monads, then so too does $\cmd(\sC)$.
\end{prop}
\begin{proof}
Since 2-functors preserve adjunctions, $\cmd(J) \colon \cmd(\mnd(\sC)) \to \cmd(\sC)$ is a right 2-adjoint to $\cmd(I)$. After composing with the isomorphism $$\mnd(\cmd(\sC)) \cong \cmd(\mnd(\sC))$$ of Lemma~\ref{mndcmdiso}, the 2-functor $\cmd(I)$ becomes the inclusion 2-functor, and thus we have constructed its right 2-adjoint, as required.
\end{proof}
Since $\cMix(\sC) = \mnd(\cmd(\sC))$, we have given a 2-functor $\cMix(\sC) \to \cmd(\sC)$. Diagramatically it is represented as
$$
\xymatrix{
(\cA, \theta, B, C) \ddtwocell<10>_{(\Sigma, \sigma, \gamma)\ \ \ \ \ \ \ }^{\ \ \ \ \ \ \ \ \ \ (\Sigma', \sigma', \gamma')}{^\alpha}  & & & &  (\cA^B, C^\theta) \ddtwocell<10>_{(\Sigma^\sigma, \tilde\gamma)\ \ \ \ \ \ \ }^{\ \ \ \ \ \ \ \ \ \ (\Sigma'^{\sigma '}, \tilde\gamma')}{^\tilde\alpha} \\
 &  & \longmapsto \\
(\cD, \psi, A, D) & & & & (\cD^A, D^\psi)
}
$$
where the notation $\tilde\gamma$ is explained as follows: by the Yang-Baxter equation, $\gamma$ is a 2-cell of monads
$$
\xymatrix{
(\cA, B) \ar[rr]^-{(C, \theta)} \ar[d]_-{(\Sigma, \sigma)} & & \ar@{}[dll]^(.25){}="a"^(.75){}="b" \ar@{=>}_-\gamma "a";"b" (\cA, B) \ar[d]^-{(\Sigma, \sigma)} \\
(\cD, A) \ar[rr]_-{(D, \psi)} & & (\cD, A)
}
$$
and applying $J \colon \mnd(\sC) \to \sC$ to this square gives a 2-cell
$$
\xymatrix{
\cA^B \ar[rr]^-{C^\theta} \ar[d]_-{\Sigma^\sigma} && \ar@{}[dll]^(.25){}="a"^(.75){}="b" \ar@{=>}_-{\tilde\gamma} "a";"b" \cA^B \ar[d]^-{\Sigma^\sigma} \\
\cD^A \ar[rr]_-{D^\psi} & & \cD^A
}
$$
Since $J \colon \mnd(\sC) \to \sC$ is a 2-functor, it sends comonads to comonads. In this situation,  a mixed distributive law $\theta \colon BC \Rightarrow CB$ is sent to a comonad $C^\theta$.

Using Proposition~\ref{cmdEM} we may freely replace $\sC$ with $\cmd(\sC)$ in any statement made about a 2-category $\sC$ which admits Eilenberg-Moore constructions for monads.
\section{Lifting through adjunctions}\label{liftingthroughadjunctions}

Here we discuss distributive
laws that are compatible in a specific way with adjunctions that generate one of the
involved comonads. We go on to explain, given a 2-category $\sC$ which admits Eilenberg-Moore constructions for monads, how the comparison 1-cell of~\cite{MR0299653} lifts to become a 1-cell between comonad distributive laws and that there is a canonical 2-functor $ \cMix(\sC) \to \cDist(\sC)$.

\subsection{The lifting theorem}
Let $\sC$ be a 2-category. Consider squares in $\sC$ of the form
$$
\xymatrix{\cB \ar[r]^ U  \ar[d]_ S & \ar@{}[dl]^(.25){}="a"^(.75){}="b" \ar@{=>}_-{\Omega} "a";"b"   \cA \ar[d]^ C\\
	\cD \ar[r]_V  & \cC}
$$
We obtain two 2-categories $\sq^h(\sC)$ and $\sq^v(\sC)$ by defining such squares to be 1-cells and defining their composition by pasting horizontally and vertically respectively (really these are the horizontal and vertical components of a double category $\sq(\sC)$, see~\cite[Observation~76]{MR2399898}). The 2-cells are pairs of 2-cells in $\sC$ that satisfy the obvious compatibility condition.

\begin{prop}\label{monadsq}
To give a monad in $\sq^v(\sC)$ is the same as to give a morphism of monads in $\sC$.
\end{prop}
\begin{proof}
A square
$$
\xymatrix{\cA \ar[r]^ \Sigma  \ar[d]_ B & \ar@{}[dl]^(.25){}="a"^(.75){}="b" \ar@{=>}_-{\sigma} "a";"b"   \cD \ar[d]^ A\\
	\cA \ar[r]_\Sigma  & \cD}
$$
is a monad in $\sq^v(\sC)$ if and only if $B,A $ are monads and $$\xymatrix{(\cA, B) \ar[rr]^-{(\Sigma, \sigma)} && (\cD, A)}$$ is a morphism of monads.
\end{proof}
Dually, to give a comonad in $\sq^v(\sC)$ is the same as to give an opmorphism of comonads in $\sC$.

Suppose that we have a square
$$
\xymatrix{\cB \ar[r]^ U  \ar[d]_ S & \ar@{}[dl]^(.25){}="a"^(.75){}="b" \ar@{=>}_-{\Omega} "a";"b"   \cA \ar[d]^ C\\
	\cD \ar[r]_V  & \cC}
$$
in $\sC$ where $\Omega$ is an isomorphism, and $U,V$ are both right adjoints
$$
\xymatrix{
\cA \ar@{}[rr]|-{\perp}\ar@/^0.5pc/[rr]^-F & & \ar@/^0.5pc/[ll]^-U \cB
},\qquad
\xymatrix{
\cC \ar@{}[rr]|-{\perp}\ar@/^0.5pc/[rr]^-G & & \ar@/^0.5pc/[ll]^-V \cD
}
$$
where $\eta, \epsilon$ denote the units and counits respectively of both adjunctions.

\begin{defn}\label{lift}
In the above situation, we say that $ C$ is an \emph{extension of} $ S$ and $ S$ is a \emph{lift
of $ C$ through the adjunctions} $F \dashv U$ and $G \dashv V$.
\end{defn}

\begin{lem}\label{sq}
The square
$$\xymatrix{
\cB \ar[rr]^-S \ar[dd]_-U && \cD \ar[dd]^-V \ar@{}[ddll]^(.25){}="a"^(.75){}="b" \ar@{=>}_-{\Omega^{-1}} "a";"b" \\
\\
\cA \ar[rr]_-C && \cC
}$$
is a right adjoint in $\sq^v(\sC)$.
\end{lem}
\begin{proof}
The left adjoint is constructed as
$$\xymatrix{
\cA \ar[r]^-C \ar[d]_-F & \cC \ar[d]^-G  \ar@{}[dl]^(.25){}="a"^(.75){}="b" \ar@{=>}_-{\Lambda} "a";"b"\\
\cB \ar[r]_-S & \cD
}$$
where $\Lambda$ is the uniquely determined \emph{mate}~\cite{MR0357542} of $\Omega$, that is, the composite
$$
\xymatrix{
GC \ar[r]^-{GC \eta} & GCUF \ar[rr]^-{G \Omega F} && GVSF \ar[r]^-{\epsilon SF} & SF
}
$$
The (co)unit for this adunction is given by taking the pair of (co)units for the adjunctions $F\dashv U$ and $G \dashv V$.
\end{proof}

The following theorem,
which closely follows \cite[Lemmata~6.1.1
and~6.1.4]{MR2094071}, constructs a canonical pair of
distributive laws from $\Lambda$.

\begin{thm}\label{arisem}
The 2-cells
$$
	\xymatrix{\theta \colon
	VGC \ar[rr]^-{V \Lambda}
	&& VSF \ar[rr]^-{\Omega^{-1} F} &&
	CUF}
$$
and
$$
	\xymatrix{\chi \colon
	GVS\ar[rr]^-{G \Omega^{-1}} & &GCU
	\ar[rr]^-{\Lambda U} && SFU}
$$
define a morphism of monads $(C, \theta) \colon (\cA, UF) \to (\cC, VG)$ and an opmorphism of comonads
$(S, \chi) \colon (\cB, FU) \to (\cD, GV)$ respectively.
\end{thm}
\begin{proof}
The adjunctions in Lemma~\ref{sq} generate a monad and a comonad
$$
\xymatrix{
\cA \ar[r]^-C \ar[d]_-{UF} & \cC \ar[d]^-{VG}  \ar@{}[dl]^(.25){}="a"^(.75){}="b" \ar@{=>}_-{\theta} "a";"b"\\
\cA \ar[r]_-C & \cC
}\qquad
\xymatrix{
\cB \ar[r]^-S \ar[d]_-{FU} & \cD \ar[d]^-{GV}  \ar@{}[dl]^(.25){}="a"^(.75){}="b" \ar@{=>}_-{\chi} "a";"b"\\
\cB \ar[r]_-S & \cD
}
$$
respectively in $\sq^v(\sC)$. By Proposition~\ref{monadsq} this is equivalent to the theorem. \end{proof}
In fact, $\theta, \chi$ satisfy a universal property:
\begin{prop}\label{univprop}
In the setting of Theorem~\ref{arisem},
\begin{enumerate}
\item The 2-cell $\theta \colon VGC \Rightarrow CUF$ is unique such that the diagram
\begin{equation}\label{univmonad}
\begin{array}{c}
\xymatrix{
VGCU \ar[r]^-{\theta U } \ar[d]_-{VG\Omega} & CUFU \ar[r]^-{CU \epsilon} & CU \ar[d]^-{\Omega} \\
VGVS \ar[rr]_-{V\epsilon S} & & VS
}
\end{array}
\end{equation}
commutes.
\item The 2-cell $\chi \colon GVS \Rightarrow SFU$ is unique such that the diagram
\begin{equation}\label{univcomonad}
\xymatrix{
CU \ar[rr]^-{C\eta U} \ar[d]_-{\Omega} & & CUFU\ar[d]^-{\Omega FU} \\
VS \ar[r]_-{V \eta S} & VGVS \ar[r]_-{V\chi} & VSFU
}
\end{equation}
commutes.
\end{enumerate}
\end{prop}
\begin{proof} We prove part 1, and remark that the proof for part 2 is similar. Consider the diagram
$$
\xymatrix@=3.5em{
VGCU \ar@{=}[d] \ar[r]^-{VGC\eta U} & VGCUFU \ar[d]_-{VGCU\epsilon} \ar[r]^-{VG\Omega FU} & VGUSFU \ar[d]_-{VGUS\epsilon}\ar[r]^-{V\epsilon S FU} & VSFU \ar[d]_-{VS\epsilon}\ar[r]^-{\Omega^{-1} FU} & CUFU \ar[d]^-{CU\epsilon} \\
 VGCU\ar@{=}[r] & VGCU \ar[r]_-{VG\Omega} & VGVS \ar[r]_-{V\epsilon S} & VS \ar[r]_-{\Omega^{-1}} & CU
}
$$
The first square commutes by a triangle identities for the adjunction $G \dashv V$ and the other squares commute by compatibility of horizontal and vertical composition in $\sC$. Thus, the outer diagram commutes, which means that diagram~\ref{univmonad} commutes also.

Let $\theta \colon VGC \Rightarrow CUF$ be a 2-cell satisfying the hypothesis. Consider the diagram
$$
\xymatrix@=3.5em{
VGC \ar[r]^-{\theta'} \ar[d]_-{VGC\eta} & CUF \ar@{=}[r] \ar[d]^-{CUF\eta} & CUF \ar@{=}[d] \\
VGCUF \ar[r]^-{\theta' UF} \ar[d]_-{VG\Omega F}& CUFUF \ar[r]^-{CU\epsilon F} & CUF \ar[d]^-{\Omega F} \\
VGVSF \ar[rr]_-{V\epsilon SF} & & VSF
}
$$
The leftmost square commutes by compatibility of composition in $\sC$, the rightmost square commutes by a triangle identity for $F \dashv U$ and the lower rectangle commutes by assumption. Thus, the outer rectangle commutes, which means exactly that $\theta' = \theta$.
\end{proof}

\begin{rem}\label{univrem}
The commutativity of diagrams~\ref{univmonad} and~\ref{univcomonad} is equivalent to the statement that
$$
\xymatrix{
(\cB, 1) \ar[rr]^-{(U, U\epsilon)} \ar[dd]_-{(S,1)} && (\cA, UF)\ar[dd]^-{(C, \theta)}  \ar@{}[ddll]^(.25){}="a"^(.75){}="b" \ar@{=>}_-{\Omega} "a";"b"\\
\\
(\cD, 1) \ar[rr]_-{(V, V\epsilon)} && (\cC, VG)
}
\qquad
\xymatrix{
(\cB, FU) \ar[rr]^-{(U, \eta U)} \ar[dd]_-{(S,\chi)} && (\cA, 1)\ar[dd]^-{(C, 1)}  \ar@{}[ddll]^(.25){}="a"^(.75){}="b" \ar@{=>}_-{\Omega} "a";"b"\\
\\
(\cD, GV) \ar[rr]_-{(V, \eta V)} && (\cC, 1)
}
$$
are 2-cells in $\mnd(\sC)$ and $\cmd(\sC^*)^*$ respectively.
\end{rem}

We now state an important corollary of Theorem~\ref{arisem}. Since we use this extensively later, we explicitly state the necessary terminology.
\begin{cor}\label{arisec}
Suppose that we have a square
$$
\xymatrix{\cB \ar[r]^ U  \ar[d]_ S & \ar@{}[dl]^(.25){}="a"^(.75){}="b" \ar@{=>}_-{\Omega} "a";"b"   \cA \ar[d]^ C\\
	\cB \ar[r]_U  & \cA}
$$
in a 2-category $\sC$ where $\Omega$ is an isomorphism, and $U$  is a right adjoint
$$
\xymatrix{
\cA \ar@{}[rr]|-{\perp}\ar@/^0.5pc/[rr]^-F & & \ar@/^0.5pc/[ll]^-U \cB.
}
$$
Let $\eta$ be the unit, and let $\epsilon$ be the counit of this adjunction, and let $\Lambda$ denote the 2-cell
$$
\xymatrix{
FC \ar[r]^-{FC \eta} & FCUF \ar[rr]^-{F \Omega F} && FUSF \ar[r]^-{\epsilon SF} & SF.
}
$$
Then:
\begin{enumerate}
\item The 2-cells
$$
	\xymatrix{\theta \colon
	UFC \ar[rr]^-{U \Lambda}
	&& USF \ar[rr]^-{\Omega^{-1} F} &&
	CUF}
$$
and
$$
	\xymatrix{\chi \colon
	FUS\ar[rr]^-{F \Omega^{-1}} & &FCU
	\ar[rr]^-{\Lambda U} && SFU}
$$
define a morphism of monads $(C, \theta) \colon (\cA, UF) \to (\cA, UF)$ and an opmorphism of comonads
$(S, \chi) \colon (\cB, FU) \to (\cB, FU)$ respectively.
\item
If $C$ and $S$ are themselves comonads, and
$(U, \Omega) \colon (\cB, S) \to (\cA, C)$ is an opmorphism of comonads, then $\theta$ is a mixed distributive law and $\chi$ is a comonad distributive law.
\end{enumerate}

\end{cor}
\begin{proof}
Part 1 follows immediately from Theorem~\ref{arisem}, in the special case where the two adjunctions are equal.
If $C,S$ are comonads and $(U, \Omega)$ is an opmorphism of comonads, then there is an adjunction $(F,\Lambda) \dashv (U, \Omega^{-1})$ in $\cmd(\sC)$. This adjunction generates a monad and comonad, yielding the distributive laws $\theta,\chi$.
\end{proof}

\begin{defn}\label{arisedeffo}
A comonad distributive law $ \chi $, or a mixed distributive law $\theta$, as in
Corollary~\ref{arisec} is said to \emph{arise from the
adjunction $F \dashv U$}.
\end{defn}

\subsection{The extremal case}\label{extremalcase}
In this section, we examine what can be said about lifts when we start with a morphism of monads, instead of the other way around.
Suppose that $\sC$ admits Eilenberg-Moore constructions for monads, with 2-adjunction $I\dashv J$.
One extremal situation in which specifying a morphism of monads uniquely determines a lift of a 1-cell is the following: consider an arbitrary morphism of monads
$$
\xymatrix{
(\cA, B) \ar[rr]^-{(C, \theta)} && (\cC, E).
}
$$
By the remarks in Section~\ref{EM2cat}, there are adjunctions
$$
\xymatrix{
\cA \ar@{}[rr]|-{\perp}\ar@/^0.5pc/[rr]^-{F^B} & & \ar@/^0.5pc/[ll]^-{U^B} \cA^B
},\qquad
\xymatrix{
\cC \ar@{}[rr]|-{\perp}\ar@/^0.5pc/[rr]^-{F^E} & & \ar@/^0.5pc/[ll]^-{U^E} \cC^E
}
$$
with unit, counit denoted by $\eta, \epsilon$ respectively where
$$
\xymatrix{
(\cA^B, 1) \ar[rr]^-{(U^B, U^B \epsilon)} && (\cA, B)
}, \qquad
\xymatrix{
(\cC^E, 1) \ar[rr]^-{(U^E, U^E \epsilon)} && (\cC, E)
}
$$
are morphisms of monads given by evaluating the counit $\xi \colon  IJ  \Rightarrow 1$  at $(\cA, B)$ and $ (\cC, E)$ respectively.

\begin{thm}\label{unique1cell}
The 1-cell $C^\theta \colon \cA^B \to \cC^E$ is unique such  that the diagram
\begin{equation}\label{liftdiag}
\begin{array}{c}
\xymatrix{
(\cA^B, 1) \ar[d]_-{(C^\theta, 1)} \ar[rr]^-{(U^B, U^B\epsilon)} && (\cA, B) \ar[d]^-{(C, \theta)} \\
(\cC^E, 1) \ar[rr]_-{(U^E, U^E\epsilon)} && (\cC, E)
}
\end{array}
\end{equation}
commutes in $\mnd(\sC)$.
\end{thm}
\begin{proof}
By 2-naturality of $\xi$, the diagram
$$
\xymatrix{IJ(\cA, B) \ar[d]_-{IJ(\cC, \theta)} \ar[r]^-\xi & (\cA, B) \ar[d]^-{(C, \theta)} \\
IJ(\cC, E) \ar[r]_-{\xi} & (\cC, E)
}
$$
commutes, which is identical to diagram~\ref{liftdiag}. If $H \colon \cA^B \to \cC^E$ is another such 1-cell then the diagram
$$
\xymatrix{IJ(\cA, B) \ar[d]_-{IH} \ar[r]^-\xi & (\cA, B) \ar[d]^-{(C, \theta)} \\
IJ(\cC, E) \ar[r]_-{\xi} & (\cC, E)
}
$$
commutes. The morphism obtained from going along the top of the diagram corresponds to $H$ under the 2-adjunction $I\dashv J$, and the morphism along the bottom corresponds to $J(C, \theta) = C^\theta$. Hence $H = C^\theta$.
\end{proof}
Note that diagram~\ref{liftdiag} says exactly that the identity is a monad 2-cell. Therefore, by Proposition~\ref{univprop} and Remark~\ref{univrem} (taking $E = VG$ and $B = UF$), we recover $(C,\theta)$ as the canonical morphism given by Theorem~\ref{arisem}. We write $\tilde\theta$ to denote the 2-cell
$$
\chi \colon \tilde E C^\theta \Rightarrow C^\theta \tilde B
$$
from Theorem~\ref{arisem}, 

\begin{exa}\label{trivv2}
We have a commutative diagram
$$\xymatrix{
\cA^B \ar[r]^-{U^B} \ar[d]_-{\tilde B} & \cA \ar[d]^-B \\
\cA^B \ar[r]_-{U^B} & \cA
}
$$
and thus, by Theorem~\ref{arisem}, a morphism of monads $(B, \theta) \colon (\cA, B) \to (\cA, B)$  defined by
$$
\xymatrix{
\theta \colon BB = UFUF \ar[rr]^-{UFUF\eta} && UFUFUF \ar[rr]^-{U\epsilon F U F} && UFUF = BB
}
$$
However, $U \epsilon F = \mu$ by diagram~\ref{EMmonad}, so in fact $\theta = \tau$ from Example~\ref{trivialdist}. Since both $\tilde B$, $B^\tau$ fit into diagram~\ref{trivialdist}, Theorem~\ref{unique1cell} tells us that $\tilde B = B^\tau$ as 1-cells.
\end{exa}

Now let us specialise to the situation that $\cC = \cA$, $E = B$, and $C$ is a comonad on $\cA$. Suppose also that $\theta$ is a distributive law of comonads.

\begin{prop}\label{everydistlawarises}
The identity 2-cell
$$
\xymatrix{\cA^B \ar[r]^-{U^B}  \ar[d]_{C^\theta} & \ar@{}[dl]^(.25){}="a"^(.75){}="b" \ar@{=>}_-{1} "a";"b"   \cA \ar[d]^ C\\
	\cA^B \ar[r]_{U^B}  & \cA
}
$$
defines an opmorphism of comonads
$$\xymatrix{ (\cA^B, C^\theta) \ar[rr]^-{(U^B, 1)} && (\cA, C).}$$
\end{prop}
\begin{proof}
There is a commutative diagram
$$
\xymatrix{
(\cA^B, 1) \ar[rr]^-{(U^B, U^B\epsilon)} \ar[drr]_-{(U^B, 1)} && (\cA, B) \ar[d]^-{(1, \eta)} \\
&& (\cA, 1)
}
$$
in $\mnd(\sC)$. The top 1-cell is the counit of $I \dashv J$, so applying $J$ maps it to the identity. Therefore $1^\eta = U^B \colon \cA^B \to \cA$. The result now follows by applying the 2-functor $\cMix(\sC) \to \cmd(\sC)$ of Proposition~\ref{cmdEM} to the 1-cell
$$
\xymatrix{
(\cA, \theta, B, C) \ar[rr]^-{(1, \eta, 1)} & & (\cA, 1, 1, C)
}
$$
in $\cMix(\sC)$.
\end{proof}
Now we are in the situation of Corollary~\ref{arisec} (taking $B = U^BF^B = UF$), and from that we recover the fact that $\theta, \chi$ are distributive laws.
\subsection{The comparison 1-cell}\label{comparisonsection}
Let $\sC$ be a 2-category which admits Eilenberg-Moore constructions for monads, with 2-adjunction $I\dashv J$,  and let
$$
\xymatrix{
\cA \ar@{}[rr]|-{\perp}\ar@/^0.5pc/[rr]^-F & & \ar@/^0.5pc/[ll]^-U \cB
}
$$
be an adjunction therein, which generates a monad $B = UF$ on $\cA$ as well as a comonad $T = FU$ on $\cB$.  Let $\epsilon$ denote the counit $T \to 1$. There is a morphism of monads
$$
\xymatrix{
(\cB, 1) \ar[rr]^-{(U, U\epsilon)} && (\cA, B)
}
$$
which is mapped by the right adjoint $J \colon \mnd(\sC) \to \sC$
 to a 1-cell
$$
\xymatrix{
\cB^1 = \cB \ar[r]^-{U^{U\epsilon}} & \cA^B.
}
$$
Equivalently, we can view this as the 1-cell corresponding to the morphism of monads $(U, U\epsilon)$ under the 2-adjunction.
\begin{defn}
We call $U^{U\epsilon}$ the \emph{comparison 1-cell} associated to the adjunction $F\dashv U$.
\end{defn}
\begin{lem}\label{compcomonad}
We have that
$$
\xymatrix{
(\cB, T) \ar[rr]^-{(U^{U\epsilon}, 1)} && (\cA^B, \tilde B)
}
$$
is an opmorphism of comonads.
\end{lem}
\begin{proof}
There are two commutative squares
$$
\xymatrix{
(\cB, 1) \ar[rr]^-{(U, 1)} \ar[d]_-{(U, U\epsilon)} && (\cA,1) \ar@{=}[d] \\
(\cA, B) \ar[rr]_-{(1, \eta)} && (\cA, 1)
}\qquad
\xymatrix{
(\cB, 1) \ar[rr]^-{(U, U\epsilon)} \ar[d]_-{(T, 1)}&& (\cA, B) \ar[d]^-{(B, \tau)} \\
(\cB, 1) \ar[rr]_-{(U, U\epsilon)} && (\cA, B)
}
$$
in $\mnd(\sC)$, mapped to the two diagrams
$$
\xymatrix{
\cB \ar[rr]^-U \ar[d]_-{U^{U\epsilon}} && \cA \ar@{=}[d] \\
\cA^B \ar[rr]_-{U^B} && \cA
}\qquad
\xymatrix{
\cB \ar[rr]^-{U^{U\epsilon}} \ar[d]_-{T} && \cA^B \ar[d]^-{\tilde B} \\
\cB \ar[rr]_-{U^{U\epsilon}} && \cA^B
}
$$
in $\sC$. Applying Theorem~\ref{arisem} to the right-hand square yields an opmorphism of comonads $(U^{U\epsilon}, \chi) \colon (\cB, T) \to (\cA^B, \tilde B)$. Using the left-hand square, we see that
$$
\xymatrix{
U \ar[rrr]^-{\eta U} \ar@{=}[d] & & & UT \ar@{=}[d] \\
U^B U^{U\epsilon} \ar[rr]_-{\eta U^B U^{U\epsilon}} && \tilde B U^B U^{U\epsilon} \ar@{=}[r] & UT
}
$$
commutes, so by part 2 of Proposition~\ref{univprop}, $\chi = 1$, proving the Lemma.
\end{proof}

Now, consider a square
$$
\xymatrix{\cB \ar[r]^ U  \ar[d]_ S & \ar@{}[dl]^(.25){}="a"^(.75){}="b" \ar@{=>}_-{\Omega} "a";"b"   \cA \ar[d]^ C\\
	\cB \ar[r]_U  & \cA}
$$
where $(U, \Omega) \colon (\cB, S) \to (\cA, C)$ is an iso-opmorphism of comonads, so that we are in the context of Definition~\ref{lift}.  By Remark~\ref{univrem},
$$
\xymatrix{
(\cB, 1) \ar[rr]^-{(U, U\epsilon)} \ar[dd]_-{(S,1)} && (\cA, B)\ar[dd]^-{(C, \theta)}  \ar@{}[ddll]^(.25){}="a"^(.75){}="b" \ar@{=>}_-{\Omega} "a";"b"\\
\\
(\cB, 1) \ar[rr]_-{(U, U\epsilon)} && (\cA, B)
}
$$
is a monad 2-cell, giving rise to a 2-cell
$$
\xymatrix{
\cB \ar[r]^-{U^{U\epsilon}} \ar[d]_-{S} & \cA^B \ar[d]^-{C^\theta} \ar@{}[dl]^(.25){}="a"^(.75){}="b" \ar@{=>}_-{\tilde\Omega} "a";"b" \\
\cB \ar[r]_-{U^{U\epsilon}} & \cA^B
}
$$
and thus giving an opmorphism of comonads
$$
\xymatrix{
(\cB, S) \ar[rr]^-{(U^{U\epsilon}, \tilde\Omega)} && (\cA^B, C^\theta)
}
$$
or equivalently, a morphism of comonads
$$
\xymatrix{
(\cB, S) \ar[rr]^-{(U^{U\epsilon}, \tilde\Omega^{-1})} && (\cA^B, C^\theta).
}
$$
\begin{rem}
The morphism of comonads $(U^{U\epsilon}, \tilde\Omega^{-1})$ is the comparison 1-cell associated to the adjunction
$$
\xymatrix{
(\cA, C) \ar@{}[rr]|-{\perp}\ar@/^0.5pc/[rr]^-{(F, \Lambda)} & & \ar@/^0.5pc/[ll]^-{(U, \Omega^{-1})} (\cB, S)
}
$$
in $\cmd(\sC)$, where $\Lambda$ denotes the mate of $\Omega$.
\end{rem}

\begin{thm}\label{distcomp1cell}
Let $\theta, \chi$ be the distributive laws given by Corollary~\ref{arisec}. Then
$$
\xymatrix{
(\cB, \chi, T, S) \ar[rrr]^-{(U^{U\epsilon}, 1, \tilde\Omega^{-1})} & & & (\cA^B, \tilde\theta, \tilde B, C^\theta)
}
$$
is a 1-cell in $\cDist(\sC)$.
\end{thm}
\begin{proof}
By the above remarks, $(U^{U\epsilon}, \tilde\Omega^{-1})$ is a morphism of comonads, and $(U^{U\epsilon}, 1)$ is an opmorphism of comonads by Lemma~\ref{compcomonad}. Therefore, we need only check that the Yang-Baxter equation holds to prove the Theorem. By definition of $\theta, \chi$ there is a commutative diagram
$$
\xymatrix{
BUS \ar@{=}[d] \ar[r]^-{B \Omega^{-1}} & BCU \ar[r]^-{\theta U} & CBU \ar@{=}[d] \\
UTS \ar[r]_-{U \chi} & UST \ar[r]_-{\Omega^{-1} T} & CUT
}
$$ in $\sC$. Therefore, in $\mnd(\sC)$, the compositions of all the monad 2-cells in the diagrams
$$
\xymatrix{
(\cB, 1)  \ar[dd]_-{(T, 1)} \ar[rr]^-{(S, 1)} &  & \ar[dd]^-{(T, 1)}   (\cB, 1) \ar@{}[ddll]^(.25){}="a"^(.75){}="b" \ar@{=>}_-{\chi} "a";"b"\ar[rr]^-{(U, U\epsilon)} & & (\cA, B) \ar[dd]^-{(B, \tau)}  \ar@{}[ddll]^(.25){}="c"^(.75){}="d" \ar@{=>}_-{1} "c";"d"  \\
& & & &\\
(\cB, 1) \ar[rr]^-{(S, 1)} \ar@/_1.5pc/[ddrr]_-{(U U\epsilon)} & & (\cB, 1)  \ar@{}[dd]^(.25){}="e"^(.75){}="f" \ar@{=>}_-{\Omega^{-1}} "e";"f"\ar[rr]^-{(U, U\epsilon)}& & (\cA, B)  \\
& & & &\\
& & (\cA, B) \ar@/_1.5pc/[uurr]_-{(C, \theta)} & &
}
$$
$$\xymatrix{
(\cB, 1)  \ar@{}[ddrr]^(.25){}="c"^(.75){}="d" \ar@{=>}^-{\Omega^{-1}} "c";"d"\ar[rr]^-{(U, U\epsilon)} &  &   (\cA, B) \ar@{}[rrdd]^(.25){}="a"^(.75){}="b" \ar@{=>}^-{\theta} "a";"b" \ar[rr]^-{(B, \tau)} & & (\cA, B) \\
& & & &\\
(\cB, 1) \ar[uu]^-{(S, 1)}\ar[rr]^-{(U, U\epsilon)} \ar@/_1.5pc/[ddrr]_-{(T, 1)} & & (\cA, B)  \ar[uu]_-{(C, \theta)}\ar@{}[dd]^(.25){}="e"^(.75){}="f" \ar@{=>}_-{1} "e";"f"\ar[rr]^-{(B, \tau)}& & (\cA, B) \ar[uu]_-{(C, \theta)} \\
& & & &\\
& & (\cB, 1) \ar@/_1.5pc/[uurr]_-{(U, U\epsilon)} & &
}
$$
are the same. Applying $J$ to each composition gives the diagram
$$
\xymatrix{
\tilde B U^{U\epsilon} S \ar@{=}[d] \ar[rr]^-{\tilde B \tilde \Omega^{-1}} & &\tilde B C^\theta U^{U\epsilon} \ar[rr]^-{\tilde \theta U^{U\epsilon}} & & C^\theta \tilde B U^{U\epsilon} \ar@{=}[d] \\
U^{U\epsilon} TS \ar[rr]^-{U^{U\epsilon} \chi} & & U^{U\epsilon} ST \ar[rr]^-{\tilde\Omega^{-1} T} & &C^\theta U^{U\epsilon} T
}
$$
which is the Yang-Baxter equation.
\end{proof}
\subsection{From mixed to comonad distributive laws}
Finally, we explain how one can functorially assign a comonad distributive law to a mixed one. Let $\sC$ be a 2-category which admits Eilenberg-Moore constructions for monads. We begin by making a small observation. Let
$$
\xymatrix{ (\cA, B) \ar[rr]^-{(\Sigma, \sigma)}&&(\cD, A)}
$$
be a morphism of monads. Then
$$
\xymatrix{
(\cA, B) \ar[rr]^-{(\Sigma, \sigma)} \ar[dd]_-{(B, \tau)} && (\cD,A) \ar[dd]^-{(A, \tau)} \ar@{}[ddll]^(.25){}="a"^(.75){}="b" \ar@{=>}_-{\sigma} "a";"b" \\
\\
(\cA, B) \ar[rr]_-{(\Sigma, \sigma)} && (\cD, A)
}
$$
is a monad 2-cell, which is mapped by $J \colon \mnd(\sC) \to \sC$ to a 2-cell
$$
\xymatrix{
\cA^B \ar[r]^-{\Sigma^\sigma} \ar[d]_-{\tilde B} & \cD^A \ar[d]^-{\tilde A} \ar@{}[dl]^(.25){}="a"^(.75){}="b" \ar@{=>}_-{\tilde\sigma} "a";"b" \\
\cA^B \ar[r]_-{\Sigma^\sigma}& \cD^A
}
$$
Using this notation, we can now state the following Proposition. We omit the proof, which can be found in~\cite[p.~160]{MR0299653}.
\begin{prop}\label{StreetJ}
There is  a 2-functor $\overline J \colon \mnd(\sC) \to \cmd(\sC^*)^*$ defined by
$$
\xymatrix{
(\cA, B ) \ddtwocell<9>_{(\Sigma, \sigma)\ \ \ \ \ }^{\ \ \ \ \ \ \  (\Sigma', \sigma')}{^\alpha}  & & & &  (\cA^B, \tilde B) \ddtwocell<9>_{(\Sigma^{\sigma}, \tilde\sigma)\ \ \ \ \  \ \ \ }^{\ \ \ \ \ \ \ \ \ \ (\Sigma'^{\sigma'}, \tilde\sigma')}{^\tilde\alpha} \\
 &  & \longmapsto \\
(\cD,  A) & & & & (\cD^A, \tilde A)
}
$$
\end{prop}

By applying Proposition~\ref{StreetJ} to the 2-category $\cmd(\sC)$, we obtain the following:
\begin{cor}\label{2funcmixdist}
The assignment
$$
\xymatrix{
(\cA, \theta, B, C ) \ddtwocell<10>_{(\Sigma, \sigma, \gamma)\ \ \ \ \ \ \  }^{\ \ \ \ \ \ \ \ \ \  (\Sigma', \sigma', \gamma')}{^\alpha}  & & & &  (\cA^B, \tilde\theta, \tilde B, C^\theta) \ddtwocell<10>_{(\Sigma^{\sigma}, \tilde\sigma, \tilde\gamma)\ \ \  \ \ \  \ \ \ }^{\ \ \ \ \ \ \ \ \ \ \ \ (\Sigma'^{\sigma'}, \tilde\sigma', \tilde\gamma')}{^\tilde\alpha} \\
 &  & \longmapsto \\
(\cD, \psi, A, D) & & & & (\cD^A, \tilde\psi,  \tilde A, \tilde D)
}
$$
defines a 2-functor $\cMix(\sC) \to \cDist(\sC)$.
\end{cor}
Thus the image under this 2-functor of a general
1-cell $\theta \to \psi$ of mixed distributive laws
can be composed with the 1-cell given by the
comparison functor of Proposition~\ref{distcomp1cell} to give a new 1-cell $\chi \to
\tilde\psi$.

\section{Eilenberg-Moore constructions in \texorpdfstring{$\cat$}{Cat}}\label{emem}
We begin by explicitly describing the general 2-categorical constructions involving monads from the preceding sections, and discussing the dual versions of several of these.
\subsection{Eilenberg-Moore categories}\label{emcatsect}
Let $B$ be a monad on a category $\cA$.
\begin{defn}\label{algebra}
The \emph{Eilenberg-Moore category} of the monad $B$, denoted $\cA^B$, is the category whose objects are pairs $(X, \beta)$, called $B$-\emph{algebras}, where $X$ is an object of $\cA$ and $\beta \colon BX \to X$ is a morphism satisfying associativity and unitality axioms, that is, the diagrams
$$
\xymatrix{
BBX \ar[r]^-{{\mu_X}} \ar[d]_-{B \beta} & BX \ar[d]^-\beta \\
BX \ar[r]_-{\beta} & X
}
\qquad
\xymatrix{
X \ar@{=}[dr]\ar[r]^-{\eta_X} & BX \ar[d]^-{\beta} \\
& X
}
$$
commute. The morphisms $f \colon (X, \beta) \to (X', \beta')$, called \emph{$B$-algebra morphisms}, are morphisms $f \colon X \to Y$ in $\cA$ which are compatible with the algebra structures, that is, the diagram
$$
\xymatrix{
BX \ar[r]^-{Bf} \ar[d]_-\beta & BX' \ar[d]^-{\beta '} \\
X \ar[r]_-{f} & X'
}
$$
commutes.
\end{defn}
Let $(\Sigma, \sigma) \colon (\cA, B) \to (\cD, A)$ be a morphism of monads. The functor $\Sigma \colon \cA \to \cD$ lifts to a functor $\Sigma^\sigma \colon \cA^B \to \cD^A$, defined on objects by
$$
\Sigma^\sigma (X, \beta) =  (\Sigma X, \xymatrix{A\Sigma X \ar[r]^-{\sigma_X} & \Sigma B X \ar[r]^-{\Sigma \beta} & \Sigma X})
$$
and defined on morphisms by $\Sigma^\sigma f = \Sigma f$.
Also, any monad 2-cell $\alpha \colon (\Sigma, \sigma) \Rightarrow (\Sigma', \sigma')$ lifts in an obvious way to give a natural transformation $\tilde\alpha \colon \Sigma^\sigma \Rightarrow \Sigma'^{\sigma '}$. It is straightforward to check that these assignments define a 2-functor $J \colon \mnd(\cat) \to \cat$, exhibiting $\cat$ as a 2-category which admits Eilenberg-Moore constructions for monads (cf.\ Section~\ref{EM2cat}).

The canonical adjunction
$$
\xymatrix{
\cA \ar@{}[rr]|-{\perp}\ar@/^0.5pc/[rr]^-{F^B} & & \ar@/^0.5pc/[ll]^-{U^B} \cA^B
}
$$
is defined as follows. $U^B$ is the obvious forgetful functor, and $F^B$ sends an object $X$ to the \emph{free $B$-algebra} $(BX, \mu_X)$, while acting as $B$ on morphisms. Thus the comonad $\tilde B$ generated by this adjunction sends a $B$-algebra $(X, \beta)$ to the free $B$-algebra on $X$.

In fact, $\cat_*$ also admits Eilenberg-Moore constructions for monads; but of course, a monad in $\cat_*$ is nothing more than a comonad in $\cat$. The construction is obtained by dualising the monad case above. In particular:
\begin{defn}
Let $T$ be a comonad on a category $\cB$. A \emph{$T$-coalgebra} is a pair $(M, \nabla)$ where $M$ is an object of $\cB$, and $\nabla \colon M \Rightarrow TM$ is a natural transformation which satisfies coassociativity and counitality axioms, that is the diagrams
$$
\xymatrix{
X \ar[r]^-\nabla \ar[d]_-\nabla & TX \ar[d]^-{\delta_X } \\
TX \ar[r]_-\nabla & TTX
}\qquad
\xymatrix{
X \ar@{=}[dr] \ar[r]^-{\nabla} & \ar[d]^-{\epsilon_X} TX \\
& X
}
$$commute.
\end{defn}
These, along with the obvious notion of coalgebra morphism, constitute the \emph{Eilenberg-Moore} category of the comonad $T$, denoted $\cB^T$.

\subsection{Algebra structures on functors}\label{algebrafunc}
For a fixed category $\cA$, there are 2-functors
\begin{align*}
[\cA, -] \colon &\cat \to \cat \\
[-, \cA] \colon &\cat^* \to \cat
\end{align*}
where $[\cX, \cA] := \cat(\cX, \cA)$ is the category of functors $\cX \to \cA$ (cf.\ Example~\ref{hom2functors}). In particular, these both map monads to monads and comonads to comonads.

Let $B$ be a monad on $\cA$, and let $X \colon \cX \to \cA$ be a functor.

\begin{defn}\label{algebrafuncdef}
We say that a natural transformation $\beta \colon BX \Rightarrow X$ is a \emph{$B$-algebra structure on $X$} if $(X, \beta)$ is a $[\cX, \cA]$-algebra, i.e.\ an object of the Eilenberg-Moore category
$[\cX, \cA]^{[\cX, B]} $. We also say that $(X, \cX, \beta)$ is a \emph{$B$-algebra}.
\end{defn}
So, this definition is basically the same as Definition~\ref{algebra} except the object $X$ becomes a functor, and the structure morphism $\beta$ becomes a natural transformation. Both definitions are actually equivalent: we recover~\ref{algebra} from~\ref{algebrafuncdef} by choosing $\cX$ to be the terminal category $\mathbbm{1}$. Thinking of algebras as functors is advantageous, however, since it allows us to dualise in the right way:

\begin{defn}
Let $Y \colon \cA \to \cY$ be a functor. We say that a natural transformation $\omega \colon YB \Rightarrow Y$ is a \emph{$B$-opalgebra structure on $Y$} if $(Y, \omega)$ is a $[B, \cY]$-algebra, i.e.\ an object of the Eilenberg-Moore category $[\cA, \cY]^{[B, \cY]}$. We also say that $(Y, \cY, \omega)$ is a \emph{$B$-opalgebra}. Explicitly, there are commutative diagrams
$$
\xymatrix{
YBB\ar[r]^-{Y \mu} \ar[d]_-{\omega B} & YB \ar[d]^-\omega \\
YB \ar[r]_-\omega & B
}
\qquad
\xymatrix{
Y \ar[r]^-{Y \eta} \ar@{=}[dr] & YB \ar[d]^-\omega \\
& Y
}
$$
\end{defn}
Of course, we can dualise in a different way to obtain the notion of \emph{coalgebra} and \emph{opcoalgebra} structures on functors.

\subsection{Kleisli categories}
The 2-category $\cat^*$ also admits Eilenberg-Moore constructions for monads. We now describe the 0-cell part of the 2-functor $\mnd(\cat^*) \to \cat^*$. Let $B$ be a monad on a category $\cA$.
\begin{defn}
The \emph{Kleisli category} of the monad $B$, denoted $\cA_B$, is the category whose objects are precisely those of $\cA$, and whose morphisms $X \to Y$ are morphisms $X \to BY$ in $\cA$. The composite of
$$
\xymatrix{
X \ar[r]^-f & Y
},
\qquad
\xymatrix{
Y\ar[r]^-g & Z
}
$$
in $\cA_B$ is given by the composite
$$
\xymatrix{
X \ar[r]^-f & TY \ar[r]^-{Tg} & TTZ \ar[r]^-{\mu_Z} & TZ
}
$$
in $\cA$. The identity morphism $X \to X$ in $\cA_B$ is given by $\eta_X \colon X \to BX$ in $\cA$.
\end{defn}
We also have that $\cat_*^*$ admits Eilenberg-Moore constructions for monads, giving rise to the dual notion of the Kleisli category $\cB_T$ for a comonad $T$ on a category $\cB$, defined in a similar way.
\subsection{The comparison functor}
Let
$$
\xymatrix{
\cA \ar@{}[rr]|-{\perp}\ar@/^0.5pc/[rr]^-F & & \ar@/^0.5pc/[ll]^-U \cB
}
$$
be an adjunction between categories $\cA, \cB$. From Section~\ref{emcatsect}, the comparison functor $U^{U\epsilon} \colon \cB \to \cA^B$ is defined on objects by
$$
U^{U\epsilon} M = (UM, \xymatrix{BUM = UFUM \ar[r]^-{U\epsilon_M} & UM}) .
$$

%
%

\chapter{Distributive laws and their coalgebras}\label{DISTRIBUTIVE}

Now we concentrate on the theory of distributive laws in the 2-category $\sC = \cat$.
Inspired by~\cite{MR2787298, MR3020336}, in Section~\ref{galoismapsect} we focus on a special case of Corollary~\ref{arisec}, and how different instances thereof are related by a \emph{Galois map} (Definition~\ref{galoisdeff}). In Section~\ref{coalgsect} we study $\chi$-coalgebras~\cite{MR2411421, MR2415479} for a distributive law $\chi$ of comonads, and view these as coefficient modules of the cyclic homology theories which appear later. We conclude this chapter in Section~\ref{examplesect} by studying some simple examples in the category of sets.

Hereafter we are concerned with distributive laws only in $\cat$, and so we write $\dist$ in place of $\cDist(\cat)$, and $\mix$ in place of $\cMix(\cat)$.

The work in this chapter is original. Sections~\ref{galoismapsect} (inspired by~\cite{MR2787298, MR3020336}) and~\ref{coalgsect} are based on~\cite{2} and~\cite[\S2--3]{1}.

\section{The Galois map}\label{galoismapsect}
Corollary~\ref{arisec} yields comonad distributive laws
from lifts through an adjunction, and different lifts
produce different distributive laws. Here we describe
how these are related in terms of suitable
generalisations of the Galois map from the theory of
Hopf algebras (see Section~\ref{whygalois} below for the
example motivating the terminology).

Suppose that
$$
\xymatrix{\cA \ar@/^{0.5pc}/[rr]^-F
\ar@{}[rr]|-{\perp}&& \ar@/^0.5pc/[ll]^-U \cB }
$$
is an adjunction between categories $\cA, \cB$ which generates a monad $B = UF$  on $\cA$
and a comonad $T = FU$ on $\cB$.
\subsection{Main application of the lifting theorem}\label{mainapp}
Suppose furthermore that $S$ is a lift of an endofunctor $C$ through a single adjunction as in Definition~\ref{lift}, i.e.\ we have a square
$$
\xymatrix{\cB \ar[r]^ U  \ar[d]_ S & \ar@{}[dl]^(.25){}="a"^(.75){}="b" \ar@{=>}_-{\Omega} "a";"b"   \cA \ar[d]^ C\\
	\cB \ar[r]_U  & \cA}
$$
where $\Omega$ is a natural isomorphism. By Corollary \ref{arisec}, we have two natural transformations
$$
	\xymatrix@=2.9em{
	\theta \colon BC = UFC \ar[r]^-{UFC\eta} & UFCUF \ar[r]^-{UF\Omega F} & UFUSF \ar[r]^-{U\epsilon SF} & USF \ar[r]^-{\Omega^{-1} F} & CUF = CB}
$$
and
$$
	\xymatrix@=2.9em{\chi \colon
	TS = FUS \ar[r]^-{F \Omega^{-1}} & FCU \ar[r]^-{FC\eta U} & FCUFU \ar[r]^-{F\Omega FU} & FUSFU \ar[r]^-{\epsilon SFU} & SFU = ST
	}
$$
such that $(C, \theta)$ is an morphism of monads, and $(S, \chi)$ is an opmorphism of comonads.

\begin{exa}\label{trivv}
A trivial example which nevertheless plays a r\^ole below is the case where $C = B$, $S = T$, and $\Omega = 1$ (cf.~Example~\ref{trivv2}). In this case, $\theta$ and $\chi$ are given by
\begin{align*}
&\xymatrix{BB = UFUF\ar[r]^-{U\epsilon F} & UF \ar[r]^-{UF\eta} & UFUF = BB} \\
&\xymatrix{TT = FUFU \ar[r]^-{\epsilon FU} & FU \ar[r]^-{F \eta U} & FUFU = TT}
\end{align*}
respectively.
\end{exa}

Functors do not
necessarily lift nor extend through an
adjunction (for example, the functor on $\cSet$
which assigns the empty set to each set does not lift
to the category of vector spaces over a given field), and if they do, they may not do so
uniquely.  Theorem~\ref{arisem} says only that once a
lift or extension is chosen, there
is a unique compatible pair of natural transformations $
\theta $ and $\chi$.

Recall from Corollary~\ref{arisec} that if $(U, \Omega)$ is an opmorphism of comonads, then $\theta$ and $\chi$ are distributive laws, in which case we say that they arise from the adjunction $F \dashv U$ cf.~Definition~\ref{arisedeffo}.

\begin{prop}
Every mixed distributive law, and every comonad distributive law, arises from an adjunction.
\end{prop}
\begin{proof}
The statement for mixed distributive laws follows from Proposition~\ref{everydistlawarises} and the following comments. The statement for comonad distributive laws is proved dually: take $\cA$ to be the Kleisli category $\cB_T$ in which case a comonad distributive law $\chi \colon TS \Rightarrow ST$ yields an extension $C$ of a functor $S$.
\end{proof}
Recall that there is a 2-functor $\mix \to \dist$ as in Corollary~\ref{2funcmixdist}.  It is those distributive laws in the image of this 2-functor that are the main ones of interest in later chapters.
Analogously, we obtain a $2$-functor $\dist \rightarrow
\mix$ by taking extensions to Kleisli
categories.

\subsection{Generalising the Galois map}\label{galoismapsct}
Suppose that in addition to $S$, we have another lift $W \colon \cB \to \cB$ of the endofunctor $C$ through the adjunction $F\dashv U$. This means that we have two squares
$$
\xymatrix{\cB \ar[r]^ U  \ar[d]_ S & \ar@{}[dl]^(.25){}="a"^(.75){}="b" \ar@{=>}_-{\Omega} "a";"b"   \cA \ar[d]^ C\\
	\cB \ar[r]_U  & \cA} \qquad
\xymatrix{\cB \ar[r]^ U  \ar[d]_ W & \ar@{}[dl]^(.25){}="a"^(.75){}="b" \ar@{=>}_-{\Phi} "a";"b"   \cA \ar[d]^ C\\
	\cB \ar[r]_U  & \cA}
$$
where $\Omega \colon CU \Rightarrow US$ and $\Phi \colon CU \Rightarrow UW$ are natural isomorphisms. For any object $X$ in $\cA$ and any object $Y$ in $\cB$, consider the diagram
$$
\xymatrix{
\cB(FX, SY) \ar[d]_-\cong \ar@{.>}[rr] & & \cB(FX, WY) \\
\cA(X, USY) \ar[rr] & & \cA(X, UWY) \ar[u]_-\cong
}
$$
where the vertical maps are induced by the adjunction $F \dashv U$, and the lower map is induced by the composition
$$
\xymatrix{
USY \ar[r]^-{\Omega_Y^{-1}}  & CUY \ar[r]^-{\Phi_Y}  & UWY.
}
$$
The dashed arrow defines one component of a natural
isomorphism
$$
	\Gamma^{S, W} \colon \cB(F-, S-) \Rightarrow
\cB(F-, W-)
$$
of functors $\cA^* \times \cB \rightarrow \cSet$.
\begin{defn}\label{galoisdeff}
We call the natural isomorphism $\Gamma^{S,W}$ the \emph{Galois map}.
\end{defn}

The following properties
are easy consequences of the definition:

\begin{prop}\label{sunshines}
Let $S$ and $W$ be two lifts
of an endofunctor $C$ through an adjunction $F \dashv
U$. Then:
\begin{enumerate}
\item The inverse of
$\Gamma^{S,W}$ is given by $\Gamma^{W,S}$.
\item The natural transformation $\Gamma^{S,W}$ maps a morphism $f
\colon F X \rightarrow S Y$ to
$$
	\xymatrix{ F X
\ar[r]^-{F \eta_X} & FUF X \ar[r]^-{FU f} &
FUS Y \ar[rr]^{F(\Phi_Y \circ \Omega^{-1}_Y)} &
& FUW Y \ar[r]^-{\epsilon_{W Y}} & W Y. }
$$
\item If $\chi^{S}$ and $ \chi^{W}$ denote the natural transformations
determined by the two lifts, then the composite
$$
\xymatrix@C=3em{
\cB( FUS, SFU ) \ar[r]^-{\Gamma^{S,W}}_-\cong & \cB (FUS, WFU) \ar[r]_-\cong & \cB( FUW, WFU)
}
$$
maps $\chi^S$ to $\chi^W$, where the right-hand isomorphism is induced by the composite $\Omega \circ \Phi^{-1} \colon UW \Rightarrow US$.
\end{enumerate}
\end{prop}

So, in the applications of Corollary~\ref{arisec}, all
distributive laws obtained from different lifts
$S,W$ of a
given comonad through an adjunction are obtained from
each other by application of $\Gamma^{S,W}$.

In particular, consider the situation of Example~\ref{trivv}, where we have a lift of $B$ itself through an adjunction, that is, a commutative square:
$$
\xymatrix{
\cB \ar[r]^-{U} \ar[d]_-{T} & \cA \ar[d]^-B \\
\cB \ar[r]_-{U} & \cA
}
$$
Let
$W$ be any other lift of $B$ through the
adjunction. By taking $X$ to be $U Y$ for an object
$Y$ of $\cB$, one obtains  a natural transformation
$ \Gamma^{T,W}
\colon \cB (T-,T-) \rightarrow \cB(T-,W-) $ that we can
evaluate on $ 1 \colon T Y \rightarrow T Y$, which
produces a natural transformation
$\Gamma^{\rT,\rV}(1) \colon T \Rightarrow W$.

Adapting \cite[Definition~1.3]{MR2651345}, we define:
\begin{defn} We say that $F$ is \emph{$W$-Galois}
if $$ \xymatrix{ \Gamma^{T,W}(1) \colon T=FU
\ar[r]^-{F \eta U} & FUFU = FUT
\ar[r]^-{F \Phi} & FUW  \ar[r]^-{\epsilon W}
& W } $$ is an isomorphism.
\end{defn}

The following proposition provides the connection to
Hopf algebra theory:

\begin{prop}\label{wisga}
If $F$ is $W$-Galois and
$ \theta \colon BB \Rightarrow BB$ is the natural transformation arising from the lift $W$ of $B$, then the
natural transformation
$$
	\xymatrix{ \beta \colon BB
	\ar[r]^-{B \eta B} & BBB \ar[r]^-{\theta B}
	& BBB \ar[r]^-{B \mu} & BB  }
$$
is an
isomorphism.
\end{prop}
\begin{proof}
If $F$ is
$W$-Galois, then $ U\Gamma^{T,W}(1)F$ is an
isomorphism
$$
	\xymatrix@C=2.455em{ UTF=UFUF
\ar[rr]^-{UF \eta UF} & & UFUFUF =
UFUT F \ar[r]^-{UF \Phi F} &
UFUW F \ar[r]^-{U\epsilon WF} &
UWF. }
$$
Let now $ \chi \colon TW \Rightarrow
WT$ be the natural transformation corresponding to $ \theta $
as in Theorem~\ref{arisem} and Corollary~\ref{arisec}.  Inserting $\varepsilon
W=(W \varepsilon) \circ \chi$ and $U \chi \circ
UF \Phi=\Phi FU \circ \theta U$ and
$B=UF$, the isomorphism becomes $$ \xymatrix@C=2.5em{
UTF=BB \ar[r]^-{B \eta B} & BBB
\ar[r]^-{\theta B} & BBB=BUFUF
\ar[r]^-{\Phi FUF} & UWFUF
\ar[r]^-{UW\epsilon F} & UWF } $$ Finally,
we have by construction $U \varepsilon F=\mu $, and
using the naturality of $\Phi$ this gives
$UW\varepsilon   F \circ \Phi FUF= \Phi F
\circ BU\varepsilon F$.  Composing the
above isomorphism with $\Phi^{-1} F$ gives $ \beta $.
\end{proof}

It is this associated map $ \beta $ that is used to
distinguish Hopf algebras amongst bialgebras, see
Section~\ref{whygalois} below.

\section{Coalgebras for distributive laws}\label{coalgsect}

We now discuss $ \chi
$-coalgebras, which serve as
coefficient modules  in the constructions of Chapter~\ref{CYCLIC}. Let $\chi \colon TS \Rightarrow ST$ be a distributive law of comonads on a category $\cB$.
\subsection{Coalgebras for distributive laws}\label{chicoalgs}

\begin{defn}\label{coalgdef}
A \emph{$\chi$-coalgebra} $(M,\rho)$ consists of an object $M$ in $\cB$, together with a morphism $\rho \colon TM \to SM$
in $\cB$ such that the following two diagrams commute:
$$
\xymatrix{ TM \ar[r]^-{\Delta M} \ar[d]_-{\rho} &
TTM \ar[r]^{T\rho} &TSM \ar[d]^-{\chi M} \\
SM \ar[r]_-{\Delta  M} & 	SSM & STM
\ar[l]^-{S\rho} } \quad\quad\quad
\xymatrix{ & TM
\ar[dl]_-{\epsilon M} \ar[d]^-{\rho}\\ M & SM
\ar[l]^-{\epsilon  M} } $$
\end{defn}

\begin{rem}\label{extrarem}
The distributive law $\chi$ is a comonad in the 2-category $\cmd(\cat)$, but a $\chi$-coalgebra is \emph{not} a coalgebra for this comonad. Rather, the composite $TS$ becomes a comonad with comultiplication and counit given by
$$
\xymatrix{
TS \ar[r]^-{\delta\delta} & TTSS \ar[rr]^-{T\chi S} && TSTS,
}
\qquad
\xymatrix{
TS \ar[r]^-{\epsilon\epsilon} & 1
}
$$
respectively~\cite{MR0241502}. There is a morphism of comonads
$$
\xymatrix{
(\cB, TS) \ar[rr]^-{(1,T\epsilon)} & & (\cB, T)
}
$$
which induces a forgetful functor $\cB^{TS} \to \cB^T$. The $\chi$-coalgebra structures $\rho$ on an object $M$ are equivalent to those $TS$-coalgebra structures on $TM$ whose image under this forgetful functor is cofree, see e.g.~\cite[Prop.~1.9]{MR2415479}.\end{rem}

For a fixed category $\cX$, the 2-functor $\cmd \colon \twocat \to \twocat$ maps the 2-functor $[\cX, -] \colon \cat \to \cat$ to the 2-functor $\cmd([\cX, -]) \colon \cmd(\cat) \to \cmd(\cat)$. This sends comonads to comonads, i.e.\ it sends distributive laws to distributive laws. Therefore,
$$
[\cX, \chi] \colon [\cX, T] \circ [\cX, S] \Rightarrow [\cX, S] \circ [\cX, T]
$$
is a distributive law of comonads. Similarly, for a fixed category $\cY$,
$$
[\chi, \cY] \colon [T, \cY] \circ [S, \cY] \Rightarrow [S, \cY] \circ [T, \cY]
$$
is a distributive law of comonads. With this in mind:

\begin{defn}
A \emph{$\chi$-coalgebra structure} on a functor $M \colon \cX \to \cB$ is a natural transformation $\rho \colon TM \Rightarrow SM$ such that $(M, \rho)$ is a $[\cX, \chi]$-coalgebra in $[\cX, \cB]$. \end{defn}

So, in a similar way to Section~\ref{algebrafunc}, it does not really matter whether we talk about $\chi$-coalgebras as objects or as functors since we may choose $\cX = \mathbbm{1}$ to recover Definition~\ref{coalgdef}.

Dually, we have:
\begin{defn}
A \emph{$\chi$-opcoalgebra structure} on a functor $N \colon \cB \to \cY$ is a natural transformation $\lambda \colon NS \Rightarrow NT$ such that $(N, \lambda)$ is a $[\chi, \cY]$-coalgebra in $[\cB, \cY]$. Explicitly, the two diagrams
$$
\xymatrix{
NS  \ar[r]^-{N\delta} \ar[d]_-\lambda & NSS \ar[r]^-{\lambda S} & NTS \ar[d]^-{N \chi} \\
NT \ar[r]_-{N \delta} & NTT & NST \ar[l]^-{\lambda T}
}
\qquad
\xymatrix{
& NS \ar[d]^-{\lambda} \ar[dl]_-{N \epsilon} \\ N  & NT \ar[l]^-{N\epsilon}
}
$$
commute.
\end{defn}

\begin{rem}
Dualising Remark~\ref{extrarem} appropriately, the $\chi$-opcoalgebra structures $\lambda$ on $N$ correspond to $TS$-opcoalgebra structures on $NS$ whose underlying $S$-opcoalgebra is opcofree.
\end{rem}

Now suppose we are in the comonad setting of Section~\ref{mainapp}, so we have an adjunction $F\dashv U$ and a square
$$
\xymatrix{\cB \ar[r]^ U  \ar[d]_ S & \ar@{}[dl]^(.25){}="a"^(.75){}="b" \ar@{=>}_-{\Omega} "a";"b"   \cA \ar[d]^ C\\
	\cB \ar[r]_U  & \cA}
$$
where $(U,\Omega)$ is an iso-opmorphism of comonads. Via Corollary~\ref{arisec}, the square gives rise to a distributive law $\chi \colon TS \Rightarrow ST$. The following characterises $ \chi $-coalgebras
in this case.

\begin{prop}\label{chicoalgprop}
Let $M \colon \cX \rightarrow \cB$
be a functor.
\begin{enumerate}
\item 
Coalgebra structures for $\chi$ on $M$ correspond to
$C$-coalgebra structures $\nabla$ on the
functor $UM \colon \cX \rightarrow \cA$.
\item Let
$S$ and $W$ be two lifts of the functor $C$ through the
adjunction, and let $\chi^S$ and $\chi^W$ denote the
comonad distributive laws determined by the lifts
$S$ and $W$ respectively. Then the composite
$\Gamma^{S,W}$ maps $\chi^S$-coalgebra
structures $\rho^S$ on $M$ bijectively to 
$\chi^W$-coalgebra structures $\rho^W$ on $M$.
\end{enumerate}
\end{prop}

\begin{proof} For part (1),
$\chi$-coalgebra structures $\rho \colon FUM
\Rightarrow SM$ are mapped under the adjunction to $\nabla
\colon UM \Rightarrow USM \cong CUM$. Part (2)
follows immediately since $\Gamma^{S,W}$ is
the composition of the adjunction isomorphisms and
$\Phi \circ \Omega^{-1}$.
\end{proof}

Dually, given an adjunction
$V \dashv G$ for the comonad $S$ and an extension
$Q $ of the comonad $T$ through the
adjunction,
$ \chi $-opcoalgebra structures on
$N \colon \cB \rightarrow \cY$ correspond in complete analogy
to $Q$-opcoalgebra structures on $NV$.

\subsection{Twisting by 1-cells}\label{twistcoeff}
Here we show how factorisations of distributive laws as
considered in \cite{2} can be used to obtain new
$ \chi $-coalgebras from old ones.

Generalising the notion of a module over a monoidal category, we make the following definition:
\begin{defn}\label{leftCmodule}
Let $\sC$ be a 2-category. A \emph{left $\sC$-module}  $\sM$ consists of a category $\sM(\cA)$ for each 0-cell $\cA$ in $\sC$ and functors $\rhd_{\cA, \cB} \colon \sC(\cA, \cB) \times \sM(\cA) \to \sM(\cB)$, called \emph{left actions}, such that the diagrams
$$
\xymatrix{
\sC(\cB, \cC) \times \sC(\cA, \cB) \times \sM(\cA) \ar[rr]^-{1 \times \rhd_{\cA, \cB}} \ar[d]_-{\circ \times 1} && \sC(\cB, \cC) \times \sM(\cB) \ar[d]^-{\rhd_{\cB, \cC}} \\
\sC(\cA, \cC) \times \sM(\cA) \ar[rr]_-{\rhd_{\cA, \cC}} && \sM(\cC)
}
$$
and
$$
\xymatrix{
\sM(\cA) \ar[rr]^-{u \times 1} \ar@{=}[drr] && \sC(\cA, \cA) \times \sM(\cA) \ar[d]^-{\rhd_{\cA, \cA}} \\
& &\sM(\cA)
}
$$
commute in $\cat$, for all 0-cells $\cA, \cB, \cC$ in $\sC$.
\end{defn}
Hereafter, we omit the subscripts from the action functors.
\begin{rem}
The functors
$$\rhd \colon \sC(\cA, \cB) \times \sM(\cA) \to \sM(\cB)$$
correspond under the closed symmetric monoidal structure of $\cat$ to functors
$$
\sC(\cA, \cB) \to  [\sM(\cA), \sM(\cB)]
$$
and thus, a left $\sC$-module is nothing more than a 2-functor $\sM \colon \sC \to \cat$.
\end{rem}
Dually, one defines a \emph{right $\sC$-module} $\sN$, with right actions $$\lhd \colon \sN(\cB) \times \sC(\cA, \cB) \to \sN(\cA)$$
which can alternatively be viewed as 2-functor $\sN \colon \sC^* \to \cat$.

For each distributive law $\chi \colon TS \Rightarrow ST$ of comonads in a category $\cB$, we define a category $\sR(\chi)$
as follows. The objects are $\chi$-coalgebras $(M, \cX, \rho)$, i.e.\ $\chi$-coalgebra structures on functors $\cX \to \cB$ where $\cX$ is allowed to vary. The morphisms
$$\xymatrix{(M, \cX, \rho) \ar[rr]^-{(F, \phi)} && (M', \cX', \rho')}$$
 consist of a functor $F \colon \cX \to \cX'$ and a natural transformation $\phi \colon M \Rightarrow M'F$ such that the diagram
$$
\xymatrix{
TM \ar[r]^-{T\phi} \ar[d]_-\rho & TM'F \ar[d]^-{\rho' F} \\
SM \ar[r]_-{S \phi} & SM'F
}
$$
commutes.

Recall that $\dist$ denotes the 2-category $\cmd(\cmd(\cat)^*)^*$ of comonad distributive laws in $\cat$, cf.~Section~\ref{finallyover}

\begin{thm}\label{twist}
The categories $\sR(\chi)$ define a left $\dist$-module.
\end{thm}
\begin{proof}
We define the action $\rhd \colon \dist(\chi, \tau) \times \sR(\chi) \to \sR(\tau)$ as follows. For a 1-cell
$$\xymatrix{
(\cB, \chi, T, S) \ar[rr]^-{(\Sigma, \sigma, \gamma)} && (\cD, \tau, G, C)
}
$$
in $\dist$, we  define
$$
(\Sigma, \sigma, \gamma) \rhd (M, \cX, \rho) = (\Sigma M, \cX, \gamma M \circ \Sigma\rho \circ \sigma M)
$$
and on morphisms we define $\alpha \rhd ({\varphi},F)$ to be the pair $(\alpha {\varphi}, F)$. First we check that this assignment is a well-defined functor. Consider the diagram
$$
\xymatrix@C=3.5em{
G \Sigma M \ar[d]_-{\sigma M} \ar[r]^{\delta \Sigma M} & GG\Sigma M \ar[r]^-{G \sigma M} & G\Sigma TM \ar[d]^-{\sigma{TM}} \ar[r]^-{G\Sigma\rho} & G \Sigma SM \ar[d]^-{\sigma SM} \ar[r]^-{G \gamma M} & GC\Sigma M \ar[d]^-{\tau \Sigma M} \\
\Sigma TM \ar[dd]_-{\Sigma \rho} \ar[rr]_-{\Sigma \delta M} & & \Sigma TTM \ar[r]_-{\Sigma T\rho}  & \Sigma TSM \ar[d]^-{\Sigma \chi M} & CG\Sigma M \ar[d]^-{C \sigma M } \\
& & & \Sigma STM \ar[r]^-{\gamma{TM}} \ar[d]_-{\Sigma S \rho}& C\Sigma TM \ar[d]^-{C \Sigma \rho} \\
\Sigma SM \ar[d]_-{\gamma{M}}\ar[rrr]_-{\Sigma \delta M}& & & \Sigma SSM \ar[r]_-{\gamma {SM}}&  C \Sigma SM \ar[d]^-{C \gamma M} \\
C\Sigma M \ar[rrrr]_-{\delta {\Sigma M}}& & & & CC \Sigma M
}
$$
The top-left and bottom rectangles commute because $\sigma, \gamma$ are compatible with comultiplication, the middle-left rectangle commutes because $M$ is a $\chi$-coalgebra, the top-right diagram commutes by the Yang-Baxter condition, and the remaining squares commute by naturality of $\sigma,\gamma$. Therefore the outer rectangle commutes.

Consider the triangle
$$
\xymatrix@=3em{
G\Sigma M \ar[rrd]_-{\epsilon {\Sigma M}} \ar[r]^-{\sigma M} & \Sigma TM \ar[dr]^-{\Sigma \epsilon M} \ar[rr]^-{\Sigma \rho} & & \Sigma SM \ar[dl]_-{\Sigma \epsilon M} \ar[r]^-{\gamma M} & C \Sigma M \ar[dll]^-{\epsilon {\Sigma M}}\\
& & \Sigma M & &
}
$$
The middle triangle commutes because $M$ is a $\chi$-coalgebra, and the other two inner triangles commute by the compatibility of $\sigma, \gamma$ with the counit. Therefore the outer triangle commutes. This shows that $\rhd$ is well-defined on objects.

Let $({\varphi},F) \colon (M,\cX, \rho) \to (M',\cX', \rho')$ be a morphism of $\chi$-coalgebras, and let $\alpha$  be a 2-cell $ (\Sigma, \sigma, \gamma) \Rightarrow (\Sigma', \sigma', \gamma')$ of distributive laws.  Consider the diagram
$$
\xymatrix@=3em{
G \Sigma M \ar[d]_-{\sigma M} \ar[r]^-{G \alpha M} & G \Sigma' M \ar[d]^-{\sigma' M} \ar[r]^-{G \Sigma' {\varphi}} & G \Sigma' M'F \ar[d]^-{\sigma' {M' F}} \\
\Sigma TM \ar[d]_-{\Sigma \rho} \ar[r]^-{\alpha {TM}} & \Sigma' TM \ar[d]^-{\Sigma' \rho} \ar[r]^-{\Sigma' T {\varphi}} & \Sigma' TM'F \ar[d]^-{\Sigma' \rho' F} \\
\Sigma SM \ar[r]_-{\alpha {SM}} & \Sigma' SM \ar[r]_-{\Sigma' S{\varphi}} & \Sigma' S M'F
}
$$
The top-left square commutes since $\alpha$ is a 2-cell, the top-right square commutes by naturality of $\sigma$, the bottom-left square commutes by naturality of $\alpha$, and the bottom-right square commutes since $({\varphi},F)$ is a $\chi$-coalgebra morphism. Thus the outer triangle commutes, which shows that $\alpha \rhd ({\varphi},F)$ is a $\chi$-coalgebra morphism.

It is clear that $\rhd$ respects identities and composition of morphisms (because the vertical and horizontal compositions of natural transformations are compatible with each other), 
so $\rhd$ is well-defined on morphisms. It is also routine to check that $\rhd$ satisfies the required axioms of Definition~\ref{leftCmodule}, thus proving the Theorem.
\end{proof}
The axioms of Definition~\ref{leftCmodule} tell us immediately:
\begin{cor}
For a fixed distributive law $\chi$, the category $\sR(\chi)$ of $\chi$-coalgebras is a strict left module category for the strict monoidal category $\dist(\chi, \chi)$.
\end{cor}

This is the main result of~\cite{2}.
Dually, one can construct the category $\sL(\chi)$ of $\chi$-opcoalgebras, and show that $\sL$ is a right $\dist$-module, and that $\sL(\chi)$ is a strict right module category for $\dist(\chi, \chi)$.

In summary, we have shown that for a 1-cell of distributive laws
$$
\xymatrix{(\cB, \chi, T, S) \ar[rr]^-{(\Sigma, \sigma, \gamma)} & & (\cD, \tau, G, C)}
$$
we can canonically twist any $\chi$-coalgebra and any $\tau$-opcoalgebra
$$
(\xymatrix{\cX \ar[r]^-{M} & \cB
}, \xymatrix{ TM \ar@{=>}[r]^-{\rho} & SM }),\qquad
(\xymatrix{\cD \ar[r]^-{N} & \cY
}, \xymatrix{ NC \ar@{=>}[r]^-{\lambda} & NG })
$$
by $(\Sigma, \sigma, \gamma)$ to yield a $\tau$-coalgebra
$$
(
\xymatrix{\cX \ar[r]^-{M} & \cB \ar[r]^-{\Sigma} & \cD},
\xymatrix{
G\Sigma M \ar[r]^-{\sigma M} & \Sigma TM \ar[r]^-{\Sigma \rho} & \Sigma SM \ar[r]^-{\gamma M} & C\Sigma M
}
)
$$
and a $\chi$-opcoalgebra
$$
(
\xymatrix{
\cB \ar[r]^-{\Sigma} & \cD \ar[r]^-N & \cY
},
\xymatrix{
N\Sigma S \ar[r]^-{N\gamma} & NC\Sigma \ar[r]^-{\lambda \Sigma} & NG\Sigma \ar[r]^-{N\sigma} & N\Sigma T
}
).
$$
This will be applied in Section~\ref{twistsec} below in
the context of duplicial functors.

\subsection{From coalgebras for comonads to those for
distributive laws}
In the remainder of this section, we discuss a class of
coefficient objects that
lead
to contractible simplicial objects; see
Proposition~\ref{trivcontract} below.  In the Hopf
algebroid setting, these are the Hopf (or entwined)
modules as studied in \cite{MR3020336,MR1604340}.

Note first that
$T$-coalgebras can be equivalently viewed as
$1$-cells \emph{from} the trivial distributive law, and dually, $T$-opcoalgebras correspond to $1$-cells \emph{to} the trivial distributive law:
\begin{prop}\label{triv1cell}
Let $\chi \colon TS\Rightarrow  ST$ be a comonad distributive law. Then:
\begin{enumerate}
  \item Any $S$-coalgebra $(M, \cX, \nabla^S)$ defines a 1-cell 
    $$
\xymatrix{
(\cX, 1, 1, 1) \ar[rr]^-{(M, \epsilon M, \nabla^S)} && (\cB, \chi, T, S)
}
  $$
  in $\dist$, and all 1-cells $1 \to \chi$ are of this form.
  \item Any $T$-opcoalgebra $(N, \cY, \nabla^T)$
defines a $1$-cell
$$
\xymatrix{
 (\cB, \chi, T, S) \ar[rr]^-{(N,  \nabla^T, N\epsilon)} && (\cY, 1, 1, 1).
}
$$
in $\dist$, and all $1$-cells $\chi \to 1$ are of this form.
\end{enumerate}
\end{prop}

Furthermore, these $1$-cells can also be viewed as $ \chi
$-coalgebras:

\begin{prop}\label{triv}
Let $\chi \colon TS \Rightarrow ST$ be a comonad
distributive law. Then:
\begin{enumerate}
\item Any $S$-coalgebra $(M,
\cX, \nabla^S)$ defines a
$\chi$-coalgebra $(M, \cX, \epsilon
\nabla^S)$.  \item Any $T$-opcoalgebra $(N,
\cY, \nabla^T)$ defines a
$\chi$-opcoalgebra $(N, \cY, \nabla^T
\epsilon^S)$.
\end{enumerate}
\end{prop}

Note, however, that there is no known way to associate
a $1$-cell in $\dist$ to an arbitrary $\chi$-(op)coalgebra.

\subsection{Entwined algebras}
Finally, we describe how
$ \chi $-coalgebras as in Proposition~\ref{triv}
are in some sense lifts of
entwined (also called mixed) algebras; throughout,
$\theta \colon BC \Rightarrow CB$ is a mixed
distributive law between a monad $B$ and a
comonad $C$ on a category $\cA$.

\begin{defn}
Let $M$ be an object in $\cA$ equipped with both a
$B$-algebra structure $\beta \colon B M \rightarrow M$ and
a $C$-coalgebra structure $\nabla \colon M \rightarrow
CM$. We say that the $(M, \beta,
\nabla)$ is an \emph{entwined algebra with respect to
$\theta$}, or a $\theta$-\emph{entwined algebra}, if the diagram
\begin{equation}\label{entwinedcondition}		
\begin{array}{c}
\xymatrix{ B M
\ar[d]_-{B \nabla} \ar[r]^-{\beta} & M
\ar[r]^-{\nabla} & CM\\ BC \rM \ar[rr]_{\theta
M} & &
CB M \ar[u]_-{C \beta} }
\end{array}
\end{equation}
commutes.
\end{defn}

As before, we have a version for functors:
\begin{defn}
A functor $M \colon \cX \to \cA$ has the structure of a \emph{$\theta$-entwined algebra} if $M$ is a $[\cX, \theta]$-entwined algebra.
\end{defn}

Once again by choosing $\cX = \mathbbm{1}$ we see that these two definitions are equivalent. Without loss of generality,
we consider entwined
algebras as objects in $\mathcal A$.
With the obvious notion of morphism (given by natural
transformations compatible with $\nabla$ and $ \beta
$), entwined algebras form a category; this is
evidently isomorphic to the category
$(\cA^B)^{C^\theta}$ of
$C^\theta$-coalgebras in $\cA^B$.
Dually we define an entwined opalgebra structure
on a
functor $N \colon \cA \rightarrow \cY$
for a distributive law $CB \Rightarrow BC$.

Suppose again that we are in the comonad setting of Section~\ref{mainapp}, so we have an adjunction $F\dashv U$ and a square
$$
\xymatrix{\cB \ar[r]^ U  \ar[d]_ S & \ar@{}[dl]^(.25){}="a"^(.75){}="b" \ar@{=>}_-{\Omega} "a";"b"   \cA \ar[d]^ C\\
	\cB \ar[r]_U  & \cA}
$$
where $(U,\Omega)$ is an iso-opmorphism of comonads, giving rise to distributive laws
$$\chi \colon TS \Rightarrow ST, \qquad \theta \colon BC \Rightarrow CB$$ via Corollary~\ref{arisec}.
The following proposition explains the relation
between $\theta$-entwined algebras and
$\chi$-coalgebras:

\begin{prop}
Let $M \colon \cX \rightarrow \cB$
be a functor and let $\nabla \colon M \Rightarrow
SM$ be a natural transformation.
\begin{enumerate}
\item If $\nabla$ is an
$S$-coalgebra structure,
then the structure morphisms
$$
\xymatrix{ BUM = UFUM
\ar[r]^-{U \epsilon M} & UM}, \qquad \xymatrix{
UM \ar[r]^-{U \nabla} & USM
\ar[r]^-{\Omega^{-1}} &CUM}
$$
turn $UM$ into an entwined
algebra with respect to $\theta$.
\item If $\cB =
\cA^B$, then the converse of (1) holds.
\end{enumerate}
\end{prop}
\begin{proof}
By the remarks in Section~\ref{comparisonsection}, the comparison functor $U^{U\epsilon} \colon \cB \to \cA^B$ is part of a morphism of comonads
$$
\xymatrix{
(\cB, S) \ar[rr]^-{(U^{U\epsilon}, \tilde\Omega^{-1})} &&(\cA^B, C^\theta)
}
$$ thus inducing a lifting
$$
\xymatrix{
\cB^S \ar[d] \ar[rr]^-{(U^{U\epsilon})^{\tilde \Omega^{-1}}} && (\cA^B )^{C^\theta} \ar[d]\\
\cB \ar[rr]_-{U^{U\epsilon}} && \cA^B
}
$$
where the undecorated arrows denote the appropriate forgetful functors.
The object map of the functor in the top row of this
diagram is the construction in part (1). If $\cB =
\cA^B$, then $U^{U \epsilon} = 1$, $\Omega =
1$ and $S = C^\theta$, so the top functor in the diagram is just the identity, implying part (2).
\end{proof}

Dually, entwined opalgebra structures on a $B$-opalgebra $(N, \cY , \omega)$ are related to
 $\chi$-opcoalgebras if
the codomain $\cY$ of $N$ is a
category with coequalisers. First, we define a functor
$N_B \colon \cA^B \rightarrow \cY$
that takes a $B$-algebra morphism $f
\colon (X, \alpha) \rightarrow (Y, \beta)$ to
$N_B(f)$ defined using coequalisers:
$$
\xymatrix@C=5em{NB X \ar[d]_-{N B f}
\ar@<+.5ex>[r]^-{\omega_X} \ar@<-.5ex>[r]_-{N \alpha} &
\ar[d]_-{N f} N X \ar@{->>}[r]^-{q_{(X,\alpha)}} &
N_B (X, \alpha) \ar@{.>}[d]^-{N_B(f)} \\ NB
Y  \ar@<+.5ex>[r]^-{\omega_Y}
\ar@<-.5ex>[r]_-{N \beta} & N Y \ar@{->>}[r]_-{q_{(Y,
\beta)}} & NB(Y, \beta) }
$$
Thus $ N_B$ generalises the functor
${-}\otimes_B N$ defined by a left module $N$
over a ring $B$ on the category of right $B$-modules.

Suppose that
$\theta$ is invertible, and that $N$ admits the
structure of an entwined $\theta^{-1}$-opalgebra, with
coalgebra structure $\nabla \colon N \Rightarrow
CN$. There
are two commutative diagrams:
$$
	\xymatrix{ NB
	X\ar[d]_-{\nabla_{ B X}} \ar[rr]^-{\omega_X} &&
	 N X
	\ar[dd]^-{\nabla_X} \\  N C B X
	\ar[d]_-{ N \theta^{-1}_X} & & \\  N B C X
	\ar[rr]_-{\omega_{ C X}} & &  N C X
	}\quad\quad\quad
	\xymatrix{  N B X
	\ar[d]_-{\nabla_{ B X}}
	\ar[rr]^-{ N\alpha} & &
	 N X \ar[dd]^-{\nabla_X} \\
	 N C B X
	\ar[d]_-{ N \theta^{-1}_X} & & \\
	 N B C X
	\ar[r]_-{ N \theta_X} &
	 N C B X
	\ar[r]_-{ N C \alpha} &
	 N C X }
$$
Hence, using coequalisers,
$\nabla$ extends to a natural transformation $\tilde
\nabla \colon  N_B  \Rightarrow  N_B
 C^\theta$,  and in fact it gives $ N_B$ the structure
of a $C^\theta$-opcoalgebra. Since
$\tilde{\theta}^{-1} \colon C^\theta
\tilde{B}  \Rightarrow \tilde{B}C^\theta$
is a comonad distributive law on
$\cA^B$, Proposition~\ref{triv} gives examples of
homologically trivial
$ \tilde \theta^{-1}$-opcoalgebras
of the form
$(N_B,
\cY, \tilde{\nabla} \epsilon)$.

\section{Examples in \texorpdfstring{$\set$}{Set}}\label{examplesect}
Our main examples of distributive laws and their coalgebras are homological in flavour, and come later in Chapter~\ref{EXAMPLES}. Here, however, we give some simple concrete examples in $\set$ to illustrate the concepts of the current chapter.
\subsection{Monads and comonads}

\begin{exa}\label{finordset}
Let $C$ be a set. This has a unique coalgebra structure with respect to $\set$ with the monoidal product given by the Cartesian product $\times$. Thus, the functor $C \times{-}$ becomes a comonad on $\set$ with counit and comultiplication given by
\begin{align*}
&\xymatrix@R=1em{C \times X \ar[r]^-\epsilon & X \\ (c,x) \ar@{|->}[r] & x} & &\xymatrix@R=1em{C \times X \ar[r]^-{\delta} & C \times C \times X \\ (c, x) \ar@{|->}[r] & (c, c, x)}
\end{align*}
\end{exa}

\begin{exa}
Assigning to a set $X$ its powerset $P(X)$ forms a monad $P$ on $\set$, where the unit and multiplication are given by
\begin{align*}
&\xymatrix@R=1em{X \ar[r]^-\eta & P(X) \\ x \ar@{|->}[r] & \{x\}} & &\xymatrix@R=1em{P(P(X)) \ar[r]^-{\mu} & P(X) \\
A \ar@{|->}[r] & \bigcup A
}
\end{align*}
\end{exa}

\begin{exa}\label{listmonad}
The functor $L \colon \set \to \set$, defined on objects by the disjoint union
$$
L(X) := \bigsqcup_{n \ge 0} X^n
$$
is part of a monad. We view $L(X)$ as the set of lists (including the empty list) of elements in $X$, and we denote these lists with square brackets as is common in computer science, for example $[x_1, \ldots, x_n]$. The monad structure is given by
\begin{align*}
&\xymatrix@R=1em{X \ar[r]^-\eta & L(X) \\ x \ar@{|->}[r] & [x]} & &\xymatrix@R=1em{L(L(X)) \ar[r]^-{\mu} & L(X) \\  \ar@{|->}[r] [[x_{1,1}, \ldots, x_{1, n_1}],  \ldots] &[x_{1,1}, \ldots, x_{1, n_1}, \ldots] }
\end{align*}
Here $\mu$ is concatenation, i.e.\ removal of the inner brackets.
\end{exa}
\begin{exa}\label{nonemptylistmonad}
The non-empty list functor $L^+ \colon \set \to \set$, given by
$$
L^+(X) := \bigsqcup_{n > 0} X^n
$$
becomes a monad with the same unit and multiplication of the list monad of Example~\ref{listmonad}. However, it also a comonad, with structure given by
\begin{align*}
&\xymatrix@R=1em{L^+(X) \ar[r]^-\epsilon & X \\ [x_1, \ldots, x_n] \ar@{|->}[r] & x_1} & &\xymatrix@R=1em{L^+(X) \ar[r]^-{\delta} & L^+(L^+(X)) \\ [x_1, \ldots, x_n] \ar@{|->}[r] & [ [x_1, \ldots, x_n], [x_2, \ldots, x_n], \ldots, [x_n]]}
\end{align*}
Explicitly, $\delta$ takes a list and removes the first element iteratively until the list is empty, while at each stage storing the result in the output list. This procedure is known as giving the \emph{tails} of the list.
\end{exa}
\begin{exa}\label{monoidcomonad}
Let $M$ be a monoid (with identity element $e$). The functor $S \colon \set \to \set$ given by $S(X) = X^M$ on objects (this is the set of functions from $M$ to $X$) becomes a comonad with structure given by
\begin{align*}
&\xymatrix@R=1em{X^M \ar[r]^-\epsilon & X \\ f \ar@{|->}[r] & f(e)} & &\xymatrix@R=1em{X^M \ar[r]^-{\delta} & (X^M)^M \cong X^{M \times M}\\ f \ar@{|->}[r] & \left((m,n) \mapsto f(mn)\right)}
\end{align*}
\end{exa}

\begin{exa}\label{filter}
Recall that a \emph{filter} $\sF$ on a set $X$ is a non-empty subset of $P(X)$ such that:
\begin{enumerate}
\item The set $X$ is in $\sF$, and $A,B \in \sF \implies A \cap B \in \sF$.
\item The empty set $\emptyset$ is not in $\sF$.
\item If $B\subseteq X$ and there exists $A \in \sF$ with $A \subseteq B$, then $B \in \sF.$
\end{enumerate}
An \emph{ultrafilter} $\sU$ is a filter such that $A \cup B \in \sU \implies A \in \sU$ or $B \in \sU$.
\end{exa}

There is a functor $F \colon \set \to \set$ which assigns to $X$ the set $FX$ of filters on $X$. A function $f \colon X \to Y$ is mapped to the function $Ff \colon FX \to FY$, defined by
$$
F(f)(\sF) = \{ A \subseteq Y \mid f^{-1}(A)  \in \sF\}.
$$
Given $\bG \in FFX$, we define the \emph{Kowalski sum} to be the filter
$$
\mu(\bG) = \{ A \subseteq X \mid \{\sF \in FX \mid A \in \sF \} \in \bG\},
$$
and given $x \in X$, we define the \emph{principal ultrafilter} on $x$ to be the (ultra)filter
$$
\eta(x) = \{ A \subseteq X \mid x \in A\}.
$$
These define natural transformations $\mu \colon FF \Rightarrow F$ and $\eta \colon 1 \Rightarrow F$ turning $F$ into a monad. By restricting to ultrafilters, we get an additional monad $U$ on $\set$. See~\cite[Chapter~II]{MR3307673} for more information on these monads.

\begin{exa}
Given a set $X$, let $DX$ be the set of (countable) probability distributions on $X$, that is,
$$
DX = \left\{ p \in [0,1]^X \ \Big|\  \left|p^{-1}(0, 1]\right| \le \aleph_0,\ \sum_{x \in X} p(x) = 1 \right\}.
$$
Given a function $f \colon X \to Y$, we get a function $Df \colon DX \to DY$ which maps a distribution $p \colon X \to [0,1]$ to the distribution $Y \to [0,1]$ defined by
$$
y \mapsto \displaystyle\sum_{x \in f^{-1}(y)} p(x)
$$
This defines a functor $D \colon \set \to \set$. Let $\eta \colon X \to DX$ be the function which maps $x$ to the characteristic function $\chi_{\{x\}} \colon X \to [0,1]$, and let $\mu \colon DDX \to DX$ be the function that sends a distribution $P \colon DX \to [0,1]$ to the distribution $\mu(P) \colon X \to [0,1]$ defined by
$$
\mu(P) (x) = \sum_{p \in DX} P(p) \cdot p(x).
$$
These define natural transformations $\eta, \mu$ that turn $D$ into a monad.

The $D$-algebras are \emph{superconvex spaces}, i.e.\ sets where one can take convex combinations of countably many elements, see~\cite{MR659884, MR1421176} for more information.
\end{exa}

\subsection{Distributive laws}
Let $C = \{c_0, \ldots, c_n\}$ be a non-empty, finite, totally-ordered set with $c_0 < \cdots < c_n$, and recall the comonad $C \times -$ of Example~\ref{finordset}.
\begin{exa}\label{suppreserving}
For any set $X$, define $\theta \colon P(C \times X) \to C \times PX$ by
$$
\theta (A \subseteq C \times X) = \begin{cases} (\sup \pi_1(A), \pi_2(A) ) & \mbox{ if } A \neq \emptyset \\
(c_0, \emptyset) & \mbox{ otherwise }
\end{cases}
$$
This defines a mixed distributive law $\theta \colon P(C \times{-}) \Rightarrow C \times P{-}$.
\end{exa}
\begin{exa}\label{supsagain}
There is a mixed distributive law $\theta \colon L(C \times {-}) \Rightarrow C \times L{-}$ defined by
$$
\theta [ (c_1, x_1), \ldots, (c_m, x_m) ] = \begin{cases}
(\sup_i c_i, [x_1, \ldots, x_m]) & \mbox{ if } m >0 \\
(c_0, [~]) & \mbox{ if } m = 0
\end{cases}
$$
where $[~]$ denotes the empty list.
\end{exa}

\begin{exa}
Let $\theta \colon D(C \times X) \to C \times DX$ be the map into the product, defined by the two maps
 \begin{align*}
&\xymatrix@R=1em{D(C \times X) \ar[r]^-{\theta_1} & C \\ p  \ar@{|->}[r] &\displaystyle\sup_{\sum_{x \in X} p(c,x) \neq 0} c} & &\xymatrix@R=1em{D(C\times X) \ar[r]^-{\theta_2} & DX\\ p \ar@{|->}[r] & \left( x \mapsto \displaystyle\sum_{c \in C} p(c,x)\right)}
\end{align*}
This defines a distributive law $\theta \colon D(C \times{-}) \Rightarrow C \times D{-}$.

\end{exa}
In Examples~\ref{11}--\ref{15}, let $M$ be a monoid, giving rise to the comonad $S$ of Example~\ref{monoidcomonad} which maps a set $X$ to $X^M$.
\begin{exa}\label{11}
The maps $\theta \colon L(X^M) \to (LX)^M$  given by
$$
\theta[f_1, \ldots, f_n](m) = [f_1(m), \ldots, f_n(m)]
$$
define a mixed distributive law $\theta \colon LS \Rightarrow SL$.
\end{exa}
\begin{exa}\label{13}
We define maps $\theta \colon P(X^M) \to (PX)^M$ as follows. For $A \in P(X^M)$ we let
$$
\theta(A ) (m) = \{ f(m) \mid f \in A\}.
$$
This defines a mixed distributive law $\theta \colon PS \Rightarrow SP$.
\end{exa}
\begin{exa}\label{14} Recall the filter monad $F$ of Example~\ref{filter}.
There is a mixed distributive law $\theta \colon FS \Rightarrow SF$, defined as follows. For each $m \in M$, let $\ev^m \colon X^M \to X$ denote the natural map $f \mapsto f(m)$. Then $\theta \colon F(X^M) \to (FX)^M$ is defined by
$$
\theta(\sF)(m) = F(\ev^m)(\sF).
$$
In a similar way, we get a mixed distributive law $US \Rightarrow SU$ by restricting to ultrafilters.
\end{exa}

\begin{exa}\label{15}
Define maps $\theta \colon D(X^M) \to (DX)^M$ as follows. For a distribution $p \colon X^M \to [0,1]$, $\theta(p)$ should be a function $M \to DX$, and so $\theta(p)(m)$ should be a distribution $X \to [0,1]$. We define this by
$$
\theta(p)(m)(x) = \sum_{\substack{f \in X^M \\ x = f(m)}} p(f).
$$
This defines a mixed distributive law $\theta \colon DS \Rightarrow SD$.

\end{exa}
\subsection{$\tilde\theta$-coalgebras and entwined algebras }
Recall that the mixed distributive laws $\theta$ in the preceding section lift to comonad distributive laws $\tilde\theta$ on the Eilenberg-Moore category of the relevant monad, cf.\ Section~\ref{extremalcase}.
\begin{exa}
Consider the powerset monad $P$. A $P$-algebra $(X, \beta)$ is the same as a partially ordered set such that every subset has a supremum. Indeed, the partial order is defined by
$$
x \le y \iff \beta\{x, y\} = y
$$
and $\beta \colon PX \to X$ points out the supremum of a given subset. Now consider the distributive law $\theta \colon P(C \times{-}) \Rightarrow C \times P{-}$ of Example~\ref{suppreserving}. By Proposition~\ref{chicoalgprop}, $\tilde\theta$-coalgebras in $\set^P$ correspond to $P$-algebras $(X, \beta)$ whose underlying set $X$ admits the structure of a $(C\times{-})$-coalgebra, which means that $X$ comes equipped with a partition of subsets indexed by the `colours' in $C$. We denote this partition by a colour function $\kappa \colon X \to C$. The $\theta$-entwined algebras are $P$-algebras of this form, such that $\kappa$ preserves suprema, i.e.\
$$
\kappa(\sup A) = \sup_{a \in A} \kappa(a)
$$
for all $A \subseteq X$.

\end{exa}

\begin{exa}\label{7}
Consider the list monad $L$. An $L$-algebra $(X, \beta)$ is the same thing as a monoid structure on $X$. The identity element $e$ is given by $\beta([~])$, and the multiplication is defined by $xy := \beta[x,y]$.
Let $\theta \colon L(C\times{-}) \Rightarrow C \times L{-}$ be the distributive law of Example~\ref{supsagain}.
By Proposition~\ref{chicoalgprop}, $\tilde\theta$-coalgebras in $\set^L$ correspond to monoids $X$ whose underlying set admits a partition $\kappa \colon X \to C$. The
$\theta$-entwined algebras are those monoids of this form such that the colour of a product in $X$ is the supremum of the colours in the product, that is
$$
\kappa(x_1 \cdots x_n ) = \sup_i \kappa (x_i).
$$
Note that the empty list must be included, so in particular the colour of the identity is minimal, that is
$$
\kappa(e) = c_0.
$$
The set $C$ becomes a monoid with unit $c_0$ and multipication given by $\sup$, and so $\kappa$ is a monoid map.
\end{exa}
\begin{exa}
Let $\theta \colon LS \Rightarrow SL $ be the distributive law of Example~\ref{11}. Let $X$ be a $L$-algebra (a monoid). By Proposition~\ref{chicoalgprop}, a $\tilde\theta$-coalgebra structure on $X$ is equivalent to a map $\nabla \colon X \to X^M$, or equivalently a map $\nabla_m \colon X \to X$ for each $m \in M$, such that
$\nabla_e = 1 $ and $\nabla_{mn} = \nabla_n \nabla_m$
for all $m,n \in M$. In other words, $\nabla$ defines a monoid map $M \to \set(X,X)^*$. The $\theta$-entwined algebras are those monoids $X$ equipped with such a map $\nabla$ whereby the image of $\nabla$ lies in the monoid endomorphisms of $X$.
\end{exa}
\begin{exa}
Let $\theta \colon PS \Rightarrow SP$ be the distributive law of Example~\ref{13}. By Proposition~\ref{chicoalgprop}, a $\tilde\theta$-coalgebra structure on a $P$-algebra $(X, \beta)$ is equivalent to a map $\nabla \colon X \to X^M$ as in Example~\ref{7}. The set $X^M$ can be given a partial order, defined by $f \le g$ if and only if $f(m) \le g(m)$ for all $m \in M$. This poset is itself a $P$-algebra, since given a subset $H \subseteq X^M$ we can define a function $\sup H \colon M \to X$ given by
$$
\sup H (m) = \sup \{ f(m) \mid f \in H \}
$$
which genuinely is the supremum of $H$. Thus the $\theta$-entwined algebra structures on $P$-algebras $(X, \beta)$ correspond to coassociative coactions $\nabla \colon X \to X^M$ which are sup-preserving.
\end{exa}

%
%

\chapter{Duplicial objects and cyclic homology}\label{CYCLIC}
We begin this chapter by recalling the definitions of simplicial and duplicial objects as well as cyclic homology in Section~\ref{simpmeth}. Following this,
in Section~\ref{cychomol} we recount the method of constructing duplicial objects from~\cite{MR2415479}. We emphasize
the self-duality of the situation by defining in fact two duplicial objects $C_T(N,M)$ and $C^*_S(N,M)$, arising from bar resolutions using comonads $T,S$, as well as suitable functors $M,N$. There is a canonical
pair of morphisms of duplicial objects between these which are mutual inverses if and
only if the two objects are cyclic (Theorem~\ref{cyc}).
In Section~\ref{duplobjsec}, we further develop the process of twisting a pair of coefficient objects
$M, N$ detailed previously in Section~\ref{twistcoeff}. We refrain from giving concrete examples in this chapter, and instead postpone this until Chapter~\ref{EXAMPLES}.

Apart from the basic definitions and cited material, the work in Sections~\ref{cychomol} and~\ref{duplobjsec} is original. The latter section is based on the work of~\cite[\S4.5--4.11]{1} but goes into a bit more detail.

\section{Simplicial methods}\label{simpmeth}
Here we go over the very basics of simplicial homotopy theory, see e.g.~\cite{MR2840650,MR1950475} for a thorough account. We also discuss the duplicial objects of Dwyer and Kan~\cite{MR826872}.
\subsection{The simplicial and augmented simplicial categories}

\begin{defn}\label{simpldef}
The \emph{simplicial category}, or the \emph{topologists' simplicial category}, is the category $\Delta$ whose objects are finite non-zero ordinals
$$
\mathbf n = \{0, 1, \dots, n\}
$$
and whose morphisms are the (weakly) order-preserving maps between these. The \emph{augmented simplicial category}, or the \emph{algebraists' simplicial category} is the category $\Delta_+$ whose objects are the same as $\Delta$ plus an additional object
$$
\mathbf{-1} = \emptyset
$$
and whose morphisms are again the (weakly) order-preserving maps.
\end{defn}
It is straightforward to show that the morphisms of $\Delta, \Delta_+$ are generated by functions
$$
\xymatrix{
\mathbf{n-1} \ar[r]^-{\delta_i} & \mathbf n
},\qquad
\xymatrix{
\mathbf{n + 1} \ar[r]^-{\sigma_i} & \mathbf n
}
$$
for all suitable $n$ and for all $0\le i \le n$, called \emph{face maps} and \emph{degeneracy maps} respectively,  such that the \emph{simplicial identities} are satisfied:
$$
\begin{array}{cc}
     \delta_j  \delta_i
     =
     \delta_i\delta_{j-1}
     &
    \mbox{ if } i < j
     \\
     \sigma_j  \sigma_i
     =
     \sigma_i  \sigma_{j+1}
     &
     \mbox{ if } i \le j
\end{array}
\qquad\qquad
\sigma_j \delta_i = \begin{cases}
\delta_i \sigma_{j-1} & \mbox{ if } i < j \\
1 &\mbox{ if } i = j, j+1 \\
\delta_{i-1} \sigma_j & \mbox{ if } i > j + 1
\end{cases}
$$

\subsection{The duplicial and cyclic categories}
\begin{defn}
The \emph{duplicial} \emph{category} $\Lambda_\infty$ is constructed as follows: we take the category $\Delta$ and adjoin in each degree an abstract morphism $\tau \colon \mathbf n \to \mathbf n$ subject to the relations:$$
\tau \delta_i =
\begin{cases}
\delta_{i-1} \tau & \mbox{ if } 1 \le i \le n \\
\delta_n & \mbox{ if } i = 0
\end{cases}
\qquad
\tau \sigma_j =
\begin{cases}
\sigma_{j-1} \tau & \mbox{ if } 1 \le j \le n \\
\sigma_n \tau^2 & \mbox{ if } j = 0
\end{cases}
$$
We obtain the \emph{cyclic category} $\Lambda$ by adding the extra requirement that in each degree, the morphism $\tau \colon \mathbf n \to \mathbf n$ satisfies the relation $\tau^{n+1} = 1$.

Similarly, we have augmented versions of these categories by considering $\Delta_+$ rather than $\Delta$ in the above definition.

\subsection{Simplicial, duplicial and cyclic objects}
Let $\cY$ be a category.
\begin{defn}
A \emph{simplicial object} in $\cY$ is an object of the functor category $[\Delta^*, \cY]$, so a contravariant functor from $\Delta$ to $\cY$.
\end{defn}

Thus to give a simplicial object $X$ in $\cY$ is to give a sequence $X_0, X_1, \ldots$ of objects in $\cY$
and in each degree $n \ge 0$ morphisms
$$
\xymatrix{
X_n \ar[r]^-{d_i} & X_{n-1}
},\qquad
\xymatrix{
X_n \ar[r]^-{s_i} &  X_{n+1}
}
$$
for all $0\le i \le n$ subject to the conditions:$$
\begin{array}{cc}
     d_i d_j
     =
      d_{j-1}d_i
     &
    \mbox{ if } i < j
     \\
     s_i s_j
     =
     s_{j+1} s_i
     &
     \mbox{ if } i \le j
\end{array}
\qquad\qquad
d_i s_j = \begin{cases}
 s_{j-1} d_i & \mbox{ if } i < j \\
1 &\mbox{ if } i = j, j+1 \\
s_j d_{i-1} & \mbox{ if } i > j + 1
\end{cases}
$$

Similarly, we define \emph{augmented simplicial objects} as functors $\Delta_+^* \to \cY$.

\begin{rem}\label{opsimpobj}
Given a simplicial object $X$, by reversing the order of both the face maps and degeneracy maps we obtain a new simplicial object. We call this the \emph{opsimplicial simplicial object associated to $X$}.
\end{rem}

\begin{defn}[see~\cite{MR826872}]
A \emph{duplicial object} in $\cY$ is an object of the functor category $[\Lambda_\infty^*, \cY]$, so a contravariant functor from $\Lambda_\infty$ to $\cY$.
\end{defn}
Thus, explicitly, a duplicial object $X$ in $\cY$ is a simplicial object $X$ together with a morphism, or \emph{duplicial operator},
$$
\xymatrix{
X_n \ar[r]^-{t} & X_n
}
$$
in each degree subject to the conditions:
$$
d_i t=
\begin{cases}
t d_{i-1} & \mbox{ if } 1 \le i \le n \\
d_n & \mbox{ if } i = 0
\end{cases}
\qquad
 s_j t =
\begin{cases}
 t s_{j-1} & \mbox{ if } 1 \le j \le n \\
  t^2 s_n & \mbox{ if } j = 0
\end{cases}
$$
Similarly again, we have the notion of \emph{cyclic objects}~\cite{MR777584}, which are functors $\Lambda^* \to \cY$, as well as augmented versions.

\begin{rem}
Some authors use term \emph{paracyclic} to mean duplicial in our sense. Beware of this terminology, since it has also been used to describe those duplicial objects where the operator $t$ or $T$ (see Section~\ref{tdef}) is an isomorphism. Furthermore,
the original sense of a duplicial structure~\cite{MR826872} on a simplicial object $X$ is precisely a duplicial structure, in our sense, on the opsimplicial simplicial object associated to $X$, cf.~Remark~\ref{opsimpobj} and~Section~\ref{duplobjsec}.
\end{rem}

\subsection{The Dold-Kan and Dwyer-Kan correspondences}
Now suppose that $\cY$ is an abelian category. Let $\C(\cY)$ denote the category of chain complexes in $\cY$, and let
$\C^+(\cY)$ denote the full subcategory of $\C(\cY)$ consisting of non-negatively graded chain complexes.

\begin{defn}[see~\cite{MR826872,MR883882}]
A \emph{duchain complex} in $\cY$ is a  triple $(X, b, B)$ where $(X, b)$ is a chain complex and $(X, B)$ is a cochain complex:
$$
\xymatrix{
\cdots \ar@<1ex>[r]^-b & \ar@<1ex>[l]^-B X_n \ar@<1ex>[r]^-b & \ar@<1ex>[l]^-B X_{n-1} \ar@<1ex>[r]^-b & \ar@<1ex>[l]^-B \cdots
}
$$
A \emph{mixed complex} is a duchain complex $(X, b, B)$ such that $bB + Bb = 0$.
\end{defn}

\begin{rem}
Since both chain and cochain complexes are involved in the definition of a duchain complex, we point out that whenever we use the terminology \emph{quasi-isomorphism}, we mean a morphism which induces an isomorphism on homology, but not necessarily on cohomology.
\end{rem}

Let $\D(\cY)$ denote the category of duchain complexes in $\cY$, where a morphism is defined to be one which is simultaneously a morphism of chain complexes and of cochain complexes. Let $\M(\cY)$ denote the full subcategory of $\D(\cY)$ given by mixed complexes. Let $\M^+(\cY) \subseteq \D^+(\cY)$ denote the non-negatively graded versions.

The following theorem is known as the \emph{Dold-Kan} correspondence. We sketch the proof, but for a full proof see~\cite[Thm.~8.4.1]{MR1269324} and its following remarks, or see~\cite[Chapter~III.2]{MR2840650}.
\begin{thm}\label{doldkan}
There is an equivalence of categories
$$
[\Delta^*, \cY] \simeq \C^+(\cY).
$$
\end{thm}
\begin{proof}[Sketch proof]
The equivalence $ N \colon [\Delta^*, \cY] \to \C^+(\cY)$ is defined on objects as follows. A simplicial object $X$ is mapped to the chain complex $NX$ given in degree $n$ by
$$
NX_n = \quotient{X_n}{ \sum_{i=0}^{n-1} \im s_i }
$$
whose differential $NX_{n} \to NX_{n-1}$ is defined by the diagram
$$
\xymatrix{
X_n \ar@{->>}[d] \ar[rr]^-{\sum_{i=0}^{n} (-1)^i d_i} && X_{n-1} \ar@{->>}[d] \\
NX_n \ar@{.>}[rr] && NX_{n-1}
}
$$
\end{proof}
The following theorem is known as the \emph{Dwyer-Kan correspondence}, proved in~\cite{MR826872}. We do not make explicit use of it in later sections, but nevertheless we include it here as we feel it highlights the importance of duplicial objects. We give only a sketch proof.
\begin{thm}\label{dwyerkan}
There is an equivalence of categories
$$
[\Lambda_\infty^*, \cY] \simeq \D^+(\cY).
$$
\end{thm}
\begin{proof}[Sketch proof]
The equivalence $ [\Lambda_\infty^*, \cY] \to \D^+(\cY) $ is defined on objects as follows. Given a duplicial object $X$, we define a codifferential on the complex $NX$, where $N$ is the equivalence of Theorem~\ref{doldkan}. The codifferential is defined by the diagram
$$
\xymatrix{
X_n \ar@{->>}[d] \ar[rr]^-{(-1)^n t s_n} && X_{n+1} \ar@{->>}[d] \\
NX_n \ar@{.>}[rr] && NX_{n+1}
}
$$
in degree $n$.
\end{proof}

\section{Cyclic homology}\label{cychomol}
In this section, we define the cyclic homology of various objects and show that the notions coincide in appropriate cases. Throughout, $\cY$ denotes an abelian category.
\subsection{Cyclic homology of a mixed complex}
Given a mixed complex $(X,b,B)$ in $\M^+(\cY)$, we construct a bicomplex in the first quadrant:
$$
\xymatrix{
\vdots \ar[d]_-b & \vdots \ar[d]_-b  & \vdots \ar[d]_-b  \\
X_2\ar[d]_-b&\ar[l]^-B  X_1\ar[d]_-b &\ar[l]^-B  X_0 \\
X_1 \ar[d]_-b&  \ar[l]^-B X_0  \\
X_0
}
$$
This just means that each square anticommutes, so there is an associated total complex $T(X)$ given in degree $n$ by
$$
T(X)_n = \bigoplus_{i \le 0} X_{n + 2i}
$$
(so the first few degrees are $X_0, X_1, X_0 \oplus X_2, X_1 \oplus X_3, X_0 \oplus X_2 \oplus X_4, \ldots$ and one can spot the pattern) with differential $b + B$.

The notion of cyclic homology is due to Connes~\cite{MR823176,MR777584} but the following definition is from~\cite{MR883882}.
\begin{defn}\label{cyclichomology}
The \emph{cyclic homology} of the mixed complex $X$, denoted $\HC(X)$, is the homology of the total complex $T(X)$, i.e.\
$$
\HC(X) := \HH(T(X))
$$
\end{defn}

Thus we have defined a functor
$$
\HC \colon \M^+(\cY) \lto \cY.
$$
Let $U \colon \M^+(\cY) \lto \C^+(\cY)$ denote the obvious forgetful functor to the category of non-negatively graded chain complexes.

The following proposition is a useful one, and we refer the reader to~\cite[Proposition~2.5.15]{MR1600246} for a proof:
\begin{prop}\label{quasi}
Let $f$ be a morphism in $\M^+(\cY)$. Then $Uf$ is a quasi-isomorphism if and only if $\HC(f)$ is an isomorphism.
\end{prop}
\subsection{Cyclic homology of a cyclic object}\label{cychomcycobj}
Let $X \colon \Lambda^{*} \to \mathcal Y$ be a cyclic object. Consider the following morphisms, defined in degree $n$:
$$
\tilde t := (-1)^n t \quad \quad \quad s_{-1} = ts_n \quad \quad \quad \mathcal N_n = \sum_{i=0}^n \tilde t^i \quad \quad \quad b_n = \sum_{i=0}^n (-1)^i d_i
$$
By Theorem~\ref{doldkan}, $(X,b)$ is a chain complex. In fact, the morphism $\hat B = s_{-1} \mathcal N$ has the property that $b \hat B + \hat B b = 0$, but unfortunately, we do not have that $\hat B^2 = 0$ and so it is not a codifferential. However, the morphism $B := (1-\tilde t)s_{-1}\mathcal N$ does square to zero, and also anticommutes with $b$. Thus, $(X,b,B)$ is a mixed complex. We have defined a functor
$$
\xymatrix{
[\Lambda^*, \cY] \ar[r]^-K & \M^+(\cY)
}
$$
and we define the cyclic homology of a cyclic object to be the composite of this functor with $\HC$ of Definition~\ref{cyclichomology}. By abuse of notation, we denote this also by $\HC$, and so we have a commutative diagram
$$
\xymatrix{
[\Lambda^*, \cY] \ar[r]^-K \ar[dr]_-{\HC} & \M^+(\cY) \ar[d]^-{\HC} \\
& \cY
}
$$

Alternatively, one could do the following: Let $NX$ be the complex associated to the underlying simplicial object of the cyclic object $X$, i.e.\
$$
( N X)_n = \quotient{X_n}{ \sum_{i=0}^{n-1} \operatorname{im} s_i}
$$
as in the proof of Theorem~\ref{doldkan}. Both the morphisms $b$ and $B$ descends to this quotient, defining a functor
$$
\xymatrix{
[\Lambda^*, \cY] \ar[r]^-{K'} & \M^+(\cY).
}
$$ It is well-known that the quotient morphism $q \colon (X,b) \to (NX, b)$ is a quasi-isomorphism (see e.g.~\cite[Thm.~8.3.8]{MR1269324}), and it descends to the quotient too. Therefore, $q$ is actually a morphism $KX \to K'X$ of mixed complexes, defining a natural transformation
$$q \colon K \Rightarrow K'.$$ By Proposition~\ref{quasi}, this gives an isomorphism after composing with $\HC$. This proves the following:
\begin{prop}
We have a square
$$
\xymatrix{
[\Lambda^*, \cY] \ar[rr]^-K \ar[dd]_-{K'} && \M^+(\cY) \ar@{}[ddll]^(.25){}="a"^(.75){}="b" \ar@{=>}_-{\HC q}^-\cong "a";"b" \ar[dd]^-{\HC} \\
\\
\M^+(\cY) \ar[rr]_-{\HC} && \cY
}
$$
in $\cat$.
\end{prop}

In summary, we have two equivalent ways of computing the cyclic homology of a cyclic object: either define the structure of a mixed complex on the actual cyclic object, or do so on the normalised complex of the underlying simplicial object.

\subsection{Cyclic homology of a duchain complex}\label{tdef}
Let $(X,b,B)$ be a duchain complex in $\D^+(\cY)$. Define an endomorphism
$$
T = 1 - bB -Bb
$$
in each degree. One easily checks that $b T = Tb$ and $BT = TB$, so clearly both $b$ and $B$ descend to the quotient $B/\im(1-T)$, and in that case $bB + Bb = 0$. This defines a functor $F \colon \D^+(\cY) \lto \M^+(\cY)$, which in fact is a left adjoint to the forgetful functor $U$.
$$
\xymatrix{
\D^+(\cY) \ar@/^1pc/[r]^-F  \ar@{}[r]|\perp& \M^+(\cY)  \ar@/^1pc/[l]^-{U}
}
$$
We define the cyclic homology of a duchain complex to be the composite of $\HC$ with this functor $F$. Note that $FU \cong 1$. Again we abuse notation to give a commutative diagram
$$
\xymatrix{
\D^+(\cY) \ar[dr]_-{\HC} \ar[r]^-F & \M^+(\cY) \ar[d]^-{\HC} \\
& \cY
}
$$  
\subsection{Cyclic homology of a duplicial object}
Let $X \colon \Lambda_\infty^* \to \cY$ be a duplicial object. In light of the above story for cyclic objects, there are two natural ways we could proceed to define cyclic homology of $X$: either construct a natural cyclic object from $X$ and take the cyclic homology of that, or instead, construct a duchain complex and take the cyclic homology of that. Here we show these two definitions are equivalent.
Let $\pi \colon \Lambda_\infty^* \lto \Lambda^*$ denote the quotient map. This induces a forgetful functor $\pi^* \colon [\Lambda^*, \cY]\lto [\Lambda_\infty^*, \cY]$. By the theory of Kan extensions, this functor has a left adjoint $\pi_!$.
$$
\xymatrix{
[\Lambda_\infty^*, \cY] \ar@/^1pc/[r]^-{\pi_!}  \ar@{}[r]|\perp& [\Lambda^*, \cY]  \ar@/^1pc/[l]^-{\pi^*}
}
$$
This left adjoint is explicitly defined by $(\pi_! X)_n = X_n / \im (1 - t^{n+1})$. The face maps and degeneracy maps descend to this quotient. Note that $\pi_! \pi^* \cong 1$. We define the cyclic homology of $X$ to be the cyclic homology of $\pi_! X$, giving a commutative diagram
$$
\xymatrix{
[\Lambda_\infty^*,\cY] \ar[r]^-{\pi_!} \ar[dr]_-{\HC} & [\Lambda^*, \cY] \ar[d]^-{\HC} \\
& \cY
}
$$

On the other hand, we can define morphisms $b,\tilde t, s_{-1}, \mathcal N, B, \hat B$ as before. Again, the morphism $\hat B$ does not square to zero, but neither does it satisfy the relation $b \hat B + \hat B b = 0$. In the duplicial case, the morphism $B$ does not satisfy these relations either. On the normalised complex $(NX, b)$ the morphism $B$ descends to the quotient (the proof in the cyclic case only uses the duplicial structure), is equal to $\hat B$, and satisfies $B^2 = 0$. This only gives us a duchain complex $(NX,b,B)$ since $bB + Bb$ is still not zero in the quotient. This defines a functor
$$
\xymatrix{
[\Lambda_\infty^*, \cY] \ar[r]^-{P} & \D^+(\cY)
}
$$
and it is important to note that this is \emph{not} the same as the Dwyer-Kan normalisation functor of Theorem~\ref{dwyerkan}.


In a similar way to Section~\ref{cychomcycobj}, the quotient morphism defines a map of mixed complexes $q \colon  K \pi_! X \to FPX $ defining a natural transformation
$$
\xymatrix{
[\Lambda_\infty^*, \cY] \ar[rr]^-{\pi_!} \ar[dd]_-{P} && [\Lambda^*, \cY] \ar@{}[ddll]^(.25){}="a"^(.75){}="b" \ar@{=>}_-{q}"a";"b" \ar[dd]^-{K} \\
\\
\D^+(\cY) \ar[rr]_-{F} && \M^+(\cY)
}
$$
which becomes an isomorphism after composing with $\HC \colon \M^+(\cY) \to \cY$, thus proving:
\begin{prop}
We have a square
$$
\xymatrix{
[\Lambda_\infty^*, \cY] \ar[rr]^-{\pi_!} \ar[dd]_-{P} && [\Lambda^*, \cY] \ar@{}[ddll]^(.25){}="a"^(.75){}="b" \ar@{=>}_-{\HC q}^-\cong "a";"b" \ar[dd]^-{\HC} \\
\\
\D^+(\cY) \ar[rr]_-{\HC} && \cY
}
$$
in $\cat$.
\end{prop}

\end{defn}
\section{From distributive laws to duplicial objects}\label{duplobjsec}
Here we develop B\"ohm and {\c S}tefan's work on duplicial objects~\cite{MR2956318,MR2415479}, by showing that some of the involved morphisms extend to simplicial morphisms. We use these morphisms to characterise the cyclicity of the relevant duplicial functor.
\subsection{The bar and opbar resolutions}
Let $\cB$ and $\cX$ be two categories.

\begin{defn}
A \emph{simplicial functor} $\cX \to \cB$ is a simplicial object in the  category $[\cX, \cB]$.
\end{defn}
Similarly, we have notions of \emph{dupicial} and \emph{cyclic functors}, as well as augmented versions of these.

 Let $(T, \delta, \epsilon)$ be a comonad on $\cB$ and let $M \colon \cX \rightarrow \cB$ be a
functor.
\begin{defn}\label{bardeffo}
The \emph{bar resolution of }$M$
is the simplicial functor
$ \rBB (T, M) \colon \cX \rightarrow
\cB$ defined by
$$
	\rBB (T, M)_n =
	T^{n+1}M, \qquad
	d_i =T^i \epsilon
	T^{n-i} M, \qquad
	s_j =T^j \delta T^{n-j}
	M,
$$
where the face and degeneracy maps
above are given in degree $n$. The \emph{opbar
resolution of }$M$,
denoted
$\rBB^*(T, M)$, is the simplicial
functor obtained by taking the op\-sim\-pli\-cial 
sim\-pli\-cial
functor (cf.~Remark~\ref{opsimpobj}) associated to $\rBB (T, M)$. Explicitly:
$$
 	\rBB^*(T, M)_n =
T^{n+1}M, \qquad d_i =T^{n-i} \epsilon T^{i} M,\qquad
	s_j =T^{n-j} \delta T^{j} M.
$$
\end{defn}
Given any functor $\rN \colon \cB \rightarrow
\cY$, we compose it with the above simplicial
functors to obtain new simplicial functors that we
denote by
$$
	\rCC_T(N,M):=N\rBB (T,M), \qquad
\rCC^*_T(N,M):=N\rBB ^*(T, M).
$$

\subsection{The B\"ohm-\c Stefan construction}
\label{evidenziatore1}
Let
$$\xymatrix{(\cB, T) \ar[rr]^-{(\Sigma, \sigma)} && (\cD, G)}$$ be an opmorphism of comonads, so in particular $\sigma $ is a natural transformation $G\Sigma \Rightarrow \Sigma T$. By abuse of notation, we let $\sigma^n$
denote the natural transformation
$G^n \Sigma \Rightarrow \Sigma T^n$ obtained by repeated application of
$\sigma$ (up to horizontal composition of identities),
where $\sigma^0 = 1$. In other words, we define $\sigma^n$ recursively: for all $n \ge 0$, $\sigma^{n+1} \colon G^{n+1}\Sigma \Rightarrow \Sigma T^{n+1}$ is given by the composite
$$
\xymatrix{
G^{n+1}\Sigma \ar[rr]^-{G\sigma^n} && G\Sigma T^n \ar[rr]^-{\sigma T^n} && \Sigma T^{n+1}
}.
$$

\begin{lem}\label{superchiepsilon}
For all $n \ge 0$ and for $0 \le i \le n$, the diagram
$$
\xymatrix{
G^{n+1}\Sigma \ar[rr]^-{\sigma^{n+1}} \ar[d]_-{G^{n-i}\epsilon T^i \Sigma} && \Sigma T^{n+1} \ar[d]^-{\Sigma T^{n-i}\epsilon T^i } \\
G^n\Sigma \ar[rr]_-{\sigma^n} && \Sigma T^n
}
$$
commutes.
\end{lem}
\begin{proof}
We proceed by induction. For $n  = 0$ and $i = 0$, the result holds since it becomes the compatibility condition of $\sigma$ with $\epsilon \colon T \Rightarrow 1$. Now, suppose the statement holds for $n - 1$, and let $i$ be such that $0 \le i < n$. Consider the diagram
$$
\xymatrix{
G^{n+1} \Sigma \ar[d]_-{G^{n-i} \epsilon G^i \Sigma } \ar[rr]^-{G\sigma^n} && G\Sigma T ^n\ar[d]_-{G\Sigma T^{n-i-1} \epsilon T^i} \ar[rr]^-{\sigma T^n } && \Sigma T^{n+1} \ar[d]^-{\Sigma T^{n-i} \epsilon T^i } \\
G^n \Sigma \ar[rr]_-{G\sigma^{n-1}} && G\Sigma T^{n-1} \ar[rr]_-{\sigma T^{n-1}} && \Sigma T^n
}
$$
The right-hand square commutes by naturality of $\sigma$, and the left-hand square commutes by the inductive hypothesis. For $i = n$, consider the diagram
$$
\xymatrix{
G^{n+1} \Sigma \ar[d]_-{\epsilon G^n \Sigma } \ar[rr]^-{G\sigma^n} && G\Sigma T^n \ar[drr]_-{\epsilon \Sigma T^n \ \ } \ar[rr]^-{\sigma T^n } && \Sigma T^{n+1} \ar[d]^-{\Sigma\epsilon T^n} \\
G^n \Sigma \ar[rrrr]_-{\sigma^n } && &&\Sigma T^n
}
$$
The right-hand triangle commutes by compatibility of $\sigma$ with $\epsilon$ and the left-hand quadrilateral commutes by naturality of $\epsilon$. In each case, the outer diagram commutes, as required.
\end{proof}
There is a similar statement for the coproducts $\delta$. Thus, we have proved:
\begin{prop}\label{supersigma}
The morphisms $\sigma^n$ define a morphism of simplicial functors
$$
\xymatrix{
 \rCC_G(1_{\mathcal D}, \Sigma)  \ar[r]^-{\sigma} & \rCC_T(\Sigma, 1_\mathcal B ).
}
$$
\end{prop}
Dually, given a morphism of comonads
$$
\xymatrix{
(\cB, S) \ar[rr]^-{(\Sigma, \gamma)} & & (\cD, C),
}
$$
the natural transformation $\gamma \colon \Sigma S \Rightarrow C\Sigma$ extends to a morphism of simplicial functors
$$
\xymatrix{
\rCC_S^*(\Sigma,1_\cB) \ar[r]^-\gamma & \rCC_C^*(1_\cD, \Sigma).
}
$$

Now consider a distributive law $\chi \colon TS \Rightarrow ST$ of comonads on $\cB$. This is equivalent to saying that
$$
\xymatrix{
(\cB, T) \ar[rr]^-{(S,~\chi)} & & (\cB, T),
}$$
is an opmorphism of comonads, and
$$
\xymatrix{
(\cB, S) \ar[rr]^-{(T,~\chi)} & & (\cB, S)
}
$$
is a morphism of comonads. Therefore, by Lemma~\ref{superchiepsilon} and its dual version, $\chi$ extends in two ways to simplicial functors
$$
\xymatrix{
\rCC_T(1_\cB, S) \ar[r]^-\chi & \rCC_T(S, 1_\cB),} \qquad
\xymatrix{\rCC_S^*(T, 1_\cB) \ar[r]^-\chi & \rCC_S^*(1_\cB, T).}
$$
Some abuse of notation occurs here, as we implicitly use the symbol $\chi^n$ to denote natural transformations $T^nS \Rightarrow ST^n$ and $TS^n \Rightarrow S^n T$.
\begin{rem}\label{superchimorphism}
A proof similar to that for Lemma~\ref{superchiepsilon} shows that for all $n \ge 0$,
$$
\xymatrix{
(\cB, T) \ar[rr]^-{(S^n,~\chi^n)} & & (\cB, T),
}$$
is an opmorphism of comonads, and
$$
\xymatrix{
(\cB, S) \ar[rr]^-{(T^n,~\chi^n)} & & (\cB, S)
}
$$ is a morphism of comonads.
\end{rem}

Now, let $(M, \cX, \rho)$ be a $\chi$-coalgebra, and let $(N, \cY, \lambda)$
be a $\chi$-opcoalgebra. We  define
natural transformations
$$
	t_T \colon
	\rCC_T(N,M)_n \Rightarrow \rCC_T(N,M)_n,
	\quad t_S \colon
	\rCC_S^*(N,M)_n \Rightarrow \rCC_S^*(N,M)_n
$$
by the diagrams
$$
	\xymatrix@C=3.5em{ NT^{n} SM
	\ar[r]^-{N  \chi^n M} &
	 N  S T^n  M
	\ar[d]^-{\lambda T^n  M}
	\\
	 N  T^{n+1} M
	\ar[u]^-{ N  T^n  \rho}
	\ar@{.>}[r]_-{t_T}  &  N  T^{n+1}  M }
	\quad\quad\quad
	\xymatrix@C=3.5em{  N T S^n M
	\ar[r]^-{ N \chi^n  M} &  N S^n  T  M
	\ar[d]^-{ N  S^n \rho}
	\\  N S^{n+1}  M \ar[u]^-{\lambda  S^n  M}
	\ar@{.>}[r]_-{t_S} & N S^{n+1}  M }
$$

\begin{thm}[B\"ohm and {\c S}tefan]\label{dup}
The simplicial functors
$
	 \rCC_T( N, M)
$
and
$
	\rCC^*_S( N,  M)
$
become duplicial functors with duplicial
operators given by $t_T$ and $t_S$ respectively.
\end{thm}
\begin{proof}
The first operator being duplicial is exactly the case considered
in \cite{MR2415479}, and the second follows from a slight
modification of their proof.  \end{proof}
At the end of Section~\ref{decalage} in the next chapter, we give a new way of deriving the duplicial functor $\rCC_T(N,M)$ of Theorem~\ref{dup}, using ideas from Hochschild (co)homology. Following this, we also show that the theorem can in fact be deduced directly from the twisting procedure given in Section~\ref{twistcoeff}.
\subsection{Cyclicity}
\label{evidenziatore2}
For each $n \ge 0$, we define a morphism
$\rho_n \colon T^{n+1} M \Rightarrow S^{n+1} M$
in the following way. For
each $0 \le i \le n$, let $\rho_{i,n}$ denote the morphism
$$
	\xymatrix{     S^i  T^{n-i + 1}  M \ar[rr]^-{     S^i
 T^{n-i} \rho} &&     S^i T^{n-i} S  M \ar[rr]^-{   S^i
\chi^{n-i}  M} & &     S^{i+1} T^{n-i}  M.}
$$
Then set $\rho_n$ to be the vertical composite
$$
	\rho_n := \rho_{n,n} \rho_{n-1,n}  \cdots \rho_{1,n} \rho_{0,n}.
$$
As above, the morphisms $\rho_{n+1} \colon T^{n+2}M \Rightarrow S^{n+2}M$ (for $n \ge 1$) can be defined recursively by either of the composites
\begin{align*}
\xymatrix{
T^{n+2}M \ar[rr]^-{T^{n+1} \rho} && T^{n+1} SM \ar[rr]^-{\chi^{n+1}M} && ST^{n+1}M \ar[rr]^-{S \rho_n} && S^{n+2}M
} \\
\xymatrix{
T^{n+2}M \ar[rr]^-{T\rho_n} && T S^{n+1} M \ar[rr]^-{\chi^{n+1}M} && S^{n+1}TM \ar[rr]^-{S^{n+1} \rho} && S^{n+2}M
}
\end{align*}
Completely symmetrically, we have morphisms $\lambda_{i,n} \colon NT^{n-i} S^{i+1} \Rightarrow NT^{n-i + 1} S^i$ given by
$$
\xymatrix{
N T^{n-i} S^{i+1} \ar[rr]^-{N \chi^{n-i} S^i} && NST^{n-i}S^i \ar[rr]^-{\lambda T^{n-i} S^i} && NT^{n- i + 1}
}
$$
which define a morphism $\lambda_n \colon NS^{n+1} \Rightarrow NT^{n+1}$ given by the composite
$$
\lambda_n := \lambda_{0,n} \lambda_{1,n} \cdots \lambda_{n-1, n} \lambda_{n,n}.
$$
\begin{lem}\label{superrhoepsilon}
For all $n \ge 0$ and for $0 \le i \le n + 1$, the diagram
\begin{equation}\label{superrho}
\begin{array}{cc}
\xymatrix{
T^{n+2}M \ar[rr]^-{\rho_{n+1}} \ar[d]_-{T^{n + 1 - i}\epsilon T^{i} M} && S^{n+2} M \ar[d]^-{S^{i}\epsilon S^{n+1-i }M} \\
T^{n+1}M \ar[rr]_-{\rho_n} && S^{n+1} M
}
\end{array}
\end{equation}
commutes.
\begin{proof}
Again, we proceed by induction. Suppose that $n = 0$ and $i = 1$. Naturality of $\epsilon$ and compatibility of $\rho$ and $\epsilon$ tells us that each inner square of the diagram
$$
\xymatrix{
TTM \ar[d]_-{\epsilon TM} \ar[rr]^-{T\rho} && TSM\ar[d]^-{\epsilon SM} \ar[rr]^-{\chi M} && STM \ar[d]^-{S\epsilon M} \ar[rr]^-{S \rho M} && SSM \ar[d]^-{S \epsilon M} \\
TM \ar[rr]_-{\rho} && SM \ar@{=}[rr] && SM \ar@{=}[rr] && SM
}
$$
commutes, and thus so does the outer diagram. Similarly, when $i = 0$, we have a commutative diagram
$$
\xymatrix{
TTM \ar[d]_-{T \epsilon M} \ar[rr]^-{T\rho} && TSM\ar[d]^-{T\epsilon M} \ar[rr]^-{\chi M} && STM \ar[d]^-{ \epsilon TM} \ar[rr]^-{S \rho M} && SSM \ar[d]^-{ \epsilon SM} \\
TM \ar@{=}[rr] && TM \ar@{=}[rr] && TM \ar[rr]_-\rho && SM
}
$$
proving that diagram~\ref{superrho} commutes in the base case.

Now suppose that the statement holds for $n-1$. Let $i$ be such that $0 < i \le n+1$, and write $j = i - 1$, so that $0 \le j \le n$. Consider the diagram
$$
\xymatrix{
T^{n+2}M \ar[rr]^-{T^{n+1} \rho} \ar[d]_-{T^{n-j} \epsilon T^{j+1} M} && T^{n+1} SM \ar[d]_-{T^{n-j} \epsilon T^j SM} \ar[rr]^-{\chi^{n+1} M} && ST^{n+1}M \ar[d]^-{ST^{n-j} \epsilon T^j M} \ar[rr]^-{S \rho_n} && S^{n+2}M \ar[d]^-{S S^j \epsilon S^{n-j} M} \\
T^{n+1}M \ar[rr]_-{T^n \rho} && T^nSM \ar[rr]_-{\chi^n M} && ST^nM \ar[rr]_-{S \rho_{n-1}} && S^{n+1}M
}
$$
The right-hand square commutes by inductive hypothesis, the middle square commutes by Lemma~\ref{superchiepsilon}, and the left-hand square commutes by naturality of $\epsilon$. The case for $i = 0$ is similarly proved, using compatibility of $\chi^n$ with $\epsilon \colon S \Rightarrow 1$ from Remark~\ref{superchimorphism}.
\end{proof}
\end{lem}
There is a similar statement regarding the coproducts $\delta$. Thus, we have proved the following:
\begin{prop}\label{simplicialrho}
The morphisms $\rho_n$ define a morphism of simplicial functors
$$
\xymatrix{
\rBB(T,M) \ar[r]^-{\rho} & \rBB^*(S,M).
}$$
\end{prop}

\begin{lem}\label{cycliclem}
For all $i,j,k$ with $i \ge 1$, $k \ge i$ and $j \ge 0$, the diagram
\begin{equation}\label{cycliclemma}
\begin{array}{c}
\xymatrix{
NT^jS^i T^{k-i+1} M \ar[dd]_-{\lambda_{i-1, i+j - 1} T^{k-i+1}M} \ar[rrr]^-{NT^j \rho_{i, k} } &&& NT^jS^{i+1}T^{k-i} M \ar[dd]^-{\lambda_{i, i+j} T^{k-i}M} \\
\\
NT^{j+1} S^{i-1} T^{k-i+1} M\ar[rrr]_-{NT^{j+1} \rho_{i-1, k-1}} & && NT^{j+1} S^i T^{k-i} M
}
\end{array}
\end{equation}
commutes.
\end{lem}
\begin{proof}
Consider the diagram
$$
\xymatrix@C=2.31em{
NT^jS^iT^{k-i + 1} M \ar[dd]^-{N \chi^jS^{i-1} T^{k-i+1}} \ar[rr]^-{NT^j S^i T^{k-i}\rho} & & NT^j S^i T^{k-i} SM \ar[dd]^-{N\chi^j S^{i-1}T^{k-i}M} \ar[rr]^-{NT^jS^i \chi^{k-i} M} && NT^j S^{i+1}T^{k-i}M \ar[dd]_-{N\chi^j S^i T^{k-i}M}\\
\\
NST^j S^{i-1} T^{k-i + 1} M \ar[dd]^-{\lambda T^j S^{i-1} T^{k-i+1} M} \ar[rr]_-{NST^j S^{i-1}T^{k-i} \rho} && NST^j S^{i-1}T^{k-i} SM \ar[rr]_-{NST^jS^{i-1}\chi^{k-i} M} \ar[dd]_-{\lambda T^{j}S^{i-1} T^{k-i}SM} && NST^j S^i T^{k-i} M \ar[dd]_-{\lambda T^jS^iT^{k-i}M} \\
\\
NT^{j+1}S^{i-1}T^{k-i+1} M  \ar[rr]_-{NT^{j+1}S^{i-1}T^{k-i}\rho} && NT^{j+1}S^{i-1}T^{k-i}SM \ar[rr]_-{NT^{j+1}S^{i-1}\chi^{k-i} M} && NT^{j+1}S^{i} T^{k-i} M
}
$$
The upper squares commute by naturality of $\chi^j$, and the lower squares commute by naturality of $\lambda$. Therefore, the outer square, which is the same as the square~\ref{cycliclemma}, commutes.
\end{proof}

Let $R_n$ denote the morphism $N \rho_n \colon NT^{n+1}M \Rightarrow N S^{n+1}M$, and let $L_n$ denote the morphism $\lambda_n M \colon   N S^{n+1}  M  \Rightarrow
 N T^{n+1}  M$.
\begin{thm}\label{cyc}
The above construction defines
two morphisms
\begin{align*}
	\xymatrix{ \rCC_T(N,M) \ar[r]^-{R }  &
\rCC_S^*(N,M),} \qquad
	\xymatrix{\rCC_S^*(N,M)
	\ar[r]^-{L } & \rCC_T(N,M)}
\end{align*}
of duplicial functors. Furthermore, $L
R  = 1$ if and only if $\rCC_T
(N,M)$ is cyclic, and $R  L  = 1$
if and only if $ \rCC_S^*(N,M)$ is cyclic.
\end{thm}
\begin{proof}
We show that $R$ is a morphism of duplicial functors, and remark that the proof for $L$ is similar. It is clear that $R$ is a morphism of simplicial objects, so it remains to show that it commutes with the duplicial operators. Let $n \ge 0$ and consider the diagram
$$
\xymatrix@C=0.799em{
 & & & NS^{n+1} M \ar[dr]^-{\ \lambda S^n M}  \ar@/^6.5ex/[dddrrr]^-{t_S }& & & \\
 & & NST^n M \ar[ddr]_-{\lambda T^n M}\ar[ur]^-{NS\rho_{n-1}}& &  NTS^n M \ar[dr]^-{N \chi^n M} & &  \\
& NT^n SM \ar[ur]^-{N \chi^n M} & & & & NS^n TM \ar[dr]_-{N S^n \rho} & \\
\ar@/^6.5ex/[uuurrr]^-{R_n }NT^{n+1}M \ar[ur]_-{NT^n \rho} \ar[rrr]_-{t_T} & & & NT^{n+1} M \ar[uur]_-{NT\rho_{n-1}} \ar[rrr]_-{R_n} & & & NS^{n+1}M
}
$$
The left-hand triangle commutes by definition of $t_T$, and the right-hand semicircle commutes by definition of $t_S$. The right-hand triangle and left-hand semicircle commute by the recursive definitions of $\rho_n$. The middle kite commutes by naturality of $\lambda$, and thus the outer diagram commutes, proving that $R$ is a morphism of duplicial functors.

We finish the proof of the theorem by showing that $(LR)_n = t_T^{n+1}$, which is the identity if and only if $\C_T(N,M)$ is cyclic.
Consider the following two triangles:
$$
\xymatrix@=2.9em{
\smxy{NT^{n+1}M} \ar[r]^-{\smxy{N\rho_{0, n}}} \ar[dr]_-{\smxy{t_T}} & \smxy{NST^n M} \ar[d]^-{\smxy{\lambda_{0,0} T^n M}}\ar[r]^-{\smxy{N\rho_{1,n}}} & \smxy{NS^2T^{n-1}M} \ar[d]^-{\smxy{\lambda_{1,1} T^{n-1}M}}\ar@{.}[r]& \smxy{NS^nTM} \ar[r]^-{\smxy{N\rho_{n,n}}} \ar[d]^-{\smxy{\lambda_{n-1,n-1}TM}} & \smxy{NS^{n+1}M} \ar[d]^-{\smxy{\lambda_{n,n}M}} \\
 & \smxy{NT^{n+1}M} \ar@{.}[dr]
  \ar[r]_-{\smxy{NT\rho_{0, n-1}}} & \smxy{NTST^{n-1}M} \ar@{.}[r] \ar@{.}[d] & \smxy{NTS^{n-1} TM} \ar@{.}[d]  \ar[r]_-{\smxy{NT\rho_{n-1,n-1}}} & \ar@{.}[d] \smxy{NTS^n M}\\
 & & \smxy{NT^{n+1}M} \ar@{.}[r] \ar@{.}[dr]& \smxy{NT^{n-1} STM} \ar[d]_-{\smxy{\lambda_{0, n-1}M}}\ar[r]_-{\smxy{NT^{n-1} \rho_{1,1}}} & \smxy{NT^{n-1}S^2M} \ar[d]^-{\smxy{\lambda_{1,n} M}} & \\
 NT^{n+1}M \ar[drr]_-{t_T} \ar[rr]^-{NT^{n-k} \rho_{0,k}}& &  {NT^{n-k}ST^k M} \ar[d]^-{\lambda_{0,n-k}T^kM} & \smxy{NT^{n+1}M} \ar[r]^-{\smxy{NT^n \rho_{0,0}}}\ar[dr]_-{\smxy{t_T}}& \smxy{NT^n SM} \ar[d]^-{\smxy{\lambda_{0,n}M}} \\
 & &  NT^{n+1}M  & & \smxy{NT^{n+1}M}
}
$$
Applying Lemma~\ref{cycliclem} a total of $\frac{n}{2}(n+1)$ times shows that each square of the right-hand triangle commutes. Each triangle along the diagonal is of the form of the lower-left triangle, for some $k$ such that $0 \le k \le n$. This lower-left triangle can also be written as the outer triangle of
$$
\xymatrix@=3em{
NT^nM \ar[ddrrr]_-{t_T}\ar[r]^-{NT^n \rho} & NT^n SM \ar[rrd]_-{N\chi^n M\ }\ar[rr]^-{NT^{n-k} \chi^k M} && NT^{n-k} ST^k M \ar[d]^-{N\chi^{n-k}T^kM} \\  && & NST^n M  \ar[d]^-{\lambda T^n M} \\
& &  & NT^{n+1}M
}
$$
which commutes since the inner shapes do. Therefore the very large outer triangle above commutes, which means exactly that $(LR)_n = t_T^{n+1}$. A similar argument shows that $(RL)_n = t_S^{n+1}$, completing the proof.
\end{proof}

\subsection{Homologically trivial $ \chi $-coalgebras}
Let $\cY$ be an abelian category.
\begin{defn}
Let $X$ be a non-negative chain complex in $\cY$. We say that $X$ is \emph{contractible} if $X$ is homotopy-equivalent to the zero chain complex.
\end{defn}
A complex $X$ is contractible if and only if the identity chain morphism $1\colon X \to X$ is homotopic to $0$. Explicitly this means that there is a contracting homotopy, i.e.\ a morphism $h \colon X \to X$ of degree 1, such that $hb + bh = 1$, where $b$ denotes the differential of $X$. This implies that the homology objects $\HH^n(X)$ are trivial for $n > 0$.

As we had announced above, $ \chi $-coalgebras as in
Proposition~\ref{triv}
lead to contractible chain complexes:

\begin{prop}
\label{trivcontract}
Let $\chi \colon
TS \Rightarrow ST$ be a comonad distributive
law on a category $\cB$, and let $(M, \cX,
\rho)$ and $(N, \cY, \lambda)$ be a $\chi$-coalgebra and $\chi$-opcoalgebra, respectively. If either of
$(N,Y, \lambda),(M, \cX, \rho)$ arises as in
Proposition~\ref{triv},
then the chain complexes associated to both
$\rCC_T (N, M)$ and $\rCC_S^* (N, M)$
under the Dold-Kan correspondence are contractible.
\end{prop}
\begin{proof}
Let $\nabla \colon N \Rightarrow NT$ be a  $T$-opcoalgebra
structure  on $N$.
The morphisms
$$
	\xymatrix@=4em{h_n \colon NS^{n+1} M
	\ar[r]^-{\nabla S^{n+1} M}
	&NTS^{n+1}M
	\ar[r]^-{N\chi^{n+1}M} &NS^{n+1} T M \ar[r]^-{NS^{n+1}
\rho} &NS^{n+2} M }
$$
provide a contracting homotopy
for the complex associated to $\rCC_S^*
(N,M)$. To see this, consider the following diagram, where $n >0$ and $i$ is such that $0 \le i \le n$:
$$
\xymatrix@C=3.5em{
NS^{n+1}M \ar[r]^-{\nabla S^{n+1}M} \ar[d]_-{NS^{n-i}\epsilon S^i M}& NTS^{n+1}M \ar[d]_-{NTS^{n-i}\epsilon S^i M}\ar[r]^-{N\chi^{n+1}M} & NS^{n+1}TM \ar[d]^-{NS^{n-i}\epsilon S^i T M} \ar[r]^-{NS^{n-1}\rho M} & \ar[d]^-{NS^{n-i}\epsilon S^{i+1}M}NS^{n+2}M \\
NS^nM \ar[r]_-{\nabla S^n M} & NTS^n M \ar[r]_-{N\chi^n M} & NS^n TM \ar[r]_-{NS^n \rho M} & NS^{n+1}M
}
$$
The left-hand square commutes by naturality of $\nabla$, the right-hand square commutes by naturality of $\epsilon$, and the middle square commutes by the version of Lemma~\ref{superchiepsilon} for $\chi^n \colon TS^n \Rightarrow S^n T$. Therefore, the outer rectangle commutes, so we have the relation $h_{n-1}d_i = d_{i+1}h_n$, where $d$ denotes the face morphisms of $\rCC_S^*(N,M)$. Similarly, we have $d_0 h_n = 1$. Therefore, we have
\begin{align*}
bh_n &=  \sum_{i=0}^{n+1} (-1)^i d_i h_n \\
&= d_0 h_n + \sum_{i=0}^n (-1)^{i+1} d_{i+1} h_n \\
&= 1 + \sum_{i=0}^n (-1)^{i+1} h_{n-1}d_i \\
&= 1 - h_{n-1}b,
\end{align*}
so $hb + bh = 1$, as required.

A similar proof shows that the morphisms $\nabla T^n M \colon NT^{n+1}M \Rightarrow
NT^{n+2}M$ provide a contracting homotopy for the
complex associated to $\rCC_T (N, M)$. The case when $M$ is an $S$-coalgebra is proved similarly.
\end{proof}

\subsection{Twisting by $1$-cells}
\label{twistsec}
Applying the twisting procedure described in
Section~\ref{twistcoeff},
a $1$-cell
$$
\xymatrix{
(\cB, \chi, T, S) \ar[rr]^-{(\Sigma, \sigma, \gamma)} & & (\cD, \tau, G, C)
}
$$
in the $2$-category $\dist$, together with a $\chi$-coalgebra $M$ and a $\tau$-opcoalgebra $N$, give rise to morphisms
between duplicial functors of the form considered
above: Theorems~\ref{dup} and~\ref{twist} yield
duplicial structures on the simplicial functors $$
\rCC_T(N\Sigma,M),\quad \rCC_S^*(N\Sigma,M),
\quad \rCC_G(N,\Sigma M),
\quad \rCC_C^*(N,\Sigma M),
$$
and from Proposition~\ref{cyc} we obtain morphisms
\begin{align*}
	&\xymatrix{ \rCC_T(N\Sigma,M)
	\ar[r]^{R^\chi}  &
	\rCC_S^*(N\Sigma,M), } &
	\xymatrix{ \rCC_S^*(N\Sigma,M)
	\ar[r]^{L^\chi} & \rCC_T(N\Sigma,M)} \\
	&\xymatrix{ \rCC_G(N,\Sigma M) \ar[r]^-{R^\tau}
& \rCC_C^*(N,\Sigma M),} & \xymatrix{
\rCC_C^*(N,\Sigma M) \ar[r]^-{L^\tau} &
\rCC_G(N,\Sigma M)}
\end{align*}
of duplicial
objects which determine the cyclicity of each functor.

We now prove that a generalised Yang-Baxter equation holds:
\begin{lem}\label{superyb}
For all $n > 0$, the diagram
$$
\xymatrix@C=3em{
\ar[d]_-{\sigma^n S}G^n\Sigma S\ar[r]^-{G^n \gamma} & G^n C\Sigma \ar[r]^-{\tau^n \Sigma} & CG^n \Sigma \ar[d]^-{C\sigma^n} \\
\Sigma T^n S \ar[r]_-{\Sigma \chi^n} & \Sigma ST^n \ar[r]_-{\gamma T^n} & C\Sigma T^n
}
$$
commutes.
\end{lem}
\begin{proof}
When $n = 1$, this is the Yang-Baxter equation. If the diagram commutes for $n$, consider the diagram
$$
\xymatrix@C=3em{
G^{n+1} \ar[d]_-{G\sigma^n S} \Sigma S \ar[r]^-{G^{n+1}\gamma} & G^{n+1}C\Sigma \ar[r]^-{G\tau^n \Sigma} & GCG^n\Sigma \ar[r]^-{\tau G^n\Sigma} \ar[d]_-{GC\sigma^n} & CG^{n+1} \Sigma \ar[d]^-{CG \sigma^n} \\
G\Sigma T^n S \ar[d]_-{\sigma T^n S} \ar[r]^-{G\Sigma \chi^n} & \ar[d]_-{\sigma ST^n} G\Sigma ST^n \ar[r]^-{G\gamma T^n} & GC\Sigma T^n \ar[r]^-{\tau \Sigma T^n}& CG\Sigma T^n \ar[d]^-{C\sigma T^n} \\
\Sigma T^{n+1}S \ar[r]_-{\Sigma T\chi^n}& \Sigma TST^n \ar[r]_-{\Sigma \chi T^n}& \Sigma ST^{n+1}  \ar[r]_-{\gamma T^{n+1}}& C\Sigma T^{n+1}
}
$$
The upper-right square commutes by naturality of $\tau$, the lower-left square commutes by naturality of $\sigma$, the lower-right rectangle commutes by the Yang-Baxter equation, and the top left-diagram commutes by inductive hypothesis. Therefore, the outer diagram commutes, proving the lemma.
\end{proof}
By applying the functors $M,N$ to the simplicial morphisms $\sigma,\gamma$ of Proposition~\ref{supersigma} and its dual, we obtain two simplicial morphisms
$$
\xymatrix{ \rCC_G(N,\Sigma M) \ar[rr]^-{N\sigma M} &&
\rCC_T(N\Sigma, M),} \quad \quad \quad
\xymatrix{ \rCC_S^*(N\Sigma, M) \ar[rr]^-{N\gamma M} &&
\rCC_C^*(N,\Sigma M).}
$$
\begin{prop}\label{rhocyclic}
The morphisms $N\sigma M$, $N\gamma M$ are morphisms of duplicial functors.
\end{prop}
\begin{proof}
We prove this only for $N\sigma M$. It suffices to show this commutes with the duplicial operators $t_T, t_G$. Recall, from Section~\ref{twistcoeff}, that the (op)coalgebra structure morphisms on $\Sigma M$ and $N \Sigma$  are respectively given by the composites
$$
\xymatrix@R=1em{
G\Sigma M \ar[r]^-{\sigma M} & \Sigma TM \ar[r]^-{\Sigma\rho} & \Sigma S M \ar[r]^-{\gamma M} & C \Sigma M,
\\
N \Sigma S \ar[r]^-{N \gamma} & N C \Sigma \ar[r]^-{\lambda \Sigma} & NG\Sigma \ar[r]^-{N \sigma} & N \Sigma T.
}
$$
Let $n \ge 0$. In the diagram
$$
\xymatrix@C=2.74em{
\smmxy{NG^n\Sigma TM} \ar[d]^-{N \sigma^n TM} \ar[r]^-{NG^n \Sigma \rho} & \smmxy{NG^n\Sigma SM} \ar[d]^-{N\sigma^n SM }\ar[r]^-{NG^n \gamma M} & \smmxy{NG^nC\Sigma M} \ar[r]^-{N\tau^n \Sigma M} & \smmxy{NCG^n\Sigma M} \ar[d]_-{NC\sigma^n M} \ar[r]^-{\lambda G^n \Sigma M} & \smmxy{NG^{n+1}\Sigma M} \ar[d]_-{NG\sigma^n M}\\
\smmxy{N\Sigma T^{n+1} M} \ar[r]_-{N\Sigma T^n \rho} & \smmxy{N\Sigma T^n SM} \ar[r]_-{N\Sigma \chi^n M} & \smmxy{N\Sigma ST^n M} \ar[r]_-{N\gamma T^n M} & \smmxy{NC\Sigma T^n M} \ar[r]_-{\lambda \Sigma T^n M} & \smmxy{NG\Sigma T^n M}
}
$$
the middle square commutes by Lemma~\ref{superyb}, the right-hand square commutes by naturality of $\lambda$, and the left-hand square commutes by naturality of $\sigma^n$. The outer diagram commutes therefore, and after pre-composing with $NG^n \sigma M$ and post-composing with $N\sigma T^n M$, we have the result.
\end{proof}
\begin{exa}
Let
$$
\xymatrix{
\cA \ar@{}[rr]|-{\perp}\ar@/^0.5pc/[rr]^-F & & \ar@/^0.5pc/[ll]^-U \cB
}
$$
be an adjunction, generating a monad $B = UF$ on $\cA$ and a comonad $T = FU$ on $\cB$. Let $$
\xymatrix{\cB \ar[r]^ U  \ar[d]_ S & \ar@{}[dl]^(.25){}="a"^(.75){}="b" \ar@{=>}_-{\Omega} "a";"b"   \cA \ar[d]^ C\\
	\cB \ar[r]_U  & \cA}
$$
be a square where $(U, \Omega) \colon (\cB, S) \to (\cA, C)$ is an iso-opmorphism of comonads, giving rise to a mixed distributive law $\theta$ and a comonad distributive law $\chi$ by Corollary~\ref{arisec}. By Theorem~\ref{distcomp1cell}, the comparison functor $\cB \to \cA^B$ gives rise to a 1-cell
$$\xymatrix{(\cB, \chi, T, S) \ar[rrr]^-{(U^{U\epsilon}, 1, \tilde\Omega^{-1})} & & & (\cA^B, \tilde\theta, \tilde B, C^\theta)
}
$$
in $\dist$. Given any $\chi$-coalgebra $M$ and $\tilde\theta$-opcoalgebra $N$, by Proposition~\ref{rhocyclic} we have isomorphisms of duplicial functors
$$
C_{\tilde B}\left(N, U^{U\epsilon} M\right) \cong C_T\left(NU^{U\epsilon}, M\right), \qquad
C_S^*\left(NU^{U\epsilon}, M\right) \cong C_{C^\theta}^* \left(N, U^{U\epsilon} M\right).
$$
Since $\sigma = 1$ in this example, the left-hand isomorphism is actually an equality.
\end{exa}
\subsection{Cyclicity and the reflection equation}
We conclude this chapter by observing that there is an interesting connection between the cyclicity of duplicial functors of the above form, and the reflection equation of physics. Let $\chi \colon TS \Rightarrow ST$ be a distributive law of comonads, and let $M$ and $N$ be a $\chi$-coalgebra and a $\chi$-opcoalgebra respectively. Let us represent the natural transformations diagrammatically as
$$
\chi = \begin{array}{c}
\begin{tikzpicture}
\draw [thick] (-1, 1) node[anchor=south]{$T$}  -- (0, -1) node[anchor=north]{$T$};
\draw [thick] (0, 1) node[anchor=south]{$S$} -- (-1,- 1) node[anchor=north]{$S$};
\end{tikzpicture}
\end{array},\qquad \qquad
\rho =
\begin{array}{c}
\begin{tikzpicture}
\draw [thick] (-1, 1) node[anchor=south]{$T$}  -- (0, -0);
\draw [thick](0, 1) node[anchor=south]{$M$} -- (0,- 1) node[anchor=north]{$M$};
\draw[thick] (0,0) -- (-1, -1) node[anchor=north]{$S$};
\end{tikzpicture}
\end{array},\qquad \qquad
\lambda =
\begin{array}{c}
\begin{tikzpicture}
\draw [thick](0,1) node[anchor=south]{$N$} -- (0, -1) node[anchor=north]{$N$};
\draw [thick](1,1) node[anchor=south]{$S$} -- (0,0);
\draw [thick](0,0) -- (1,-1) node[anchor=north]{$T$};
\end{tikzpicture}
\end{array}
$$
Using graphical calculus and reading from top to bottom, in degree 1, the two morphisms $t_T^2, t_S^2$ look like
$$
\begin{array}{c}
\begin{tikzpicture}
\draw [thick](0, 5) node[anchor=south]{$N$} -- (0, 0) node[anchor=north]{$N$};
\draw [thick](1,5) node[anchor=south]{$T$} -- (3,3) -- (0,1) --(1,0) node[anchor=north]{$T$};
\draw [thick](2,5) node[anchor=south]{$T$} -- (3,4) --(0,2) -- (2,0) node[anchor=north]{$T$};
\draw [thick](3,5) node[anchor=south]{$M$} -- (3,0) node[anchor=north]{$M$};
\end{tikzpicture}
\end{array},
\qquad
\begin{array}{c}
\begin{tikzpicture}
\draw [thick](0, 5) node[anchor=south]{$N$} -- (0, 0) node[anchor=north]{$N$};
\draw[thick](1,5) node[anchor=south]{$S$} -- (0,4) -- (3,2) --(1,0) node[anchor=north]{$S$};
\draw[thick](2,5) node[anchor=south]{$S$} -- (0,3) -- (3,1) --(2,0) node[anchor=north]{$S$};
\draw[thick](3,5) node[anchor=south]{$M$} -- (3,0) node[anchor=north]{$M$};
\end{tikzpicture}
\end{array}
$$
We can view $N,M$ as walls and the two inner strands may be viewed as the trajectories of two distinct particles bouncing between them. Now, instead of viewing the particles as bouncing off the right-hand walls, let us ignore the wall and see the particle's trajectory as a straight line. In other words, we remove the right-hand wall, straighten out each of the right-hand kinks, but \emph{preserve} the points where the lines cross. Heuristically, as long as a particle $T$ is furthest to the right, it may undergo a `state-change' and turn into the particle $S$ which allows us to use the distributive law $\chi$ to cross the particles. Ignoring the labels, if we redraw the above diagrams with this viewpoint, we have:
$$
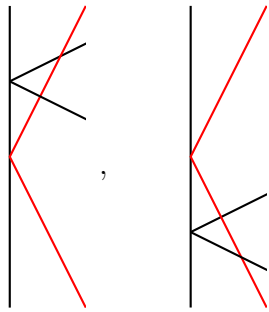
\begin{array}{c}
\begin{tikzpicture}
\useasboundingbox (0,4) rectangle (1,0);
\clip (-1, 4) rectangle (1,0);
\draw [thick] (0,4) -- (0,0);
\draw [thick] (2, 4) -- (0,3);
\draw [red] [thick] (1,4) -- (0,2) -- (1,0);
\draw [thick] (0,3) -- (2,2);
\end{tikzpicture}
\end{array}, \qquad
\begin{array}{c}
\begin{tikzpicture}
\useasboundingbox (0,5) rectangle (1,1);
\clip (-1, 5) rectangle (1,1);
\draw [thick] (0,5) -- (0,1);
\draw [thick] (2, 3) -- (0,2);
\draw [red] [thick] (1,5) -- (0,3) -- (1,1);
\draw [thick] (0,2) -- (2,1);
\end{tikzpicture}
\end{array}
$$
Each diagram is one side of the reflection equation, which describes the trajectories of two particles with different velocities bouncing off a wall in different orders. We use the colour red to distinguish between the particles in the physical interpretation.

We may construct similar pairs of diagrams for $n > 1$, as well as reflection equations for more than two particles. For example, the reflection equation for $n = 2$ (i.e.\  for 3 particles) is
$$
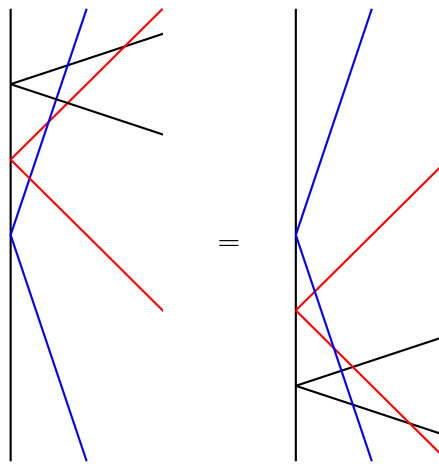
\begin{array}{c}
\begin{tikzpicture}
\useasboundingbox (0,6) rectangle (2,0);
\clip	(-1,6) rectangle (2,-1);
\draw [thick] (0,6) -- (0,0);
\draw [thick] (3,6) -- (0,5);
\draw [thick] [red] (2,6) -- (0,4);
\draw [thick] [blue] (1,6) -- (0,3) -- (1, 0);
\draw [thick] [red] (0,4) -- (3,1);
\draw[thick] (0,5) -- (3,4);
\end{tikzpicture}
\end{array} \quad = \quad
\begin{array}{c}
\begin{tikzpicture}
\useasboundingbox (0,6) rectangle (2,0);
\clip(-1, 6) rectangle (2, -1);
\draw [thick] (0,6) -- (0,0);
\draw[thick] [blue] (1,6) -- (0,3);
\draw[thick] [red] (3,5) -- (0,2);
\draw [thick] (3, 2) -- (0, 1);
\draw[thick] (0,1) -- (3,0);
\draw[thick] [red] (0, 2) -- (3,-1);
\draw[thick] [blue] (0,3) -- (1,0);
\end{tikzpicture}
\end{array}
$$Therefore, we immediately obtain another characterisation of cyclicity.
\begin{prop}
The duplicial functor $\rCC_T(N,M)$ is cyclic if and only if the appropriate side of the reflection equation in each degree is equal to the identity.
\end{prop}
Of course, there is a similar statement for the duplicial functor $C_S^*(N,M)$.
\chapter{Hochschild viewpoint of dupliciality}\label{AUSTRALIA}
The purpose of this chapter is to explain that Hochschild homology and cohomology of algebras can be imported to the world of 2-categories, in turn giving an insight into the nature of duplicial structure. We begin by recounting the classical case (Section~\ref{hochclassic}). In Sections~\ref{HOCHLAX} and~\ref{duplapp} we construct an upgraded version of Hochschild (co)homology and see that we recover the duplicial object $C_T(N,M)$ of Chapter~\ref{CYCLIC}. We also characterise duplicial structure on the nerve of a category $\cC$ in terms of left adjoint functors from groupoids into $\cC$.
The work in these latter sections is original and is based on that carried out in~\cite{woohoo}.

\section{The classical case}\label{hochclassic}
Let $k$ be a commutative ring and let $A$ be a (unital, associative) $k$-algebra. We denote the monoidal product $\otimes_k$ of $\kmod$ by an unadorned tensor product $\otimes$. Let $A^*$ denote the opposite algebra to $A$, and let $\Ae$ denote the enveloping algebra $A \otimes A^*$. If we denote the category of $(A,A)$-bimodules (with symmetric action of $k$) by $\amoda$, there are isomorphisms
$$
\amoda \cong \aemod \cong \modae.
$$
and we therefore identify all of these.

Consider the endofunctor $B = {-}\otimes A$ on $\amod$ which takes a left $A$-module $X$ to the module $X \otimes A$ with left $A$-action given by $a\cdot(x \otimes a') := ax \otimes a'$.
The natural maps
\begin{align*}
&\xymatrix@R=1em{X \otimes A \otimes A \ar[r]^-{\mu} & X \otimes A \\ x\otimes a\otimes a' \ar@{|->}[r] &  x\otimes aa'} & &\xymatrix@R=1em{ X \ar[r]^-{\eta} & X \otimes A \\ x \ar@{|->}[r] & x\otimes 1}
\end{align*}
turn the functor $B = {-}\otimes A$ into a monad on $\amod$, which lifts to a comonad $\tilde B$ on the Eilenberg-Moore category $\amod^B$. In fact, there is another isomorphism
$$
\amod^B \cong \modae
$$
so we view $\tilde B$ as a comonad on the category of bimodules, defined by the morphisms
\begin{align*}
&\xymatrix@R=1em{M \otimes A  \ar[r]^-{\delta} & M \otimes A \otimes A \\ m\otimes a \ar@{|->}[r] &  m\otimes 1 \otimes a}
& &\xymatrix@R=1em{ M \otimes A \ar[r]^-{\epsilon} & M \\ m \otimes a \ar@{|->}[r] & ma.}
\end{align*}
So, for every bimodule $M$, we have that $\tilde B(M) = M \otimes A$ becomes a bimodule with actions given by
$$
a\cdot(m \otimes a') \cdot a'' := am \otimes a'a''.
$$
\subsection{Hochschild homology and cohomology}\label{hochsub}
Viewing $A$ as a left $\Ae$-module, there is a functor $${-} \otimes_{\Ae}A \colon \modae \to \kmod.$$
Also, let $M$ be any $(A,A)$-bimodule, viewed as a right $\Ae$-module.
\begin{defn}\label{hochdef}
The \emph{Hochschild homology of $A$ with coefficients in $M$}, denoted by $\rH_\bullet(A,M)$,  is the graded $k$-module given by the homology of the chain complex associated to the simplicial $k$-module $\rCC_{\tilde B}({-}\otimes_{\Ae} A, M)$.
\end{defn}
The above simplicial object is seen again in Chapter~\ref{EXAMPLES}. However, there is a less complicated simplicial object which may be used to define Hochschild homology. For every left $A$-module $P$, there are $k$-module isomorphisms
$$
\xymatrix@R=1em{
(P \otimes A)\otimes_{\Ae} A \ar[r]^-{\cong} & P \\
(p \otimes a) \otimes_{\Ae} a' \ar@{|->}[r] & aa'p,}
$$ and so we have isomorphisms for $n \ge 0$
$$
\rCC_{\tilde B}(- \otimes_{\Ae} A, M)_n = \left(M \otimes A^{\otimes(n+1)}\right) \otimes_{\Ae} A \cong M \otimes A^{\otimes n}.
$$
Thus, we get a simplicial $k$-module $C_\bullet(A, M)$ defined by $C_n(A, M) = M \otimes A^{\otimes n}$, with face maps and degeneracy maps
$$d_i \colon C_n(A,M) \to C_{n-1}(A,M), \qquad s_i \colon C_{n}(A,M) \to C_{n+1}(A,M)$$ given by
\begin{align*}
d_i(m \otimes a_1 \otimes \cdots \otimes a_n) &=
\begin{cases}
a_n m \otimes a_1 \otimes \cdots \otimes a_{n-1} &\mbox{ if } i = 0 \\
m \otimes a_1 \otimes \cdots \otimes a_{n-i}a_{n-i+1} \otimes \cdots \otimes a_n & \mbox{ if } 1 \le i < n \\
ma_1 \otimes a_2 \otimes \cdots \otimes a_n & \mbox{ if } i = n
\end{cases} \\
s_i(m \otimes a_1 \otimes \cdots \otimes a_n) &=
m \otimes \cdots \otimes a_{n-i} \otimes 1 \otimes a_{n-i+1} \otimes \cdots \otimes a_n
\end{align*}
The homology of the chain complex associated to this complex is therefore isomorphic to $\rH_\bullet(A, M)$.  The zeroth Hochschild homology is given by
 $$
 \rH_0(A, M) \cong {M}/\langle am - ma \rangle
 $$
where $\langle am - ma \rangle$ denotes the submodule of $M$ generated by all expressions of the form $ma - am$ for $m \in M$, $a \in A$.

Consider the contravariant functor
$$
\Hom_\Ae({-}, M) \colon \modae \to \kmod
$$
which assigns to a bimodule $N$ the $k$-module of $(A,A)$-bimodule maps $N \to M$.
\begin{defn}
The \emph{Hochschild cohomology of $A$ with coefficients in $M$}, which we denote by $\rH^\bullet(A,M)$,  is the graded $k$-module given by the cohomology of the cochain complex associated to the cosimplicial $k$-module $\rCC_{\tilde B}(\Hom_\Ae({-}, M), A).$
\end{defn}
The zeroth Hochschild cohomology is given by the centre of $M$, explicitly:
$$
\rH^0(A, M) \cong ZX = \{ m \in M \mid ma = am \mbox{ for all } m \in A\}.
$$

Note that both $\rH_0 (A, {-} )$ and $\rH^0(A, {-})$ define functors $\amoda \to \kmod$.
We do not delve into Hochschild cohomology further, as we do not discuss anything beyond this zeroth homology module.
\subsection{Universal properties}\label{hochunivprop}
For any $k$-module $X$, the set
$$
[M,X] := \Hom_k(M,X)
$$
becomes an $(A,A)$-bimodule with actions given by
$$
(a \cdot f \cdot b )  (m) = f(bma).
$$
We now present a universal coefficients theorem:
\begin{thm}\label{hochuniv}
There is an isomorphism of $k$-modules
$$
[\rH_0(A, M) , X] \cong \rH^0(A, [M, X] ),
$$
natural in $M$ and $X$.
\end{thm}
\begin{proof}
By the above remarks, to prove this is equivalent to proving that
$$
\left[{M}/{\langle am - ma \rangle}, X \right] \cong Z([M,X]).
$$
Given $f \in Z([M,X])$, we have that $a f = f a$ for all $a \in A$, which means exactly that $f(ma) = f(am)$ for all $a \in A, m \in M$. Therefore, $f$ descends to a map $\bar f$ on the quotient:
$$
\xymatrix{
M \ar@{->>}[r] \ar[dr]_-{f} & M / \langle am-ma\rangle \ar[d]^-{\bar f} \\
& X
}
$$
The assignment $f \mapsto \bar f$ defines the required isomorphism, whose inverse is given by composing with the quotient map. Naturality follows easily.
\end{proof}
\begin{thm}
We have an adjunction
$$
\xymatrix{
\amoda \ar@{}[rr]|-{\perp}\ar@/^0.5pc/[rr]^-{\rH_0(A, {-})} & & \ar@/^0.5pc/[ll]^-{[A, {-}]} \kmod.
}
$$
\end{thm}
\begin{proof}
We have that $\rH_0(A,{-}) \cong {-} \otimes_\Ae A$ which is left adjoint to $[A,{-}]$ by the tensor-hom adjunction (see e.g.~\cite[Theorem 2.75]{MR2455920}).
\end{proof}
Let us now apply the cohomology functor $\rH^0(A, -)$ to the unit of the adjunction $\rH_0(A, {-} ) \dashv [A, {-}]$, giving a morphism
$$
\xymatrix{
\rH^0(A, M) \ar[rr]^-{\rH^0(A, \eta)} && \rH^0 (A, [A, \rH_0(A, M)]) \ar[r]^-{\cong} & [\rH_0(A,A), \rH_0 (A, M) ]
}
$$
where the isomorphism is that of Theorem~\ref{hochuniv}. Under the closed monoidal structure of $\kmod$, this morphism corresponds to one
$$
\xymatrix{
\rH^0(A, M) \otimes \rH_0(A, A) \ar[r]^-\cap &\rH_0 (A, M),
}
$$
which we call the \emph{cap product.}
\begin{rem}
The reason for this terminology is that it is a special case of the cap product in homological algebra (see~\cite{MR1731415} and~\cite{MR3281654} for a more general version for Hopf algebroids).
\end{rem}
\section{The lax categorical case}\label{HOCHLAX}
We now carry out similar constructions to those of Section~\ref{hochclassic} in the context of categories. Prior to this, we remark that what follows can probably be done with reference to an arbitrary monoid in a symmetric monoidal closed bicategory (so in particular Section~\ref{hochclassic} becomes a special case of this section), but it is beyond the scope of this thesis.

Throughout, let $\cA$ be a monoidal category (playing the r\^ole of $A$ in the previous section) with tensor product $\otimes$ and unit $I$. For simplicity, we assume that $\cA$ is strict monoidal, and any module categories upon which it acts are strict also.
\subsection{The 2-category $\AmodA$}
Let $\cx$ be a category which is both a left-module category and a right-module category for $\cA$, with actions
$$
\xymatrix{
\cA \times \cx \ar[r]^-{\rhd} & \cx, \qquad \cx \times \cA \ar[r]^-{\lhd} & \cx
}
$$
that are \emph{laxly compatible}, in the sense that we have morphisms
$$
\xymatrix{
A \rhd (M \lhd B) \ar[r]^-\chi & (A \rhd M) \lhd B
}
$$
natural in $A,M,B$ satisfying some coherence conditions: namely, the diagrams
\begin{align*}
\xymatrix@C=2.6em{
(A \otimes A') \rhd (M \lhd B) \ar[rr]^-{\chi} \ar@{=}[d] & & ( (A \otimes A') \rhd M ) \lhd B \ar@{=}[d] \\
A \rhd (A' \rhd (M \lhd B)) \ar[r]_-{A \rhd \chi} & A \rhd((A\ \rhd M) \lhd B) \ar[r]_-\chi & (A \rhd(A' \rhd M)) \lhd B
} \\
\xymatrix@C=2.6em{
A \rhd (M \lhd (B \otimes B') ) \ar[rr]^-\chi \ar@{=}[d]  & & (A \rhd M) \lhd (B \otimes B') \ar@{=}[d] \\
A \rhd((M \lhd B) \lhd B') \ar[r]_-{\chi} & (A \rhd (M \lhd B)) \lhd B' \ar[r]_-{\chi \rhd B'} & ((A \rhd M) \lhd B) \lhd B'
}
\end{align*}
commute, and such that $\chi = 1$ whenever $A$ or $B$ is the unit $I$. A 1-cell $P \colon \cx \to \mathcal N$ of such categories is defined to be a functor $P$ such that $PM \lhd B = P(M \lhd B)$ naturally, together with morphisms
$$
\xymatrix{
A \rhd PM \ar[r]^-{r} & P(A  \rhd M)
}
$$
natural in $A,M$, such that the three diagrams
$$
\begin{array}{c}
\xymatrix@C=2.5em{
(A \otimes B) \rhd PM \ar[rr]^-r \ar@{=}[d] &&\ar@{=}[d] P( (A \otimes B) \rhd M) \\
A \rhd(B \rhd PM) \ar[r]_-{A \rhd r} & A \rhd P(B \rhd M) \ar[r]_-r & P(A \rhd (B \rhd M))
}
\\
\xymatrix@C=2.5em{
A\rhd P(M \lhd B) \ar@{=}[d] \ar[rr]^-r && \ar@{=}[d] P(A \rhd (M \lhd B) ) \\
A \rhd(PM \lhd B) \ar[r]_-\chi & ( A \rhd PM) \lhd B \ar[r]_-{r \lhd B} & P(A \rhd M) \lhd B
}
\\
\xymatrix{
I \rhd PM \ar@{=}[dr] \ar[r]^-r & P(I \rhd M) \ar@{=}[d] \\
& PM
}
\end{array}
$$
commute. In other words, $P$ preserves the right action, but only laxly preserves the left action. A 2-cell $\alpha \colon (P, r) \Rightarrow (P', r')$ is a natural transformation $\alpha \colon P \Rightarrow P'$ such that the diagram
$$
\xymatrix{
A \rhd PM \ar[r]^-r \ar[d]_-{A \rhd \alpha} & P(A \rhd M) \ar[d]^-\alpha \\
A \rhd P'M \ar[r]_-{r'} & P'(A \rhd M)
}
$$
commutes. These structures constitute a 2-category, which we denote by $\AmodA$.
\subsection{(Co)homology}\label{cohomo}
We now define the analogous structures to the zeroth Hochschild (co)homology modules of Section~\ref{hochclassic}. Let $\cx$ be a $0$-cell in $\AmodA$. We begin with cohomology, as it is slightly simpler.
\begin{defn}
The category $\rH^0(\cA, \cx)$ has as objects pairs $(M, \rho)$ where $M$ is an object of $\cx$, and $\rho \colon A \rhd M \Rightarrow M \lhd A$ is a natural morphism such that $\rho \colon I \rhd X \to X \lhd I$ is the identity and the diagram
$$
\xymatrix@C=3em{
(A \otimes B)  \ar[d]_-\rho \rhd M \ar@{=}[r] & A \rhd(B \rhd M) \ar[r]^-{A \rhd \rho} & A \rhd(M \lhd B) \ar[d]^-\chi\\
M \rhd (A \otimes B) \ar@{=}[r] & (M \lhd A) \lhd B & \ar[l]^-{\rho \lhd B} (A \rhd M) \lhd B
}
$$
commutes. A morphism $f \colon (M, \rho) \to (M', \rho')$ is a morphism $f \colon M \to M'$ in $\cx$ such that the diagram
$$
\xymatrix@C=3em{
A\rhd M \ar[d]_-\rho \ar[r]^-{A \rhd f} & A \rhd M' \ar[d]^-{\rho '} \\
M \lhd A \ar[r]_-{f \lhd A} & M' \rhd A
}
$$
commutes.
\end{defn}
The category $\rH^0(\cA, \cM)$ can also be described as a \emph{lax descent object}, see~\cite{MR1935980, MR903151, woohoo}.
\begin{defn}
The category $\rH_0(\cA, \cM)$ is constructed as follows: it has the same objects as the category $\cM$, but the morphisms are given by taking the morphisms of $\cM$ and adjoining morphisms
$$
\xymatrix{
M \lhd A \ar[r]^-{\phi} & A \rhd M
}
$$
natural in $A, M$, such that $ \phi \colon M \lhd I \to  I \rhd M$ is the identity and the diagram
$$
\xymatrix@C=2.5em{
M \lhd (A \otimes B) \ar@{=}[r]  \ar[d]_-\phi& (M \lhd A) \lhd B \ar[r]^-\phi & B \rhd( M \lhd A) \ar[d]^-\chi \\
(A \otimes B) \rhd M \ar@{=}[r] & A \rhd (B \rhd M) & \ar[l]^-{\phi} (B \rhd M) \lhd A
}
$$
commutes.
\end{defn}
Dually to cohomology, we can describe the category $\rH_0(\cA, \cM)$ as a \emph{lax codescent object}.
Both $\rH_0(\cA, {-})$ and $\rH^0(\cA, {-})$ define 2-functors $\AmodA \to \cat$.
\subsection{Universal properties}
For any category $\cY$, the functor category $[\cM, \cY]$ becomes a 0-cell in $\AmodA$, with left and right actions on a functor $F \colon \cM \to \cY$ given by
$$
(A \rhd F \lhd B) (M) = F(B \rhd M \lhd A)
$$
on objects. This defines a 2-functor
$$
\xymatrix{
\cat \ar[rr]^-{[\cA,{-}]} && \AmodA.
}$$
Given a category $\cX$, the functor category $[\cX, \cM]$ becomes a 0-cell in $\AmodA$ in an analagous way.
\begin{thm}\label{hochgen}
There is an isomorphism of categories
$$
[\rH_0(\cA, \cM) , \cY] \cong \rH^0(\cA, [\cM, \cY] ),
$$
natural in $\cM$ and $\cY$.
\end{thm}
\begin{proof}
We only define the isomorphisms on objects as their action on morphisms is clear. The functor $$\xymatrix{ [\rH_0(\cA, \cM) , \cY] \ar[r]^-\Phi & \rH^0(\cA, [\cM, \cY] )}$$ is defined as follows. Given a functor $F \colon \rH_0(\cA, \cM) \to \cY$, we define $\Phi(F)$ to be the pair $( F, \rho)$ where the natural morphisms
$$\rho \colon A\rhd F \Rightarrow F \lhd A$$ are defined by
$$
\xymatrix{
F(M \lhd A) \ar[rr]^-{F\phi} && F(A \rhd M).
}
$$
The inverse functor
$$
\xymatrix{  \rH^0(\cA, [\cM, \cY] ) \ar[r]^-\Theta & [\rH_0(\cA, \cM) , \cY] }
$$
maps a pair $(G, \rho)$ to the functor $G^\rho \colon \rH_0(\cA, \cM) \to \cY$ given by $G$ on the morphisms of $\cA$, and on the extra morphisms
$$\xymatrix{
X \lhd A \ar[r]^-{\phi} & A \rhd X }
$$
by
$$
\xymatrix{
G(X \lhd A) = (A \rhd G)(X) \ar[r]^-{\rho} & (G \lhd A)(X) = G(A \rhd X).
}
$$
\end{proof}
\begin{thm}
There is a 2-adjunction
$$
\xymatrix{
\AmodA \ar@{}[rrr]|-{\perp}\ar@/^0.7pc/[rrr]^-{\rH_0(\cA, {-})} & & & \ar@/^0.7pc/[lll]^-{[\cA, {-}]} \cat.
}
$$
\end{thm}
\begin{proof}
Let $\cM$ be an object of $\AmodA$. For any object $M$ in $\cM$, consider the functor $\eta(M) \colon \cA \to \rH_0(\cA, \cM)$ which maps an object $A$ to $A\rhd{M }$. We have the functorial equalities
\begin{align*}
(A \rhd \eta(M)) (B) :&= \eta(M)(B \otimes A) \\
&= (B \otimes A) \rhd M \\
&= B \rhd (A \rhd M) \\
&= \eta(A \rhd M) (B)
\end{align*}
thus proving that $(A \rhd \eta(M) ) = \eta( A \rhd M)$. This means that $\eta$ itself is a 1-cell in $\AmodA$, so we have defined a 2-natural transformation $$\eta \colon 1 \Rightarrow [\cA, \rH_0(\cA, {-})].$$

For any category $\cY$, we define a functor
$$
\xymatrix{
\rH_0(\cA, [\cA, \cY]) \ar[r]^-\epsilon & \cY
}
$$
on objects by $F \mapsto F(I)$. On a morphism of $[\cA, \cY]$, i.e.\ a natural transformation, $\epsilon$ maps this to the component at the unit object $I$. The extra morphisms
$$\xymatrix{
F \lhd A \ar[r]^-\phi & A \rhd F
}$$
are mapped to the identity morphism
$$
(F \lhd A) (I) = F(A \otimes I) = F(I \otimes A) = (A\rhd F)(I).
$$
We have thus defined a 2-natural transformation
$$
\epsilon \colon \rH_0 (\cA, [\cA, {-} ] ) \Rightarrow 1
$$
and it is routine to check that $\eta, \epsilon$ satisfy the triangle identities.
\end{proof}
For an object $\cM$ in $\AmodA$, we take the unit $\eta \colon \cM \to [\cA, \rH_0(\cA, \cM)]$ and apply the functor $\rH^0(\cA,{-})$ as well as the isomorphism of Theorem~\ref{hochgen} to construct a morphism
$$
\xymatrix{
\rH^0(\cA, \cM) \ar[rr]^-{\rH^0(\cA, \eta)} & & \rH^0(\cA, [\cA, \rH_0(\cA, \cM)]) \ar[r]^-{\cong} & [\rH_0(\cA, \cA), \rH_0(\cA, \cM)].
}
$$
which we call the \emph{cap product} (cf.\ Section~\ref{hochunivprop}) which we denote by $\cap$.
\section{Application to duplicial objects}\label{duplapp}
We now focus on a special case of the constructions in the previous section. Consider the augmented simplicial category $\Delta_+$ (cf.~Definition~\ref{simpldef}). This becomes a monoidal category with tensor product $\oplus$ defined on objects by
$$
\mathbf{m \oplus n = m + n+1},
$$
and the tensor product of two morphisms $f \colon \mathbf{m} \to \mathbf{m'}$, $g \colon \mathbf{n} \to \mathbf{n'}$ is given by
$$
(f \oplus g )(i) =  \begin{cases} f(i) & \mbox{ if } 0 \le i \le m \\ g(i - m - 1) &\mbox{ if } m < i \le m+n+1\end{cases}
$$ The unit is given by $\mathbf{-1} = \emptyset$.
\subsection{Actions}
We now set $\cA = \Delta_+^*$ for the remainder of this chapter, and apply the results of Section~\ref{HOCHLAX}.
Let $\cx$ be a category.
\begin{prop}\label{leftaction}
Strict left actions
$$\xymatrix{
\cA \times \cx \ar[r] & \cx}
$$
correspond to comonads on $\cx$.
\end{prop}
\begin{proof}
An action $\cA \times \cx \to \cx$ corresponds to a strict monoidal functor $\cA \to [\cx, \cx]$, which in turn corresponds to a comonad in the monoidal category $[\cx, \cx]$ which finally corresponds to a comonad in $\cx$.
\end{proof}

Explicitly, given a left action $\rhd \colon \cA \times \cx \to \cx$, the corresponding comonad $T$ on $\cx$
is defined as $T = \mathbf 0 \rhd{-}$, with counit $\epsilon$ given on components $X$ by
$$
\xymatrix{
T(X) = \mathbf 0 \rhd X \ar[rr]^-{d_0 \rhd X} && \mathbf{-1} \rhd X = X
}
$$
where $d_0 \colon \mathbf{0 \to -1}$ is the unique face map in $\cA$. The coproduct $\delta$ is defined on components by
$$
\xymatrix{
T X = \mathbf{0} \rhd X \ar[rr]^-{s_0 \rhd X} & & \mathbf{1} \rhd X = \mathbf 0 \rhd (\mathbf 0 \rhd X) = T T X
}
$$

The augmented simplicial category $\Delta_+$ becomes a monoidal category with the reverse tensor product, given explicitly on objects by
$$
\mathbf{m \oplus^{\operatorname{rev}} n = n \oplus m}.
$$
We denote this by $\Delta_+^{\operatorname{rev}}$. There is an isomorphism $\Delta_+^{\operatorname{rev}} \cong \Delta_+$ of monoidal categories~\cite{woohoo}, which of course carries over to the dual categories $\cA^{\operatorname{rev}} \cong \cA$. Right actions of $\cA$ correspond by to right actions of $\cA^{\operatorname{rev}}$ by the monoidal isomorphism, and these in turn correspond to left actions of $\cA$, which correspond to comonads by Proposition~\ref{leftaction}.
Thus we have:
\begin{prop}
Strict right actions
$$\xymatrix{
\cx \times \cA \ar[r] & \cx
}$$
correspond to comonads on $\cx$.
\end{prop}
Given a right action $\lhd$, the corresponding comonad $S$ is explicitly defined in an analagous way to comonads corresponding to left actions.
\begin{lem}\label{whenwhen} Distributive laws of comonads correspond to 0-cells of $\AmodA$.
\end{lem}
\begin{proof}
Suppose that $\cM$ is an object in $\AmodA$, so that we have natural morphisms
$$
\xymatrix{
\mathbf{m} \rhd (M \lhd \mathbf n) \ar[r]^-{\chi} & (\mathbf{m} \rhd M) \lhd \mathbf n.
}
$$
By choosing $\mathbf{m} = \mathbf{n} = \mathbf 0$, we get morphisms
$$
\xymatrix{
T S(M)  \ar[r]^-\chi & ST(M)
}
$$
natural in $M$. The naturality of $\chi$ in $\mathbf{m},\mathbf n$, combined with compatibility with the unit object $\mathbf{-1}$, tells us that $\chi$ is compatible with the comultiplication and counits of both comonads, realising it as a distributive law.

Conversely, let $\chi$ be a distributive law with components
$$
\xymatrix{
T S(M)  \ar[r]^-\chi & ST(M).
}
$$
These morphisms can easily be upgraded by iterating $\chi$ up to horizontal composition of identities, giving morphisms $T^nS^m (M) \to S^mT^n(M)$ for any $m,n \ge 0$ \`a la Section~\ref{evidenziatore1}. By definition, the coherence conditions required for $\chi$ to turn $\cM$ into an object of $\AmodA$ are satisfied.
 \end{proof}
 \subsection{(Co)homology}
 For this subsection, fix a distributive law $\chi \colon TS\Rightarrow ST$. Let $\cM$ denote the $0$-cell of $\AmodA$ corresponding to $\chi$ under Lemma~\ref{whenwhen}.
 \begin{lem}\label{chicoalghoch}
 An object of $\rH^0(\cA, \cM)$ corresponds to a $\chi$-coalgebra in $\cM$.
 \end{lem}
 \begin{proof}
 Let $\rho \colon TM \Rightarrow SM$ be a $\chi$-coalgebra structure on an object $M$ in $\cM$. By Proposition~\ref{simplicialrho}, $\rho$ extends to a morphism of simplicial objects
 $$
\xymatrix{
\rBB(T,M) \ar[r]^-{\rho} & \rBB^*(S,M),
}$$
and in fact extends to a morphism of augmented simplicial objects by setting $\rho$ to be $1 \colon M \to M$ in degree $-1$. Then, by construction, the pair $(M,\rho)$ is an object of $\rH^0(\cA, \cM)$.

Conversely, given an object $(M, \rho)$ in $\rH^0(\cA, \cM)$, the morphism $\rho$ in degree 0 gives a $\chi$-coalgebra structure on $M$.
 \end{proof}
 Similarly, we have the following:
 \begin{lem}
Functors $\rH_0(\cA, \cM) \to \cY$ correspond to $\chi$-opcoalgebras $\cM \to \cY$. \end{lem}
\subsection{The d\'ecalage comonads}\label{decalage}

Let $X$ be an augmented simplicial object in a category $\cY$.
\begin{defn}
The \emph{right d\'ecalage} (French for \emph{shift}) of $X$, denoted $\Decr(X)$, is the simplicial object given in degree $n$ by
$$\Decr(X)_n = X_{n+1},$$ whose faces are given by discarding the last face of $X$, and similarly for the degeneracies.
\end{defn}
See~\cite{MR2399898,MR3065943} for  more on d\'ecalage.
Pictorially (ignoring the degeneracies), the right d\'ecalage of an augmented simplicial object
$$
\xymatrix{
\cdots X_2 \ar@<2.5ex>[r]^-{d_0} \ar[r]^-{d_1}  \ar@<-2.5 ex>[r]^-{d_2}& X_1 \ar@<1.25ex>[r]^-{d_0} \ar@<-1.25ex>[r]^-{d_1} & X_0 \ar[r]^-{d_0} & X_{-1}
}
$$
looks like
$$
\xymatrix{
\cdots X_3 \ar@<2.5ex>[r]^-{d_0} \ar[r]^-{d_1}  \ar@<-2.5 ex>[r]^-{d_2}& X_2 \ar@<1.25ex>[r]^-{d_0} \ar@<-1.25ex>[r]^-{d_1} & X_1  \ar[r]^-{d_0} & X_{0}
\\ \\}
$$
In fact, $\Decr$ is a comonad on $[\cA, \cY]$, where the counit $\epsilon  \colon \Decr(X) \to X$ is given in each degree by the missing face map, and the comultiplication $\delta \colon \Decr(X) \to \Decr \Decr(X)$ is given by the missing degeneracy.

In a similar way, we define the left d\'ecalage of an augmented simplicial set by discarding the zeroth face and degeneracy maps. We denote this by $\Decl$. Since $\Decr\Decl$ is naturally equal to $\Decl \Decr$, we get the following:

\begin{prop}\label{decalagep}
  The identity natural transformation defines a distributive law $$\xymatrix{{\Decr} {\Decl} \ar@{=>}[r]^-\chi &  {\Decl} {\Decr}.}$$
\end{prop}
Now, let us view this distributive law on $[\cA, \cY]$ as an object of $\AmodA$.
\begin{thm}\label{ultimate}
The category $\rH^0(\cA, [\cA, \cY] )$ is isomorphic to the category of augmented duplicial objects in $\cY$.
\end{thm}
\begin{proof}
By Lemma~\ref{chicoalghoch}, to give an object $(X, t)$ of $\rH^0(\cA, [\cA, \cY] )$ is the same as to give a $\chi$-coalgebra in $[\cA, \cY]$. This is a morphism $t \colon \Decr X \to \Decl X$, so we get an operator
$$
\xymatrix{
X_n \ar[r]^-t & X_n
}
$$ in each degree. That $t$ commutes with the faces and degeneracies is precisely that the equations
\begin{align*}
d_i t = t d_{i-1}, \qquad s_i t = ts_{i-1}
\end{align*}
hold, and that $t$ is compatible with $\epsilon$ and $\delta$ as a $\chi$-coalgebra structure is precisely that the equations
$$
d_i t = d_n, \qquad s_it = t^2s_n
$$
respectively hold; that is, $t$ is a duplicial operator.
\end{proof}
\begin{rem}\label{remrem}
The category $\Delta_+$ becomes both a left-module and right-module over itself with actions given by the monoidal product $\oplus$. These actions are strictly compatible, and moreover they restrict to the simplicial category $\Delta$. This carries over to the dual categories, and so both $\Delta^*$ and $[\Delta^*, \cY]$ become $0$-cells in $\AmodA$. We can then consider the (co)homology categories
$$
\rH_0(\cA, \Delta^*), \qquad \rH_0(\cA, [\Delta^*, \cY]),
$$
the latter of which is isomorphic to the category of unaugmented duplicial objects in $\cY$.
 \end{rem}
Thus, we have another way of deriving the duplicial functor at the heart of B\"ohm and {\c S}tefan's Theorem~\ref{dup}. Indeed, let $M$ be a $\chi$-coalgebra in $\cM$, viewed as an object of $\rH^0(\cA, \cM)$. Let $N \colon \cM \to \cY$ be a $\chi$-opcoalgebra, viewed as a functor $\rH_0(\cA, \cM) \to \cY$. Applying the cap product
$$
\xymatrix{
\rH^0(\cA, \cM) \ar[r]^-{\cap} & [\rH_0(\cA, \cA), \rH_0(\cA, \cM)]
}
$$
to $M$ gives us a functor
$$
\xymatrix{ \rH_0(\cA, \cA) \ar[rr]^-{\cap(M)} && \rH_0(\cA, \cM)}
$$
that we may compose with $N$ to obtain a functor
$$
\xymatrix{ \rH_0(\cA, \cA) \ar[rr]^-{X} && \cY}.
$$
This is of course an object in $[\rH_0(\cA, \cA), \cY]$ which is isomorphic to $\rH^0(\cA, [\cA, \cY])$ by Theorem~\ref{hochgen}. By Theorem~\ref{ultimate}, this is in turn isomorphic to the category of augmented duplicial objects in $\cY$. By replacing $\cM$ with $[\cX, \cM]$ for a general category $\cX$, we recover the duplicial functor $\rCC_T(N,M)$ of Theorem~\ref{dup}.

\subsection{Duplicial objects from twisting}\label{bonusect}
Here, we show that Theorem~\ref{dup} can be deduced directly from the twisting procedure given in Section~\ref{twistcoeff}. Recall from Proposition~\ref{decalagep} that for an arbitrary category $\cY$, the identity defines a distributive law $1\colon \Decr\Decl \Rightarrow \Decl\Decr$.

Suppose we have a distributive law $\chi \colon TS \Rightarrow ST$ of comonads on a category $\cB$, together with a $\chi$-opcoalgebra $(N, \lambda) \colon \cB \to \cY$.
Consider the functor
$$
\xymatrix{
\cB \ar[rr]^-{\rC_T(N,{-})} && [\cA, \cY]
}
$$
that assigns to every object $X$, the simplicial object $\rC_T(N,X)$ of Section~\ref{duplobjsec}.

For every $n\ge 0$, we have
$$
C_T(N, SX)_n = NT^{n+1} SX, \qquad \Decl \rC_T(N, X)_n = \rC_T(N, X)_{n+1} = NT^{n+2} X
$$
The morphisms
$$
\xymatrix{
NT^{n+1}SX \ar[rr]^-{N \chi^{n+1}_X} & & NST^{n+1}X \ar[rr]^-{\lambda_{T^{n+1} X}}  & & NT^{n+2}X
}
$$
then define a natural transformation
$$
\xymatrix{
\rC_T(N,{-}) \circ S \ar@{=>}[r]^-{\gamma} & {\Decl} \circ~\rC_T(N, {-}).
}
$$
We also have
$$
{\Decr} \rC_T(N, X)_n = \rC_T(N, X)_{n+1} = NT^{n+2}X, \qquad \rC_T(N, TX)_n = NT^{n+2} X
$$
and it turns out that the identity defines a natural transformation
$$
\xymatrix{
{\Decr}\circ~\rC_T(N,{-}) \ar@{=>}[r]^-{1} & \rC_T(N, {-}) \circ T.
}
$$

It is routine to check the following:
\begin{prop}\label{ultimate1}
We have that
$$
\xymatrix{
(\cB, \chi, T, S) \ar[rrr]^-{(\rC_T(N, -), 1, \gamma)} & & & ([\cA, \cY], 1, \Decr, \Decl)
}
$$
is a 1-cell in $\dist$.
\end{prop}

Now let $M$ be a $\chi$-coalgebra in $\cB$. The 1-cell in Proposition~\ref{ultimate1} acts on this to give a $\chi$-coalgebra in $[\cA, \cY]$ which is precisely an augmented duplicial object in $\cY$ by Lemma~\ref{chicoalghoch}. Given a $\chi$-coalgebra structure on a functor $M \colon \cX \to \cB$, we recover the more general duplicial functor of Theorem~\ref{dup} by carrying out the replacements
 $$
\xymatrix@=1em{
\cB \ar@{|->}[r] & [\cX, \cB], & \cY \ar@{|->}[r] & [\cX, \cY], & N \ar@{|->}[r] & [\cX, N], \\
& T \ar@{|->}[r] & [\cX, T], & S \ar@{|->}[r] & [\cX, S]
}
$$
in the above arguments.

\subsection{Restriction to $\cat$}\label{nerve}
In this final section of the chapter, we view $\cat$ as a 0-cell of $\AmodA$ via the nerve functor, and study what it means to be an object of the cohomology category $\rH^0(\cA, \cat)$. Throughout, $\cC$ denotes a (small) category.

\begin{defn}
The \emph{nerve of $\cC$}, denoted by $N \cC$, is the simplicial set where $(N\cC)_n$ is the set of composable $n$-tuples of morphisms. In degree $n$, the faces and degeneracies, defined on an $n$-tuple
$$
\xymatrix{
A_0 \ar[r]^-{f_0} & A_1 \ar[r]^-{f_1} & \cdots  \ar[r]^-{f_{n-2}} & A_{n-1} \ar[r]^-{f_{n-1}} & A_{n}
}
$$
are given by
\begin{align*}
d_i (f_{n-1}, \ldots, f_{0} ) &= \begin{cases} (f_{n-1}, \ldots, f_1) & \mbox{ if } i=0 \\
(f_{n-1}, \ldots, f_i f_{i-1}, \ldots, f_0) & \mbox{ if } 0 < i < n \\
(f_{n-2}, \ldots, f_0) & \mbox{ if } i = n \end{cases} \\
s_i (f_{n-1}, \ldots, f_{0} ) &= (f_{n-1}, \cdots, f_i, 1, f_{i-1}, \cdots, f_0 ).
\end{align*}
\end{defn}
The nerve actually defines a fully faithful functor $N \colon \cat \to [\Delta^*, \set]$ (see e.g.~\cite[II.4.22]{MR1950475}) so we may view $\cat$ as a full subcategory of the category of simplicial sets. We do not give precise details, but the actions of $\cA$ on $[\Delta^*, \set]$ mentioned in Remark~\ref{remrem} restrict to actions on $\cat$, and so $\cat$ becomes a 0-cell in $\AmodA$ in its own right (for a more detailed explanation, see~\cite{woohoo}). We are then able to consider the cohomology category $\rH^0(\cA, \cat)$. Explicitly, an object of this category is a (small) category $\cC$ together with a duplicial structure on its nerve $N\cC$.

To give a duplicial structure on $N\cC$ amounts to giving the following structure on $\cC$:
\begin{itemize}
\item for each object $A$, an object $tA$
\item for each morphism $f\colon A \to B$ a morphism $tf \colon tB \to A$
\end{itemize}
such that $t^2(1_A) = 1_{tA}$ for objects $A$ of $\cC$, and $t$ maps any commutative triangle of morphisms as
$$
\begin{array}{c}
\xymatrix{
A \ar[r]^-f \ar[dr]_-{h} & B \ar[d]^-g \\
& C
}\end{array}
\mapsto
\begin{array}{c}
\xymatrix{
tC \ar[r]^-{th} \ar[dr]_-{tg} & A \ar[d]^-f \\
& B
}
\end{array}
$$
\begin{defn}\label{corefdef}
We call a morphism $t$ as described above a \emph{coreflector}.

\end{defn}
The following theorem characterises such structures, and Corollary~\ref{motivation} makes clear the motivation for the terminology in Definition~\ref{corefdef}.
\begin{thm}\label{auspara}
A category $\cC$ has a duplicial structure on its nerve if and only if there exist a groupoid $\cG$ and an adjunction
$$
\xymatrix{
\cG \ar@{}[rr]|-{\perp}\ar@/^0.5pc/[rr]^-I & & \ar@/^0.5pc/[ll]^-R \cC.
}$$
\end{thm}
\begin{proof}
Suppose that we are given such an adjunction, with unit $\eta$ and counit $\epsilon$. For an object $A$ in $\cC$ we define $tA  := IRA$, and for a morphism $f\colon A \to B$ in $\cC$ we define $tf \colon tB \to A$ to be the composite
$$
\xymatrix{
IRB \ar[rr]^-{I(Rf)^{-1}} & &I RA \ar[r]^-{\epsilon_A} & A.
}
$$
Note that $tf$ is well defined because $Rf$ is invertible, being a morphism in a groupoid.

Consider an identity morphism $1 \colon A\to A$. We have that $t^2(1)$ is equal to
$$
\xymatrix{
IRA \ar[rr]^-{IR(\epsilon_A)^{-1}} && IRIRA \ar[rr]^-{\epsilon_{IRA}} & & IRA.
}
$$
One triangle identity for $I \dashv R$ implies that the diagram
$$
\xymatrix{RA \ar@{=}[dr] \ar[r]^-{\eta_{RA}} & RIRA \ar[d]^-{R\epsilon_A} \\
& RA}
$$
commutes, which
tells us that $R(\epsilon_A)^{-1} = \eta_{RA}$. Therefore, using the other triangle identity, we have that the diagram
$$
\xymatrix@C=3.5em{
IRA \ar@{=} [dr]\ar[r]^-{IR(\epsilon_A)^{-1}}   & IRIRA \ar[d]^-{\epsilon_{IRA}} \\
& RA
}
$$
commutes, that is $t^2(1) = 1$.

Now suppose that we have a commutative triangle
$$\xymatrix{
A \ar[r]^-f \ar[dr]_-{h} & B \ar[d]^-g \\
& C
}$$
Since $h = gf$ we have $(Rh)^{-1} = (Rf)^{-1} (Rg)^{-1}$ and so $Rf (Rh)^{-1} = (Rg)^{-1}$. In the diagram
$$
\xymatrix@C=3.5em{
IRC \ar[r]^-{I(Rh)^{-1}} \ar[dr]_-{I(Rg)^{-1}\ } & IRA \ar[d]^-{IRf}  \ar[r]^-{\epsilon_A} & A \ar[d]^-f \\
& IRB \ar[r]_-{\epsilon_B} & B
}
$$
the inner triangle commutes by the aforementioned relation, and the right-hand square commutes by naturality of $\epsilon$. The outer diagram commutes, which says that $tg = f(th)$, completing one direction of the proof.

Conversely, suppose that we have a coreflector $t$ giving rise to a duplicial structure on $N\cC$. There is an induced functor $G \colon \cC \to \cC$, defined on objects and morphisms respectively by
$$
G(A) := tA, \qquad G\!\xymatrix{(A \ar[r]^-f & B)} := \xymatrix{tA \ar[r]^-{t^2 f} & tB.}
$$
Let $f\colon A \to B$ be any morphism in $\cC$. By applying $t$ to two commutative triangles as follows
$$
\begin{array}{c}
\xymatrix{
A \ar[dr]_-f \ar[r]^-f & B \ar[d]^-1 \\
& B
}
\end{array}
\mapsto
\begin{array}{c}
\xymatrix{
tB \ar[dr]_-{t1} \ar[r]^-{tf} & A \ar[d]^-f \\
& B
}
\end{array}, \qquad
\begin{array}{c}
\xymatrix{
tB \ar[dr]_-{tf} \ar[r]^-{tf} & A \ar[d]^-1 \\
& A
}
\end{array}
\mapsto
\begin{array}{c}
\xymatrix{
tA \ar[dr]_-{t1} \ar[r]^-{t^2f} & B \ar[d]^-{tf} \\
& B
}
\end{array}
$$
we see that the diagram
$$
\xymatrix{
tA\ar[d]_-{t^2 f} \ar[r]^-{t1} & A \ar[d]^-f \\
tB\ar[ur]_-{tf} \ar[r]_-{t1} & B
}
$$
 commutes. By defining $\epsilon_A := t\!\xymatrix{(A \ar[r]^-1 & A)}$, we therefore obtain a natural transformation $\epsilon \colon G \Rightarrow 1$. Also, together with $\epsilon$, $G$ becomes a well-copointed endofunctor: that is, for any $A$ we have
$$
G(\epsilon_A) = t^2(\epsilon_A) = t^3(1_A) = t(1_{tA} ) = \epsilon_{GA}.
$$

Let $f\colon A \to B$ be any morphism. By applying $t$ twice:
$$
\begin{array}{c}
\xymatrix{
t A \ar[dr]_-{f\epsilon_A} \ar[r]^-{\epsilon_A} & A \ar[d]^-f \\
& B
}
\end{array}
\mapsto
\begin{array}{c}
\xymatrix{
tB \ar[dr]_-{tf} \ar[r]^-{t(f\epsilon_A)} & tA \ar[d]^-{\epsilon_A} \\
& A
}
\end{array}
\mapsto
\begin{array}{c}
\xymatrix{
tA  \ar[dr]_-{t(\epsilon_A)} \ar[r]^-{t^2 f} & tB \ar[d]^-{t(f\epsilon_A)} \\
& t A
}
\end{array}
$$
we have that $$t(f\epsilon_A)\circ t^2f = t(\epsilon_A) = 1.$$ Furthermore, by naturality of $\epsilon$ we have $$f\epsilon_A = \epsilon_B \circ t^2 f$$ and applying $t$ once more:
$$
\begin{array}{c}
\xymatrix{
tA  \ar[dr]_-{f\epsilon_A} \ar[r]^-{t^2 f} & tB \ar[d]^-{\epsilon_B} \\
& B
}
\end{array}
\mapsto
\begin{array}{c}
\xymatrix{
tB \ar[dr]_-{t(\epsilon_B)} \ar[r]^-{t(f\epsilon_A)} & tA \ar[d]^-{t^2 f} \\
& tB
}
\end{array}
$$
tells us that $$
t^2f \circ t(f\epsilon_A) = t^2 f \circ t(\epsilon_B \circ t^2 (f) ) = 1.
$$
Thus $Gf = t^2f$ is invertible for any morphism $f$. It follows that $G$ becomes an (idempotent) comonad with counit $\epsilon$ and comultiplication $ \delta := G(\epsilon)^{-1}$. The groupoid $\cG$ that we require is the Eilenberg-Moore category $\cC^G$, together with the canonical associated adjunction (cf.\ Section~\ref{EM2cat}), which is indeed a groupoid since $G$ inverts morphisms.
\end{proof}
Later, in Section~\ref{finally}, we see that this example of a duplicial structure can be derived from a special case of one involving enriched functor categories.

\begin{cor}\label{motivation}
A category $\cC$ admits a duplicial structure on its nerve if and only if $\cC$ has a coreflective subcategory which is a groupoid.
\end{cor}
\begin{proof}
Given two categories $\cG, \cC$ and a left-adjoint functor $I \colon \cG \to \cC$, the unit $\eta \colon 1 \Rightarrow RI$ is necessarily an isomorphism so $I$ is full and faithful~\cite[Prop.~3.4.1]{MR1291599}. Thus, the corollary follows immediately from Theorem~\ref{auspara}.
\end{proof}

\begin{rem}\label{catrem}
Suppose that we have two groupoids $\cG$ and $\cG'$, as well as two adjunctions
$$
\xymatrix{
\cG \ar@{}[rr]|-{\perp}\ar@/^0.5pc/[rr]^-I & & \ar@/^0.5pc/[ll]^-R \cC,
}\qquad
\xymatrix{
\cG' \ar@{}[rr]|-{\perp}\ar@/^0.5pc/[rr]^-{I'} & & \ar@/^0.5pc/[ll]^-{R'} \cC.}
$$
The functor
$$
\xymatrix{
\cG \ar[r]^-{I} & \cC \ar[r]^-{R'} & \cG' 
}
$$
defines an equivalence of categories $\cG \simeq \cG'$. In fact, in the context of Theorem~\ref{auspara}, we may always choose $\cG$ to be the fundamental groupoid $\Pi_1(\cC)$ (this is the localisation of $\cC$ at all morphisms~\cite{MR0210125}, see also~\cite{woohoo}). 
\end{rem}
Thus we have another way to state Theorem~\ref{auspara}:
\begin{cor}
A category $\cC$  admits a duplicial structure on its nerve if and only if the universal functor $\pi\colon \cC \to \Pi_1(\cC)$ has a left adjoint.
\end{cor}

\begin{cor}
A category $\cC$ admits a cyclic structure on its nerve if and only if $\cC$ is a groupoid.
\end{cor}
\begin{proof}
Suppose that a category $\cC$ is equipped with a cyclic structure on its nerve, induced by a coreflector $t$. By assumption, $t$ is the identity on objects, and $t^2$ is the identity on morphisms. By the proof of Theorem~\ref{auspara}, for all morphisms $f$ in $\cC$, the morphism $t^2 f = f$ is invertible and thus $\cC$ is a groupoid.

Conversely, if $\cC$ is a groupoid, then we apply Theorem~\ref{auspara} to the the identity functor $1\colon \cC \to \cC$, which is trivially a left-adjoint functor.
\end{proof}
\begin{rem}
Note that a groupoid $\cC$ may have multiple cyclic structures on its nerve. Indeed,  given a natural transformation
$$
\xymatrix{
\cC \rrtwocell^1_{1}{\ \alpha} &  & \cC
}
$$
we define a coreflector $t$ on objects and morphisms respectively by
$$
tA := A, \qquad t\!\xymatrix{(A \ar[r]^-f & B)} := \xymatrix{B \ar[r]^-{\alpha_B} & B \ar[r]^-{f^{-1}} & A.}
$$
\end{rem}
\begin{rem}
In~\cite{woohoo} we also characterise duplicial structure on the nerve of a bicategory~\cite{MR920944,MR1897816} (and so in particular, on the nerve of a monoidal category).
\end{rem}
\chapter{Examples}\label{EXAMPLES}
In this chapter we give examples with an algebraic flavour.
In Section~\ref{cychomalg} we explain how the cyclic homology of an associative algebra~\cite{MR823176,MR777584} arises as an instance of Theorem~\ref{dup}. We also show that an algebra map $\sigma$ induces a 1-cell in $\dist$ which acts on cyclic homology to give $\sigma$-twisted cyclic homology~\cite{MR1943179}. Next, in Section~\ref{brugsec} we show that Hopf-cyclic homology of bialgebroids (as in~\cite{MR2803876}) arises from Theorem~\ref{dup} and analyse this phenomenon in terms of opmonoidal adjunctions and their opmodules~\cite{MR3020336}. In Section~\ref{hopfywopfy} we study the various notions of bimonad and Hopf monad~\cite{MR3020336,MR2793022,MR1942328,MR3175323,MR1887157}, as well as give a new example of a bimonad that is not a Hopf monad (Section~\ref{newbimonad}). We conclude the chapter in Section~\ref{finally} by showing that the category of enriched functors from a Hopf category~\cite{Hop, MR1458415} contains a duplicial object under certain conditions.

The work contained herein is original; Section~\ref{cychomalg} is based on~\cite[\S4]{2} and Sections~\ref{brugsec} and~\ref{hopfywopfy} are based on\cite[\S5--6]{1}.
\section{Cyclic homology of algebras}\label{cychomalg}

\subsection{Flat connections}
Let $B$ be a monad on a category $\cA$, and let
$$
	\xymatrix{ \cA
\ar@/^{0.5pc}/[rr]^-F \ar@{}[rr]|-{\perp}&&
\ar@/^0.5pc/[ll]^U\cA^B
}
$$
be the canonical adjunction (cf.\ Sections~\ref{EM2cat}, \ref{emcatsect}). As before, let $\tilde B$ denote the comonad $FU$ on $\cA^B$ generated by the adjunction, and let $\Sigma \colon \cA^B \to \cA^B$ be an endofunctor.

For any $B$-algebra $(X, \beta)$, we have natural isomorphisms
$$
\cA^B ( \tilde B \Sigma (X,\beta), \Sigma \tilde B (X, \beta) ) \cong \mathcal \cA ( U\Sigma (X, \beta) , U \Sigma \tilde B (X, \beta) )
$$
given by the adjunction, so there is a one-to-one correspondence between natural transformations $\sigma \colon \tilde B \Sigma \Rightarrow \Sigma \tilde B$  and natural transformations $\nabla \colon U \Sigma \Rightarrow U \Sigma \tilde B$. In fact, $(\Sigma, \sigma)$ is an opmorphism of comonads if and only if the diagrams
$$
\xymatrix{
U\Sigma \ar[r]^-{\nabla} \ar[d]_-\nabla & U\Sigma \tilde B \ar[d]^-{\nabla \tilde B} \\
U \Sigma \tilde B \ar[r]_-{U \Sigma \tilde \delta} & U \Sigma \tilde B \tilde B
}
\quad\quad\quad
\xymatrix{
U \Sigma \ar[r]^-\nabla \ar@{=}[dr] & U \Sigma \tilde B \ar[d]^-{U \Sigma \tilde \epsilon} \\
& U \Sigma
}
$$
commute, that is, if $(U\Sigma, \nabla)$ is a $\tilde B$-opcoalgebra.
\begin{defn}
We say that the natural transformation $\sigma$ is a \emph{connection} if
$\epsilon $ is compatible with $\sigma$,
i.e.\ the second diagram above commutes for the
corresponding natural transformation $\nabla$. We say
that a connection $\sigma$ is \emph{flat} if $
\delta $ is compatible with $\sigma$, i.e.\ $(\Sigma,\sigma)$ is an opmorphism of comonads, or equivalently, both diagrams above commute.
\end{defn}

The terminology is motivated by the special case
discussed in detail in Section~\ref{connectexample}.

\subsection{Cyclic and twisted cyclic homology}
Let $A$ be a unital associative algebra over a commutative ring $k$. Consider the simplicial $k$-module $C_\bullet(A,A)$ of Section~\ref{hochsub}, defined degreewise by $$C_n(A,A) = A^{\otimes (n+1)}$$ which gives rise to the Hochschild homology module $\rH_\bullet(A, A)$. This simplicial object is in fact a cyclic object. Indeed, in degree $n$ we define a cyclic operator $t$ on $A^{\otimes (n+1)}$ by
$$
t(a_0 \otimes \cdots \otimes a_n) = a_1 \otimes \cdots \otimes a_n \otimes a_0.
$$
\begin{defn}
The \emph{cyclic homology of $A$}, denoted $\HC(A)$, is the cyclic homology of the cyclic object $C_\bullet(A,A)$.
\end{defn}

Now let $\sigma \colon A \to A$ be an algebra map. We define a simplicial module $C_\bullet (A, A)_\sigma$ as follows:
take the degeneracies and faces of the Hochschild simplicial object $C_\bullet(A,A)$ but in each degree, we change the last face $d_n \colon A^{\otimes(n+1)} \to A^{\otimes n}$ to the map $d^\sigma_n$ defined by
$$
d^\sigma_n (a_0 \otimes \cdots \otimes a_n) = (a_0 \sigma(a_1) , a_2, \cdots, a_n).
$$
In fact, this is a duplicial object with duplicial operator $t^\sigma$ defined by
$$
t^\sigma(a_0 \otimes \cdots \otimes a_n) = \sigma(a_1) \otimes \cdots \otimes a_n \otimes a_0.
$$
This is a cyclic object if and only if $\sigma = 1$.
\begin{defn}
The \emph{cyclic homology of $A$ twisted by $\sigma$}, denoted $\HC_\sigma(A)$, is the cyclic homology of the duplicial object $C_\bullet(A,A)_\sigma$.
\end{defn}
It is the aim of this section to show that the assignment
$$
\HC (A)\mapsto \HC_\sigma(A)
$$
can be realised in the abstract framework of the previous chapters.

\subsection{$(A,A)$-bimodules}\label{connectexample}
Let $\cA= \amod$ be the category of left $A$-modules and let $B$ be the comonad $= {-}\otimes A$, so that $\cA^B$ is the category of $(A,A)$-bimodules as in Section~\ref{hochclassic}.

The functor $C= A \otimes  {-} \colon \cA \to \cA$, together with the natural morphisms
\begin{align*}
\delta \colon A \otimes X &\lto A \otimes A \otimes X & \epsilon \colon A \otimes X & \lto X \\
a \otimes x &\lmapsto  a \otimes 1 \otimes x & a \otimes x & \lmapsto ax
\end{align*}
defines a comonad on $\cA^B$. There is a mixed distributive law $\theta \colon BC \Rightarrow CB$ given by rebracketing on components
$$
\theta \colon (A \otimes X) \otimes A \to A \otimes (X \otimes A)
$$
so by Corollary~\ref{arisec}, this lifts to a comonad distributive law $\tilde\theta \colon \tilde B C^\theta \Rightarrow C^\theta \tilde B$ on $\cA^B$.

Let $N$ be an $(A,A)$-bimodule and $\Sigma \colon
\cA^B \to \cA^B$ be the endofunctor defined by $\Sigma (M) = M \otimes_A N$. We have that $\Sigma  C^\theta = C^\theta \Sigma$ so that
$$
\xymatrix{
(\cA^B, C^\theta) \ar[rr]^-{(\Sigma, 1)} && (\cA^B, C^\theta)
}
$$
is a morphism of comonads.

The component of a natural transformation $\nabla \colon U \Sigma \Rightarrow U \Sigma \tilde B$ is given by a left $A$-linear map
\begin{align*}
  \nabla_M \colon M \otimes_A N & \to (M \otimes A) \otimes N \cong M \otimes N
\end{align*}
The natural transformation $\nabla$ defines a
connection if and only if each $\nabla_M$ splits the
quotient map $M \otimes N \to M \otimes_A N$.
Taking $M=A$ yields an $A$-linear splitting of the action $A
\otimes  N \to N$, so $N$ is $k$-relative
projective~\cite[Definition~8.6.5]{MR1269324}. Conversely, given a splitting
$n \mapsto n_{(-1)} \otimes n_{(0)}$ of the action, we
obtain $\nabla_M$ as
$\nabla_M(m \otimes_A n)=mn_{(-1)} \otimes n_{(0)}$.
Thus we have:

\begin{prop}
The functor $\Sigma$ admits a connection $ \sigma $ if and only if $N$ is
$k$-relative projective as a left $A$-module.
\end{prop}

Composing $\nabla_A$ with the noncommutative De Rham
differential
$$
	\operatorname d \colon A \lto
	\Omega^1_{A,k},\quad
	a \lmapsto 1 \otimes a-a \otimes 1
$$
gives the notion of connection in noncommutative
geometry \cite[III.3.5]{MR1303779}.

If $N$ is not just $k$-relative projective but
$k$-relative free, i.e.\
$N \cong A \otimes V$ as left $A$-modules, for some
$k$-module $V$, then the assignment $\nabla_M ( m \otimes_A (a \otimes v) ) = ma \otimes (1 \otimes v)$ defines a flat connection. Explicitly, the flat connection $\sigma \colon\tilde B \Sigma \Rightarrow \Sigma \tilde B$ is given by
\begin{align*}
\sigma_M \colon (M \otimes_A N) \otimes A &\lto (M \otimes A) \otimes_A N \cong M \otimes N \\
 (m \otimes_A n) \otimes b &\lmapsto \nabla_M (m \otimes_A n)b.
\end{align*}

Thus we have:
\begin{prop}\label{bimoduletwist}
The triple
$(\Sigma, \sigma, 1)$ defines a 1-cell
$$
\xymatrix{
(\cA^B, \tilde\theta, \tilde B, C^\theta) \ar[rr]^-{(\Sigma, \sigma, 1)}  & & (\cA^B, \tilde\theta, \tilde B, C^\theta)
}
$$
in $\dist$.
\end{prop}

In particular, let $ \sigma \colon A \to A$ be an
algebra map and let $N=A_\sigma$, the $(A,A)$-bimodule
which is $A$ as a left $A$-module with right action of $a
\in A$ given by right multiplication by $ \sigma
(a)$.
Then we have
$\Sigma (M) = M \otimes_A A_\sigma \cong M_\sigma $.
Since $A_\sigma$ is free as a left $A$-module we get a 1-cell
$(\Sigma, \sigma,1)$ by Proposition~\ref{bimoduletwist},
where $\sigma \colon \tilde{B} \Sigma \Rightarrow \Sigma\tilde B$ is the flat connection defined on components by
\begin{align*}
\sigma_M \colon M_\sigma \otimes_k A & \lto (M \otimes_k A)_\sigma \\
m \otimes a &\lmapsto m \otimes \sigma(a).
\end{align*}
Note that we use $ \sigma $ to denote both the algebr
map and the flat connection it induces.
\subsection{Twisting by 1-cells}
From the general theory developed in
Section~\ref{twistcoeff} we obtain therefore an action of
the group of endomorphisms of $A$ on the categories $\sR(\tilde\theta)$ and $\sL(\tilde\theta)$ of $\tilde\theta$-coalgebras and $\tilde\theta$-opcoalgebras, respectively.
In particular, we can act
on the standard cyclic object $C_\bullet(A,A)$ associated to $A$, which corresponds to
the following data.

Consider $A$ as an object of $\cA^B$. Since $\tilde B A = C^\theta A = A \otimes A$ we have a morphism $\rho = 1 \colon \tilde B A \to C^\theta A$. It follows that $(A, \rho)$ is a $\tilde\theta$-coalgebra in $\cA^B$.

Considering $(A,A)$-bimodules as either left or
right
$\Ae = A \otimes A^*$-modules, we view the zeroth
Hochschild homology as a functor
$H = - \otimes_{\Ae} A \colon
\cA ^ B \to \kmod$. Then, as in Section~\ref{hochclassic}, we have an isomorphism of simplicial objects
$$
C_\bullet(A,A) \cong \rC_{\tilde B}(H, A).
$$
We define a natural transformation
$\lambda \colon H  C^\theta  \Rightarrow H \tilde B$ by
\begin{align*}
\lambda_M \colon ( A\otimes M) \otimes_{\Ae} A &\lto (M \otimes A )\otimes_{\Ae} A \cong M \\
 (a \otimes m) \otimes_{\Ae} b  &\lmapsto mba
\end{align*}
Therefore, by Theorem~\ref{dup}, $C_\bullet(A,A)$ becomes a duplicial $k$-module, and indeed is
 the cyclic object defining the cyclic homology
$\HC(A)$.

By acting on the $\tilde\theta$-coalgebra $M$ with the $1$-cell $(\Sigma, \sigma, 1)$ and applying Theorem~\ref{dup}, we obtain a duplicial structure on the simplicial $k$-module
 $$
 C_\bullet(A,A)_\sigma
 $$
 whose cyclic homology is the $\sigma$-twisted cyclic homology of $A$.
Thus the action of 1-cells in $\dist$
generalises this twisting procedure.

\section{Hopf-cyclic homology}\label{brugsec}
In this section we assume that the monoidal categories mentioned and their actions on other categories are strict, with the symbols
$
I, \otimes,  \rhd
$
used to denote the unit, tensor product, and left action respectively.
\subsection{Opmonoidal adjunctions}\label{notsure}

Let $\cE, \cH$ be monoidal categories.
\begin{defn}
An \emph{opmonoidal functor} $\cE \to \cH$ is a triple $(P, \Xi, \Xi_0)$ where $P \colon \cE \to \cH$ is a functor,
$$\Xi \colon P({-}_{1} \otimes {-}_{2} ) \Rightarrow P({-}_{1}) \otimes P({-}_{2})$$
is a natural transformation, and $\Xi_0 \colon PI \to I$ is a morphism, such that the three diagrams
$$\xymatrix{
P(X \otimes Y \otimes Z) \ar[rr]^-{\Xi} \ar[d]_-{\Xi} && P(X \otimes Y) \otimes P(Z) \ar[d]^-{\Xi \otimes PZ} \\
PX \otimes P(Y \otimes Z) \ar[rr]_-{PX \otimes {\Xi}} && PX \otimes PY \otimes PZ
}$$
$$
\xymatrix{
P(X \otimes I) \ar[r]^-\Xi \ar@{=}[dr] & PX \otimes PI \ar[d]^-{PX \otimes \Xi_0} \\
& PX
} \qquad
\xymatrix{
P(I \otimes X) \ar[r]^-\Xi \ar@{=}[dr] & PI \otimes PX \ar[d]^-{\Xi_0 \otimes PX} \\
& PX
}
$$
commute.
\end{defn}
Let $(P, \Xi, \Xi_0)$ and $(G, \Phi, \Phi_0)$ be two monoidal functors $\cE \to \cH$.

\begin{defn}
An \emph{opmonoidal natural transformation} $\alpha \colon (P, \Xi, \Xi_0) \Rightarrow (G, \Phi, \Phi_0)$ is a natural
transformation $\alpha \colon P \Rightarrow G$ such that the two diagrams
$$
\xymatrix{
P(A \otimes B) \ar[r]^-{\Xi} \ar[d]_-\alpha & PA \otimes PB \ar[d]^-{\alpha \otimes \alpha} \\
G(A \otimes B) \ar[r]_-{\Phi} & GA \otimes GB
}\qquad
\xymatrix{
PI \ar[r]^-{\alpha} \ar[dr]_-{\Xi_0} & GI \ar[d]^-{\Phi_0} \\
& I
}
$$
commute.
\end{defn}

This defines a 2-category $\mathbf{OpMonCat}$ of monoidal categories, opmonoidal functors, and opmonoidal natural transformations~\cite[Example~2.4]{MR2664622}. One example of Corollary~\ref{arisec} is provided by an adjunction in this 2-category. Explicitly:

\begin{defn}
An \emph{opmonoidal adjunction}
$$
	\xymatrix{ \cE
\ar@/^{0.5pc}/[rr]^-{(P, \Xi, \Xi_0)} \ar@{}[rr]|-{\perp}&&
\ar@/^0.5pc/[ll]^-{(Q, \Psi, \Psi_0)} \cH
}
$$
consists of an adjunction $P \dashv Q$ between the underlying categories, such that the unit and counit are opmonoidal natural transformations.
\end{defn}

Opmonoidal adjunctions are a special case of doctrinal adjunctions, so it follows that $\Psi$ and $\Psi_0$ as above are in fact isomorphisms~\cite[Theorem~1.4]{MR0360749}. Some authors call opmonoidal adjunctions \emph{comonoidal adjunctions} or
\emph{bimonads}. We refer {e.g.} to~\cite{MR3020336,MR2793022,MR1942328,MR3175323,MR1887157}
for more information.

It follows that
$$
	PI \otimes {-}
	\qquad QPI
	\otimes {-}
$$
form a compatible pair of comonads
as in Corollary~\ref{arisec}
whose comonad structures are induced by the natural
coalgebra (comonoid) structures on $I$.

\subsection{Opmodule adjunctions}\label{opmoduleadj}

The examples we are more interested in are
given by opmodule adjunctions, as defined below. They were introduced under the
name comodule adjunctions in
\cite[Definition~4.1.1]{MR3020336}.

Let $(P, \Xi, \Xi_0) \colon \cE
\rightarrow \cH$ be an
opmonoidal functor and let $\cA$ and
$\cB$ be (left)
$\cE$-module $\cH$-module categories, respectively.

\begin{defn}

A $(P, \Xi, \Xi_0)$-\emph{opmodule} is a functor
$F \colon \cA \rightarrow \cB$ together with a
natural transformation
$$
\Theta \colon F({-}_1 \rhd {-}_2) \Rightarrow
P ({-}_1) \rhd F({-}_2)
$$ such that the diagrams
$$
	\xymatrix{
	\rF((X \otimes Y) \rhd  Z)\ar@{=}[d]
\ar[rr]^-\Theta & & P(X \otimes Y) \rhd FZ \ar[d]^-{\Xi \otimes FZ} \\
F(X \rhd (Y \rhd Z) )  \ar[d]_-\Theta & & (PX \otimes PY ) \rhd FZ \ar@{=}[d] \\
	PX \rhd F(Y \rhd Z)  \ar[rr]_-{PX \rhd \Theta} & &  PX \rhd (PY \rhd FZ) \\
}
$$
and
$$
	\xymatrix{	
	F(I \rhd Z) \ar[r]^-\Theta
	\ar@{=}[d] &
	PI \rhd F Z \ar[d]^-{\Xi_0
\rhd FZ}\\
	 FZ \ar@{=}[r] &
	I \rhd F Z}
$$
commute.
\end{defn}
Let $(P, \Xi, \Xi_0)$ and $(G, \Phi, \Phi_0)$ be two opmonoidal functors $\cH \to \cE$, and let $(F, \Theta)$ and $(K, \kappa)$ be two opmodules $\cA \to \cB$ over $(P, \Xi, \Xi_0)$ and $(G, \Phi, \Phi_0)$ respectively.
\begin{defn}
An \emph{opmodule morphism} $(\alpha, \beta) \colon (F, \Theta) \Rightarrow (K, \kappa)$ consists of an opmonoidal natural transformation
$$\alpha \colon (P, \Xi, \Xi_0) \Rightarrow (G, \Phi, \Phi_0) $$
and a natural transformation
$$
\beta \colon F \Rightarrow K
$$
 such that the diagram
 $$
 \xymatrix{
 F(X \rhd Z) \ar[r]^-{\Theta} \ar[d]_-\beta \ar[r]^-{\Theta} & PX \rhd FZ \ar[d]^-{\alpha \rhd \beta} \\
 K(X \rhd Z) \ar[r]_-{\kappa} & GX \rhd KZ
 }
 $$
 commutes.
\end{defn}
    This defines a 2-category $\mathbf{OpMod}$ whose 0-cells are actions of a monoidal category on another category, whose 1-cells are opmodules over opmonoidal functors, and whose 2-cells are opmodule morphisms~\cite[Remark~4.3]{MR3020336}.  We now give an explicit definition of adjunctions in this 2-category:
\begin{defn}
An \emph{opmodule
adjunction}
$$
\xymatrix{ \cA
\ar@/^{0.5pc}/[rr]^-{(F,\Theta)} \ar@{}[rr]|-{\perp}&&
\ar@/^0.5pc/[ll]^-{(U, \Omega)} \cB
}
$$
over an opmonoidal adjunction
$$
\xymatrix{ \cE
\ar@/^{0.5pc}/[rr]^-{(P, \Xi, \Xi_0)} \ar@{}[rr]|-{\perp}&&
\ar@/^0.5pc/[ll]^-{(Q, \Psi, \Psi_0)} \cH
}
$$
 is an adjunction $F
\dashv U$ such that
$$
	\xymatrix{
	X \rhd Z
	\ar[d]_-{\eta}
	\ar[rr]^-{\eta \rhd \eta}
	 & &
	QP X \rhd U F Z \\
	UF (X\rhd Z) \ar[rr]_-{U \Theta} & &
	U(P X \rhd F Z) \ar[u]_-{\Omega}}
$$
and
$$
	\xymatrix{
	P Q L \rhd FU M
	\ar[rr]^-{\varepsilon \rhd
	\varepsilon} & &
	L \rhd M\\
	F(Q L \rhd U M) \ar[u]^\Theta & &
	F U (L \rhd M) \ar[ll]^-{F \Omega}
	\ar[u]_-{\varepsilon}}
$$
commute.
\end{defn}

By the theory of doctrinal adjunctions, it follows that
$\Omega$ is an isomorphism
(see \cite[Proposition~4.1.2]{MR3020336} and again
\cite[Theorem~1.4]{MR0360749}).

Now any coalgebra $D$ in $\cH$
defines a compatible pair of comonads
$$
	S= D \rhd{-},\quad
	C=Q D \rhd{-}$$
on $\cB$
and $\cA$ respectively. It is such an
instance of Corollary~\ref{arisec} that provides the
monadic generalisation of the setting from
\cite{MR2803876}, see Section~\ref{coeffsforhopf}.

\subsection{Bialgebroids and Hopf algebroids}
Opmonoidal
adjunctions can be seen as categorical
generalisations of bialgebras and more generally
(left) bialgebroids. We briefly recall the definitions
but refer to \cite{MR2553659,MR1458415,MR2803876,MR1984397} for further
details and references.

\begin{defn}
If $E$ is a $k$-algebra,
then an \emph{$E$-ring} is a $k$-algebra map $ \eta : E
\rightarrow H$.
\end{defn}

In particular, when $E=\Ae:=A \otimes A^*$ is the
enveloping algebra of a $k$-algebra $A$, then
$H$ carries two left actions $\lact, \blact$ and two right actions $\ract, \bract$ of $A$, given by
$$
	a \lact h \ract
	b:=\eta(a \otimes b)h,\quad
	a \blact h \bract b:=
	h \eta(b \otimes a).
$$
These actions give rise to four $A$-bimodule structures on $H$, and so we use the actions as subscripts to make clear which of these structures is under discussion. For example, ${}_\blact H_{\ract}$ denotes the $A$-bimodule $H$, with left $A$-action $a \blact h$ and right $A$-action $h \ract b$.

Recall that $\aemod$ is a monoidal category, with tensor product $\otimes_A$ and unit $A$.
\begin{defn}[see~\cite{MR0506407}]
A \emph{bialgebroid} is an $\Ae$-ring
$ \eta : \Ae \rightarrow H$
for which ${}_\lact H_\ract$
is a coalgebra in $\aemod$
whose coproduct
$
	\Delta \colon {}_\lact H_\ract \rightarrow {}_\lact H_\ract \otimes_A
	{}_\lact H_\ract
$
satisfies
$$
	a \blact \Delta(h) =\Delta (h) \bract a,\quad
	\Delta
	(gh)=\Delta(g)\Delta(h),
$$
and whose counit
$	
\varepsilon \colon {}_\lact H_\ract \rightarrow A
$
defines a unital $H$-action on $A$ given by
$h(a):=\varepsilon (a \blact h)$.
\end{defn}

Finally, by a Hopf algebroid we mean \emph{left}
rather than \emph{full} Hopf algebroid, so there is
in general no antipode \cite{3}:

\begin{defn}[see~\cite{MR1800718}]\label{hadef}
A \emph{Hopf algebroid} is a bialgebroid with bijective
\emph{Galois map}
$$
	\beta \colon {}_\blact H_\ract \otimes_{A^*}
{}_\lact H_\ract  \rightarrow {}_\blact H_\ract \otimes_A {}_\lact H_\ract,\quad
g \otimes_\Aop h \mapsto \Delta (g)h.  $$ \end{defn}

As usual, we abbreviate
\begin{equation}
	\label{hunger}
	\Delta (h) =: h_{(1)} \otimes_A
	h_{(2)},\qquad
	\beta^{-1}(h \otimes_A 1) =: h_+
	\otimes_\Aop h_-.
\end{equation}
This symbolic notation we use here resembles \emph{Sweedler notation} for Hopf algebras (see e.g.~\cite[\S 1.4]{MR1243637}), and is fully explained in
~\cite[\S 2.3]{MR3281654}.

\subsection{The opmonoidal adjunction}\label{bimfromhopf}
Every $E$-ring $H$ defines a forgetful functor
$$
	Q \colon H\cMod \rightarrow E\cMod
$$
with left
adjoint $P=H \otimes_E -$.
In the next section, we abbreviate
$\cH:=H\cMod$ and $\cE:=E\cMod$.
If $H$ is a bialgebroid (so $E = \Ae$)
then $\cH$ is monoidal with tensor product
$K \otimes_\cH L$ of
two left $H$-modules $K$ and $L$ given by the tensor
product $K \otimes_A L$ of the underlying
$A$-bimodules whose $H$-module structure is given by
$$
	h(k \otimes_\cH l):=
	h_{(1)}(k) \otimes_A h_{(2)}(l).
$$
So by definition, we have $Q(K \otimes _\cH L) =
Q K \otimes_A Q L$. The opmonoidal structure $\Xi$
on $P$ is defined by the map \cite{MR2793022,MR3020336}
\begin{align*}
	P(X \otimes_A Y)=H \otimes_\Ae (X \otimes_A Y)
	&\rightarrow P X \otimes_\cH P Y= (H \otimes_\Ae
	X) \otimes_A (H \otimes_\Ae Y), \\
	h \otimes_\Ae (x \otimes_A y) &\mapsto
	(h_{(1)} \otimes_\Ae x) \otimes_A
	(h_{(2)} \otimes_\Ae y).
\end{align*}

Schauenburg proved that this establishes a bijective
correspondence between bialgebroid structures on $H$
and monoidal structures on $H\cMod$
\cite[Theorem~5.1]{MR1629385}:
\begin{thm}
The following data are equivalent for an $\Ae$-ring
$ \eta \colon \Ae \rightarrow H$:
\begin{enumerate}
\item A bialgebroid structure on $H$.
\item A monoidal structure on ${H\cMod}$
such that the adjunction
$$
	\xymatrix{\aemod \hspace{-10mm}&
	\ar@/^2mm/[r] & \ar@/^2mm/[l] & \hspace{-10mm}
	H\cMod}
$$
induced by $ \eta $ is
opmonoidal.
\end{enumerate}
\end{thm}

Consequently,
we obtain an opmonoidal monad
$$
	QP={}_\blact H_\bract
	\otimes_\Ae -
$$
on $\cE=\Ae\cMod$. This
takes the unit object $I = A$ to the cocentre $H
\otimes_\Ae A$ of the $A$-bimodule ${}_\blact
H_\bract$, and the comonad $PI
\otimes_A -$ is given by
$$
	(H \otimes_\Ae A) \otimes_A -,
$$
where the $A$-bimodule structure on the
cocentre is given by the actions
$\lact,\ract$ on $H$.

The lift to $\cH={H\cMod}$ takes a left
$H$-module $L$ to
$
	(H \otimes_\Ae A)
	\otimes_A L
$
with action
$$
	g \cdot ((h \otimes_\Ae 1) \otimes_A l)=
	(g_{(1)}h \otimes_\Ae 1) \otimes_A g_{(2)}l,
$$
and the
distributive law resulting from Corollary~\ref{arisec} is
given by
$$
	\chi \colon g \otimes_\Ae ((h \otimes_\Ae 1)
	\otimes_A l) \mapsto (g_{(1)}h \otimes_\Ae 1) \otimes_A
(g_{(2)} \otimes_\Ae l).
$$
That is, it is the map
induced by the \emph{Yetter-Drinfel'd braiding}
$$
	H_\bract \otimes_A {}_\lact H \to H_\ract \otimes_A
	{}_\lact H	,\quad g \otimes_A h \mapsto g_{(1)}h
	\otimes_A g_{(2)}.
$$

For $A=k$, that is, when $H$ is a Hopf algebra, and
also trivially when $H=\Ae$, the monad and the comonad
on $\Ae\cMod$ coincide and are also a bimonad in the sense of
Mesablishvili and Wisbauer,
cf.\ Section~\ref{wisbimonad}.  An example where the two
are different is the Weyl algebra, or more generally,
the universal enveloping algebra of a
Lie-Rinehart algebra \cite{MR1625610}. In these
examples, $A$ is commutative but
not central in $H$ in general.

\subsection{Doi-Koppinen data}
The instance of Corollary~\ref{arisec}
that we are most interested in is an opmodule
adjunction associated to the following structure:
\begin{defn}
Following e.g.~\cite{MR1877862}, a \emph{Doi-Koppinen datum} is a triple
$(H,D,W)$ of an $H$-module coalgebra $D$ and an $H$-comodule
algebra $W$ over a bialgebroid $H$.
\end{defn}

This means that $D$ is a coalgebra in
the monoidal category ${H\cMod}$.
Similarly, the category $H\cComod$ of left $H$-comodules is
also monoidal (see {e.g.}~\cite[Section~3.6]{MR2553659}), and this defines the notion of a
comodule algebra. Explicitly, $W$ is an $A$-ring
$
	\eta_W \colon A \rightarrow W
$
together with a coassociative coaction
$$
	\delta \colon W \rightarrow H_\ract \otimes_A
W,\quad b
	\mapsto b_{(-1)} \otimes_A b_{(0)},
$$
which is counital and an algebra map,
$$
	\eta _W (\varepsilon
	(b_{(-1)}))b_{(0)}=b,\quad
	(b d)_{(-1)} \otimes (b d)_{(0)} =
	b_{(-1)} d_{(-1)}
	\otimes b_{(0)} d_{(0)}.
$$
Here again we use a Sweedler-esque notation to denote the coaction, as in~\cite[\S 2.5]{MR1243637}. Similarly, as in the
definition of a bialgebroid itself, for this condition
to be well-defined one must also require
$$
	b_{(-1)}
	\otimes_A b_{(0)} \eta_W(a)= a \blact b_{(-1)}
	\otimes_A b_{(0)}.
$$

The key example that reproduces
\cite{MR2803876} is the following.

\subsection{The opmodule
adjunction}\label{associatedsec}
For any Doi-Koppinen datum $(H,D,W)$, the
$H$-coaction $ \delta $ on $W$ turns the Eilenberg-Moore adjunction
\!\!
$\xymatrix{{A\cMod} \ar@/^{0.3pc}/[r] &
\ar@/^0.3pc/[l] {W\cMod}}$
\!\!
for the monad $ B:=W \otimes_A - $ into an opmodule
adjunction for the opmonoidal adjunction
$\xymatrix{\cE \ar@/^{0.3pc}/[r]
& \ar@/^0.3pc/[l] \cH}$ defined in
Section~\ref{bimfromhopf}.  The $\cH$-module
category structure of $W\cMod$ is given by the left
$W$-action
$$
	b(l \otimes_A m):=
	b_{(-1)}l \otimes_A
	b_{(0)}m,
$$
where $b \in W$,
$l \in L$ (an $H$-module), and $m \in
M$ (a $W$-module).

Hence, as explained in Section~\ref{opmoduleadj},
$D$ defines a compatible pair of
comonads $ D \otimes_A - $ on ${W\cMod}$ and ${A\cMod}$. The
distributive law resulting from Corollary~\ref{arisec}
generalises the Yetter-Drinfel'd braiding, as it is
given for a $W$-module $M$ by
\begin{eqnarray*}
	\chi \colon W \otimes_A (D
	\otimes_A M) &\rightarrow& D \otimes_A (W \otimes_A M),
\\
	b \otimes_A (c \otimes_A m) &\mapsto& b_{(-1)}c
\otimes _A (b_{(0)} \otimes_A m).
\end{eqnarray*}

\subsection{The main example}\label{coeffsforhopf}
If $H$ is a bialgebroid, then
$D:=H$ is a module coalgebra with
left action given by multiplication and coalgebra
structure given by that of $H$.
If $H$
is a Hopf algebroid, then $W:=H^*$ is
a comodule algebra with unit map $\eta_W(a):=\eta(1
\otimes_k a)$ and coaction
$$
 	\delta \colon H^* \rightarrow
	H_\ract \otimes_A {}_\blact H^*, \quad
 b \mapsto b_- \otimes_A b_+.
$$
In the sequel we write $B$ as $-
\otimes_\Aop H$ rather than $H^* \otimes_A -$
to work with $H$ only.  Then the distributive law
becomes
\begin{eqnarray*}
\nonumber
	\chi \colon (H \otimes_A M) \otimes_\Aop H
&\rightarrow& H \otimes_A (M \otimes_\Aop H),
\\
(c \otimes_A m) \otimes_\Aop b &\mapsto& b_- c \otimes _A (m
\otimes_\Aop b_+),
\end{eqnarray*}
for $b, c \in H$.

Proposition~\ref{chicoalgprop} completely
characterises the $ \chi $-coalgebras: in this
example, they are given by right $H$-modules and left
$H$-comodules $M$ with $ \chi $-coalgebra
structure
$$
	\rho: m \otimes_\Aop h \mapsto
	h_-m_{(-1)}
	\otimes_A m_{(0)} h_+.
$$
In general, the characterisation of
$ \chi $-opcoalgebras mentioned after Proposition~\ref{chicoalgprop}
does not provide us with such an explicit description.
Note, however, that
one obtains $ \chi $-opcoalgebras from
(left-left) Yetter-Drinfel'd modules:

\begin{defn}
A \emph{Yetter-Drinfel'd module} over $H$ is a left
$H$-comodule and left $H$-module $N$ such that
for all $ h\in H,n \in N$, one has
$$
	(hn)_{(-1)} \otimes_A (hn)_{(0)}=
	h_{+(1)} n_{(-1)} h_{-} \otimes_A
	h_{+(2)}n_{(0)}.
$$
\end{defn}

Indeed, each such Yetter-Drinfel'd
module defines a $ \chi$-opcoalgebra
$$
	- \otimes_H N \colon H^*\cMod \rightarrow
	k\cMod
$$
whose $ \chi $-opcoalgebra structure is given by
\begin{equation*}
\lambda: (h \otimes_A x) \otimes_H n
	\mapsto (xn_{(-1)+} h_+ \otimes_\Aop h_- n_{(-1)-})
	\otimes_H n_{(0)}.
\end{equation*}
The resulting duplicial object
$\rCC_T({-}\otimes_H N,M)$ is the one
studied in \cite{MR2803876, Kow:GABVSOMOO}.

Identifying
$(- \otimes_\Aop H) \otimes_H N \cong - \otimes_\Aop N$,
the $ \chi $-opcoalgebra structure
  becomes
\begin{equation*}
\lambda: (h \otimes_A x) \otimes_H n
	\mapsto
	xn_{(-1)+} h_+ \otimes_\Aop h_- n_{(-1)-} n_{(0)}.
\end{equation*}
Using this identification, we give
explicit expressions of the operators $L$ and $R$
as well as $t_\mathbb{T}$ that appeared in
Sections~\ref{evidenziatore1} and~\ref{evidenziatore2}:
first of all,
observe that the right $H$-module structure on
$
	S M := {}_\lact H_\ract \otimes_A M
$
is given by
$$
	(h \otimes_A m)g :=
	g_- h \otimes_A mg_+,
$$
whereas the right $H$-module structure on
$
T M := M \otimes_\Aop {}_\blact H_\ract
$
is given by
$$
	(m \otimes_\Aop h)g :=
	m \otimes_\Aop hg.
$$
The cyclic operator from Section \ref{evidenziatore1}
then results as
\begin{equation*}
\begin{split}
&t_T
(m  \otimes_\Aop h^1 \otimes_\Aop \cdots \otimes_\Aop h^n \otimes_\Aop n)
\\
&=
m_{(0)} h^1_+ \otimes_\Aop h^2_+ \otimes_\Aop \cdots \otimes_\Aop h^n_+ \\
& \qquad
\otimes_\Aop (n_{(-1)} h^n_- \cdots h^1_- m_{(-1)})_+
\otimes_\Aop (n_{(-1)} h^n_- \cdots h^1_- m_{(-1)})_-
n_{(0)},
\end{split}
\end{equation*}
and for the operators $L$ and $R$ from Section~\ref{evidenziatore2}
 one obtains with the help of the properties
\cite[Prop.~3.7]{MR1800718} of the translation map
(\ref{hunger}):
\begin{equation*}
\begin{split}
L:  (h^1 & \otimes_A \cdots \otimes_A h^{n+1}
\otimes_A m) \otimes_H n
\mapsto \\
& (mn_{(-1)+} h^1_+ \otimes_\Aop
h^1_- h^2_+ \otimes_\Aop \cdots \otimes_\Aop h^{n+1}_- n_{(-1)-})
\otimes_H n_{(0)},
\end{split}
\end{equation*}
along with
\begin{equation*}
\begin{split}
R: (m & \otimes_\Aop h^1
\otimes_\Aop \cdots \otimes_\Aop h^n \otimes_\Aop 1)
\otimes_H n
\mapsto \\
&
(m_{(-n-1)}  \otimes_A m_{(-n)}h^1_{(1)} \otimes_A
m_{(-n+1)}h^1_{(2)}h^2_{(1)}  \otimes_A \cdots \\
& \quad
\otimes_A m_{(-1)} h^1_{(n)}
h^2_{(n-1)} \cdots h^n_{(1)}
\otimes_A m_{(0)}) \otimes_H h^1_{(n+1)} h^2_{(n)} \cdots h^n_{(2)} n.
\end{split}
\end{equation*}
Compare these maps with those obtained in \cite[Lemma~4.10]{MR2803876}.
Hence, one has:
\begin{equation*}
\begin{split}
(L &\circ R)\big((m  \otimes_\Aop h^1
\otimes_\Aop \cdots \otimes_\Aop h^n \otimes_\Aop 1)
	\otimes_H n\big)
= \\
&
m_{(0)} (h^1_{(n+1)}h^2_{(n)} \cdots h^n_{(2)} n)_{(-1)+} m_{(-n-1)+}
\otimes_\Aop m_{(-n-1)-}  m_{(-n)+}  h^1_{(1)+}
\\
&
\qquad
\otimes_\Aop
h^1_{(1)-} m_{(-n)-}  m_{(-n+1)+}  h^1_{(2)+} h^2_{(1)+} \otimes_\Aop
\cdots
\\
&
\qquad
\otimes_\Aop h^n_{(1)-} \cdots h^1_{(n)-} m_{(-1)-} (h^1_{(n+1)} \cdots h^n_{(2)} n)_{(-1)-} (h^1_{(n+1)} \cdots h^n_{(2)} n)_{(0)}
\\
&
=
m_{(0)} \big((h^1_{(2)} \cdots h^n_{(2)} n)_{(-1)} m_{(-1)}\big)_+
\otimes_\Aop h^1_{(1)+} \otimes_\Aop \cdots
\\
&
\quad
\otimes_\Aop h^n_{(1)+}
\otimes_\Aop h^n_{(1)-} \cdots h^1_{(1)-} \big((h^1_{(2)} \cdots h^n_{(2)} n)_{(-1)} m_{(-1)} \big)_- (h^1_{(2)} \cdots h^n_{(2)} n)_{(0)}.
\end{split}
\end{equation*}
Finally, if $M \otimes_\Aop N$ is a stable anti
Yetter-Drinfel'd module \cite{MR2415479},
that is, if
$$
m_{(0)}(n_{(-1)}m_{(-1)})_+ \otimes_\Aop (n_{(-1)}m_{(-1)})_- n_{(0)} = m \otimes_\Aop n
$$
holds for all $n \in N$, $m \in M$,
we conclude by observing that
\begin{equation*}
\begin{split}
(L \circ R)(m & \otimes_\Aop h^1 \otimes_\Aop \cdots
\otimes_\Aop h^n \otimes_\Aop n)
\\
&
=
m \otimes_\Aop h^1_{(1)+} \otimes_\Aop \cdots \otimes_\Aop h^n_{(1)+}
\otimes_\Aop h^n_{(1)-} \cdots h^1_{(1)-} h^1_{(2)} \cdots h^n_{(2)} n
\\
&
=
m \otimes_\Aop h^1 \otimes_\Aop \cdots \otimes_\Aop h^n
\otimes_\Aop n.
\end{split}
\end{equation*}
Observe that in \cite{Kow:GABVSOMOO} this
cyclicity condition was obtained for a different complex
which, however, computes the same homology.

\subsection{The antipode as a $1$-cell}
\label{viviverde}
If $A=k$, then the four actions
$\lact,\ract,\blact,\bract$ coincide and
$H$ is a Hopf algebra with antipode
$S \colon H \rightarrow H$ given by
$S(h)=\varepsilon (h_+)h_-$. The
aim of this brief section is to remark that
this defines a 1-cell that connects the two instances
of Corollary~\ref{arisec} provided by the opmonoidal
adjunction and the opmodule adjunction considered
above.

Indeed, in this case we have $\Ae\cMod\cong A\cMod=k\cMod$,
but (unless $H$ is commutative) $H^*\cMod \neq H\cMod$. However, $S$ defines
a morphism of monads
$$\xymatrix{
(\kmod, H \otimes -) \ar[rr]^-{(1, \sigma)}  && (\kmod, - \otimes H)
}
$$ where $\sigma \colon{-} \otimes H \Rightarrow H \otimes{-}$ is given in components by
$$
		\sigma \colon X \otimes H
	\rightarrow H \otimes X, \quad
	x \otimes h \mapsto S(h) \otimes x.
$$
The fact that this $(1, \sigma)$ is a morphism of monads is
equivalent to the fact that $S$ is an algebra
anti-homomorphism.
Also, the lifted comonads agree and are given by
$H \otimes{-}$ with comonad structure given by the
coalgebra structure of $H$;
clearly, $\gamma = 1
\colon  H \otimes {-} \Rightarrow H \otimes {-}$
defines an opmorphism of monads
$$
\xymatrix{
(\kmod, H \otimes{-}) \ar[rr]^-{(1, \gamma)} & & (\kmod, H \otimes{-})
}
$$ Furthermore, the Yang-Baxter
condition is satisfied, so we have that $(1,
\sigma, \gamma)$ is a $1$-cell in the $2$-category of mixed
distributive laws. If we apply the $2$-functor of Corollary~\ref{2funcmixdist} to
this, we get a $1$-cell $(\Sigma, \tilde \sigma, \tilde
\gamma)$ between a comonad distributive law on the
category of left $H$-modules and one on the category of
right $H$-modules. The identity lifts to the functor
$\Sigma \colon H\mbox{-}\mathbf{Mod} \rightarrow
\mathbf{Mod}\mbox{-}H$ which sends a left $H$-module
$X$ to the right $H$-module with right action given by
$$
x \cdot h := S(h) x.
$$
\section{Hopf monads \`a la Mesablishvili-Wisbauer}\label{hopfywopfy}
\label{wisbimonad}
\subsection{Bimonads}
A \emph{bimonad} in the sense of
\cite[Def.~4.1]{MR2787298} is a sextuple
$(B,\mu,\eta,\delta, \varepsilon,\theta)$,
where $B
\colon \cA \rightarrow \cA$ is a functor, $(B, \mu,\eta)$ is
a monad, $(B,\delta,\varepsilon)$ is a comonad and
$\theta \colon BB \Rightarrow BB$ is a mixed distributive
law satisfying a list of compatibility conditions. A \emph{Hopf monad} as in \cite[Def.~5.2]{MR2787298} is a bimonad $B$ equipped with
a natural transformation $B \Rightarrow B$, called the \emph{antipode}, satisfying various compatibility conditions mirroring those
for Hopf algebras.

In particular, for a bimonad $B$, the multiplication $ \mu $ and comultiplication $ \delta $ are required to
be compatible in the sense that there is a commutative
diagram \begin{equation}\label{wisbauerdiagram}
\begin{array}{c} \xymatrix{BB \ar[d]_{B
{\delta}}
\ar[r]^\mu & B \ar[r]^{\delta} & BB\\
BBB
\ar[rr]_{\theta B} & & BBB \ar[u]_{B\mu}} \end{array}
\end{equation}
The other defining conditions govern the
compatibility between the unit and the counit with each
other and with $ \mu $ and $ \delta$ respectively, see
\cite{MR2787298} for the details.

It follows immediately that we also obtain an instance
of Corollary~\ref{arisec} in this situation: if we take
$\cB=\cA^B$ to be the Eilenberg-Moore
category of the monad $B$ as in
Section~\ref{extremalcase}, then the mixed distributive
law $\theta$ defines a lift
$W$ of the comonad
$B$ to $\cB$.

Note that in general, neither $\cA$ nor $\cB$ need to
be monoidal, so $B$ is in general not an opmonoidal
monad. Conversely, recall that for the examples of
Corollary~\ref{arisec} obtained from opmonoidal monads,
$B$ need not equal $C$ as functors.

\subsection{Examples from bialgebras}\label{whygalois}
In the main
example of bimonads in the above sense, we in fact do
have $B=C$ and we are in the situation of
Section~\ref{bimfromhopf} for a bialgebra $H$ over
$A=k$.  The commutativity of (\ref{wisbauerdiagram})
amounts to the fact that the coproduct is an algebra
map.

This setting provides an instance of
Proposition~\ref{sunshines} since there are two lifts
of $B=C$ from $\cA={k\cMod}$ to $\cB={H\cMod}$: the
canonical lift $S=T=FU$ which takes a left
$H$-module $L$ to the $H$-module $H \otimes L$ with
$H$-module structure given by multiplication in the
first tensor component, and the lift $W$ which takes
$L$ to $H \otimes L$ with $H$-action given by the
codiagonal action
$
	g(h \otimes y)=
	g_{(1)}h \otimes
	g_{(2)}y,
$
that is, the one defining the monoidal
structure on $\cB$.

In this example, the map $ \beta $ from
Proposition~\ref{wisga} is given by
$$
	H \otimes L \rightarrow H \otimes
	L,\quad
	g\otimes y \mapsto g_{(1)} \otimes
	g_{(2)}y
$$
which for $L=H$ is the Galois map
from Definition~\ref{hadef}. This is bijective for all
$L$ if and only if it is so for $L=H$, which is also
equivalent to $H$ being a Hopf algebra. However,
this Galois map should not just be viewed as a $k$-linear
map, but as a natural $H$-module morphism between the two
$H$-modules $ TL$ and $WL$, and this is
the natural transformation $ \Gamma ^{T,W}(1)$ from
Section~\ref{galoismapsct}.

As shown in \cite[Theorem~5.8(c)]{MR3320218}, this
characterisation of Hopf algebras in terms of the
bijectivity of the Galois map extends straightforwardly
to Hopf monads.

\subsection{An example not from bialgebras}\label{newbimonad}
Another
example of a bimonad is the \emph{nonempty list monad}
$L^+$ on $\cSet$ (see Example~\ref{nonemptylistmonad}), which assigns to a
set $X$ the set $L^+X$ of all nonempty lists
of elements in $X$, denoted $[x_1, \ldots, x_n]$.  The mixed
distributive law
$$
	\theta \colon L^+ L^+ \Rightarrow
	L^+ L^+
$$
is defined as follows: given a list
$$
	[ [ x_{1,1}, \ldots, x_{1,
	n_1} ] , \ldots, [x_{m,1}, \ldots, x_{m, n_m}]]
$$
in
$L^+L^+ X$,
its image under $\theta$ is the list with $$
\sum_{i=1}^m n_i (m-i+1) $$ terms, given by 
 \begin{align*}
\Big[[x_{1,1}, x_{2,1}, x_{3,1} \ldots, x_{m,1}],
\ldots, [x_{1, n_1}, x_{2,1}, x_{3,1}, \ldots,
x_{m,1}]&, \\ [x_{2,1}, x_{3,1} \ldots, x_{m,1}],
\ldots, [x_{2, n_2}, x_{3,1}, \ldots x_{m,1} ]&,\\
\ldots&,  \\  [ x_{m,1} ] , [ x_{m,2} ], \ldots,
[x_{m,n_m} ] \Big].&
\end{align*}

One verifies straightforwardly:

\begin{prop} $L^+$ becomes a bimonad on $\set$ whose (monad) Eilenberg-Moore category is
$\cSet^{L^+} \cong \mathbf{SemiGrp}$, the
category of (nonunital) semigroups.  \end{prop}

The second lift $W \colon \mathbf{SemiGrp \to SemiGrp}$ of the comonad $L^+$ that
one obtains from the bimonad structure is as follows. Given a semigroup $X$,
we have $W X = L^+ X$ as sets, but the binary
operation is given by \begin{align*} W X \times W X
&\rightarrow W X \\ [x_1, \ldots, x_m][y_1, \ldots, y_n] &:=
[x_1y_1, \ldots, x_my_1, y_1, \ldots, y_n]. \end{align*}

Following Proposition~\ref{chicoalgprop}, given a
semigroup $X$, the unit turns the underlying set of $X$ into an
$L^+$-coalgebra and hence we get a $
\chi$-coalgebra structure on $X$. Explicitly, $\rho \colon T X \rightarrow W X$ is given by
$$
\rho[ x_1, \ldots, x_n] = [x_1 \cdots x_n, x_2 \cdots
x_n, \ldots, x_n].
$$
The image of $\rho$ is known as
the \emph{left machine expansion} of $X$
\cite{MR745358}.

\begin{prop}
The only $\theta$-entwined algebra is the trivial semigroup $\emptyset$.
\end{prop}
\begin{proof}
An $L^+$-coalgebra structure $\beta \colon X \to L^+ X$ is equivalent to $X$ being a forest of at most countable height (rooted) trees, where each level may have arbitrary cardinality. The structure map $\beta$ sends $x$ to the finite list of predecessors of $x$. A $\theta$-entwined algebra is therefore such a forest, which also has the structure of a semigroup such that for all $x,y \in X$ with $\beta(y) = [y, y_1, \ldots, y_n]$ we have
$$
\beta(xy) = [xy, xy_1, \ldots, xy_n , y, y_1, \ldots, y_n].
$$

Let $X$ be a $\theta$-entwined algebra. If $X$ is non-empty, then there must be a root. We can multiply this root
with itself to generate branches of arbitrary height. Suppose that we have a
branch of height two; that is to say, an element $y \in X$ with $\beta(y) = [y,x]$
(so, in particular, $x \neq y$). Then $\beta(xy) = [xy, y]$,  but
$\beta(xx) = [xx, xy, x, y]$. This is impossible since $x$ and $y$ cannot both be the predecessor of $xy$.
\end{proof}
\section{Enriched functor categories}\label{finally}
We conclude the main part of the thesis by showing how to construct a duplicial object that generalises some of the examples we have previously seen.
\subsection{Enriched categories}
Let $\cV$ be a monoidal category.
\begin{defn}
A $\cV$-\emph{category} $\cH$ consists of
\begin{itemize}
\item a class $|\cH|$ whose elements we call objects
\item for any $A, B \in |\cH|$, an object $\cH(A, B) \in |\cV|$
\item for each object $A \in |\cH|$, a morphism $u_\cA \colon I \to \cH(A, A)$ in $\cV$, called the \emph{unit}
\item for any $A, B, C$ in $|\cH|$, a morphism $\circ_{A,B,C} \colon \cH(B, C) \otimes \cH(A, B) \to \cH(A, C)$, called
\emph{composition}
\end{itemize}
satisfying associativity and unitality conditions, that is, commutativity of the two diagrams
\begin{align*}
\xymatrix{
\cH({}C, {}D) \otimes \cH({}B, {}C) \otimes \cH({}A, {}B)\ar[d]_-{\circ_{B,C,D} \otimes 1} \ar[rr]^-{1 \otimes {\circ_{A,B,C}}} && \cH({}C, {}D) \otimes \cH({}A, {}C) \ar[d]^-{\circ_{A,C,D}} \\
\cH({}B, {}D) \otimes \cH({}A, {}B) \ar[rr]_-{\circ_{A,B,D}} && \cH({}A, {}D)
} \\
\\
\xymatrix{
\cH({}A, {}B) \ar@{=}[drr] \ar[rr]^-{u_B \otimes 1} \ar[d]_-{1 \otimes u} && \cH({}B, {}B) \otimes \cH({}A, {}B) \ar[d]^-{\circ_{A,B,B}} \\
\cH({}A, {}B) \otimes \cH({}A, {}A) \ar[rr]_-{\circ_{A,A,B}} & & \cH({}A, {}B)
}
\end{align*}
for all ${}A, {}B, {}C, {}D \in |\cH|$.
\end{defn}
\begin{exa}
A $\set$-category is an ordinary category.
\end{exa}
\begin{exa}
A $\cat$-category is a 2-category.
\end{exa}
We now suppose that $\cV$ is a complete, cocomplete, closed symmetric monoidal category. Furthermore, we assume for simplicity that the monoidal structure of $\cV$ is strict, but we make no assumption that the symmetry is strict. The tensor product is denoted by $\otimes$, the unit by $I$, and the closed structure is given by the functor
$$
[{-},{-}] \colon \cV^* \times \cV \to \cV.
$$

Since left adjoint functors preserve limits, it follows that that the tensor product commutes with coproducts. That is,  we have natural isomorphisms
$$
A \otimes \sum_{i \in I} B_i \cong \sum_{i \in I} A \otimes B_i
$$
where $\sum$ denotes the coproduct, and $\{B_i\}_{i \in I}$ is a family of objects in $\cV$.

For any $\cV$-categories $\cH$ and $\cK$, there is an associated (ordinary) $\cV$-functor category $[\cH, \cK]$. This can also be given the structure of a $\cV$-category, but we do not need it here and thus do not give any further details. We are interested particularly in the case that $\cK = \cV$. By definition, a $\cV$-functor $F \colon \cH \to \cV$ consists of
\begin{itemize}
\item for each object $X$ in $\cH$, an object $FX$ in $\cV$
\item for any objects $X,Y$ in $\cH$, a morphism $F_{X,Y} \colon \cH(X,Y) \to [FX, FY]$ in $\cV$
\end{itemize}
satisfying appropriate unitality and associativity axioms. We now omit subscripts and just write $F$ in place of $F_{X,Y}$.

Every morphism
$$
F\colon \cH(X, Y) \to [FX, FY]
$$
corresponds to an \emph{action}
$$
\overline{F} \colon \cH(X,Y) \otimes FX \to FY
$$
using the closed structure of $\cV$,
and thus the axioms defining a $\cV$-functor can be rewritten as the two commutative diagrams
$$
\xymatrix@C=3em{
\cH(Y,Z) \otimes \cH(X,Y) \otimes FX \ar[rr]^-{1 \otimes \overline F} \ar[d]_-{\circ \otimes 1} && \cH(Y,Z) \otimes FY \ar[d]^-{\overline F} \\
\cH(X,Z) \otimes FX \ar[rr]_-{\overline F}& & FZ
}
$$
$$
\xymatrix@C=3em{
FX \ar[rr]^-{u \otimes 1} \ar@{=}[drr] & & \cH(X,X) \otimes FX \ar[d]^-{\overline F} \\
& & FX
}
$$
Similarly, a $\cV$-natural transformation $\alpha \colon F \Rightarrow G$ is defined as a collection of morphisms $\alpha \colon FX \to GX$ in $\cV$ such that the diagram
$$
\xymatrix@C=3em{
\cH(X,Y) \otimes FX \ar[rr]^-{\overline F} \ar[d]_-{1 \otimes \alpha} && FY \ar[d]^-\alpha \\
\cH(X,Y) \otimes GX \ar[rr]_-{\overline G}  && GY
}
$$
commutes.
\subsection{Hopf categories}
\begin{defn}
A \emph{comonoidal $\cV$-category} is a category $\cH$ enriched over the category of coalgebras in $\cV$. Explicitly, each object $\cH(X,Y)$ is a coalgebra in $\cV$ in such a way that the composition and unit of $\cH$ are coalgebra morphisms.
\end{defn}
\begin{defn}
A \emph{Hopf $\cV$-category} is a comonoidal $\cV$-category $\cH$ equipped with a collection of morphisms
$$
\xymatrix{
\cH(X,Y) \ar[r]^-{S} & \cH(Y,X)
}
$$
such that the diagrams
$$
\xymatrix{
\cH(X,Y) \otimes \cH(X,Y) \ar[rr]^-{1\otimes S} && \cH(X,Y) \otimes \cH(Y,X) \ar[d]^-{\circ} \\
\cH(X,Y) \ar[u]^-{\delta} \ar[r]_-{\epsilon} & I \ar[r]_-{u} & \cH(X,X)
}
$$
and
$$
\xymatrix{
\cH(X,Y) \otimes \cH(X,Y) \ar[rr]^-{S \otimes 1} && \cH(Y,X) \otimes \cH(X,Y) \ar[d]^-{\circ} \\
\cH(X,Y) \ar[u]^-{\delta} \ar[r]_-{\epsilon} & I \ar[r]_-{u} & \cH(Y,Y)
}
$$
commute.
\end{defn}
Our terminology comes from~\cite{Hop} but these are called Hopf $\cV$-algebroids in~\cite{MR1458415}.
\subsection{The comonad $T_\cH$}
Let $\cH$ be any $\cV$-category.
We now construct a pair of comonads on the contravariant enriched functor category $[\cH^*, \cV]$ and a distributive law between them. We do this however, by first defining a comonad on the covariant enriched functor category and pulling a few tricks. We define an endofunctor $T_\cH$ on $[\cH,\cV]$ as follows: given a $\cV$-functor $\cH \to \cV$, let $T_\cH(F)$ be defined on objects by
$$
T_\cH(F)(X) = \sum_{Y} \cH(Y,X) \otimes FY.
$$
The action $\overline{T_\cH(F)}$ is defined by lifting the composition
$$
\xymatrix@C=2.5em{
\cH(X,Y) \otimes \cH(Z,X) \otimes FZ \ar[r]^-{\circ \otimes 1 }
& \cH(Z, Y) \otimes FZ
}
$$
to the coproduct, so it is clear that $T_\cH(F)$ is a well-defined $\cV$-functor. Given a $\cV$-natural transformation $\alpha \colon F \Rightarrow G$, we define $T_\cH(\alpha) \colon T_{\cH}(F) \Rightarrow T_\cH(G)$ by lifting the morphisms
$$\xymatrix{
\cH(Y,X) \otimes FX \ar[rr]^-{1 \otimes \alpha} & &\cH(Y,X) \otimes GX}
$$ to the coproduct, and so clearly $T_\cH$ is a well-defined functor.

The morphisms
$$
\xymatrix@C=2.5em{
\cH(Y,X) \otimes FY \ar[rr]^-{1 \otimes u \otimes 1} & & \cH(Y,X) \otimes \cH(Y,Y) \otimes FY,
}
$$
and
$$
\xymatrix{
\cH(Y,X) \otimes FY \ar[r]^-{\overline F} & FX
}
$$
lift to the coproduct and define natural transformations $\delta \colon T_\cH \Rightarrow T_\cH T_\cH$ and  $\epsilon \colon T_\cH \Rightarrow 1$ respectively, which endow $T_\cH$ with the structure of a comonad.

Now, suppose that $\cH$ is a comonoidal $\cV$-category. Then, the category $[\cH, \cV]$ is monoidal~\cite[p.~143]{MR1458415} with tensor product $\otimes$ and unit $I$ given pointwise by those of $\cV$, i.e.\
$$
(F \otimes G)(X) = FX \otimes GX, \qquad I(X) = I
$$
and with respect to this structure, $T_\cH$ is in fact an opmonoidal comonad. Therefore, there is a comonad $S_\cH = T_\cH(I) \otimes {-}$ on $[\cH, \cV]$ and a distributive law $T_{\cH}S_{\cH} \Rightarrow S_{\cH}T_{\cH}$ (cf.~Section~\ref{notsure} and e.g.~\cite{MR2948490}). Explicitly, $S_\cH$ is defined on objects $F$ by
$$
S_\cH (F) (X) = \sum_Y \cH(Y,X) \otimes FX
$$
and the composites $T_{\cH}S_{\cH}$ and $S_{\cH}T_{\cH}$ are given by
\begin{align*}
(T_\cH S_\cH) (F) (X) &= \sum_{Y,Z} \cH(Y,X) \otimes \cH(Z,Y) \otimes FY, \\
(S_\cH T_\cH) (F) (X) &= \sum_{Y,Z} \cH(Y,X) \otimes \cH(Z,X) \otimes FZ .
\end{align*}
The distributive law is defined by the diagram
$$
\xymatrix{
\cH(Y,X) \otimes \cH(Y,Z) \otimes FY \ar@{.>}[r] \ar[d]_-{\delta \otimes 1 \otimes 1} & \cH(Z,X) \otimes \cH(Y,X) \otimes FY \\
\cH(Y,X) \otimes \cH(Y,X) \otimes \cH(Z,Y) \otimes FY \ar[r]_-{\cong} &
\cH(Y,X) \otimes \cH(Z,Y) \otimes \cH(Y,X) \otimes FY \ar[u]_-{\circ \otimes 1 \otimes 1}
}
$$
where the bottom isomorphism swaps the two inner tensorands.

Of course, since $\cH$ was arbitrary, we can replace it with $\cH^*$ to obtain a comonad $T_{\cH^*}$.
\subsection{The comonad $R_\cH$}
Now suppose that $\cH$ is a Hopf $\cV$-category. Then $[\cH^*, \cV]$ becomes a left module category for $[\cH, \cV]$, with action
$$
\xymatrix{
[\cH, \cV] \times [\cH^*, \cV] \ar[rr]^-{\rhd} & & [\cH^*, \cV]
}
$$ defined on objects by $$
(F \rhd G)(X) = FX \otimes GX.
$$
The action of the $\cV$-functor $\overline{F \rhd G}$ is given by
$$
\xymatrix@C=4em{
\cH(Y,X) \ar@{.>}[ddd]_-{\overline{F \rhd G}}\otimes FX \otimes GX \ar[r]^-{\delta \otimes 1 \otimes 1} & \cH(Y,X) \otimes \cH(Y,X) \otimes FX \otimes GX  \ar[d]^-{1 \otimes S \otimes 1 \otimes 1}\\
  & \cH(Y,X) \otimes \cH(X,Y) \otimes FX \otimes GX \ar[d]^-{1 \otimes \overline F \otimes 1} \\
  & \cH(Y,X) \otimes FY \otimes GX \ar[d]^-{\cong} \\
FY \otimes GY  & FY \otimes \cH(Y,X) \otimes GX \ar[l]^-{1 \otimes \overline G}
}
$$
In particular, if we choose $F$ to be the coalgebra $T_\cH(I)$, then we get a comonad
$$
\xymatrix{
[\cH^*, \cV]\ar[rr]^-{T_\cH (I) \rhd {-}} & & [\cH^*, \cV]
}
$$
which we denote by $R_\cH$. Unravelling everything, we have that $R_\cH$ is defined on objects $F$ by
$$
R_\cH (F)(X) =S_\cH (F) (X) = \sum_Y \cH(Y,X) \otimes FX
$$
with action $\overline{R_\cH(F)}$ given by lifting to the coproduct the morphisms
$$
\xymatrix@C=4em{
\cH(Y,X) \ar@{.>}[ddd] \otimes \cH(Z,X) \otimes FX \ar[r]^-{\delta \otimes 1 \otimes 1} & \cH(Y,X) \otimes \cH(Y,X) \otimes \cH(Z,X) \otimes FX  \ar[d]^-{1 \otimes S \otimes 1 \otimes 1}\\
  & \cH(Y,X) \otimes \cH(X,Y) \otimes \cH(Z,X) \otimes FX \ar[d]^-{1 \otimes \circ \otimes 1} \\
  & \cH(Y,X) \otimes \cH(Z, Y) \otimes FX \ar[d]^-{\cong} \\
\cH(Z,Y) \otimes FY  & \cH(Z,Y) \otimes \cH(Y,X) \otimes FX \ar[l]^-{1 \otimes \overline F}
}
$$
The comonad structure is induced in the obvious way from the coalgebra structure on each $\cH(Y,X)$.
\subsection{The distributive law $\chi$}
Again, we automatically have a distributive law (cf.~Section~\ref{opmoduleadj})
$$
\chi \colon T_{\cH^*} R_{\cH} \Rightarrow R_{\cH} T_{\cH^*}.
$$
Explicitly we have
\begin{align*}
(T_{\cH^*} R_{\cH} ) (F)(X) &= \sum_{Y,Z} \cH(X,Y) \otimes \cH(Z,Y) \otimes FY, \\
(R_{\cH} T_{\cH^*} ) (F) (X) &= \sum_{Y,Z} \cH(Y,X) \otimes \cH(X,Z) \otimes FZ,
\end{align*}
and the distributive law is induced by the composites
$$
\xymatrix@C=4em{
\cH(X,Y) \ar@{.>}[dd] \otimes \cH(Z,Y) \otimes FY \ar[r]^-{\delta \otimes 1 \otimes 1} & \cH(X,Y) \otimes \cH(X,Y) \otimes \cH(Z,Y) \otimes FY  \ar[d]^-{1 \otimes S \otimes 1 \otimes 1}\\
  & \cH(X,Y) \otimes \cH(Y,X) \otimes \cH(Z,Y) \otimes FY \ar[d]^-{1 \otimes \circ \otimes 1} \\
\cH(Z,X) \otimes \cH(X,Y) \otimes FY  & \cH(X,Y) \otimes \cH(Z, X) \otimes FY \ar[l]^-{\cong}
}
$$
\subsection{The $\chi$-(op)coalgebras and duplicial object}
Consider the unit object $I$ in $[\cH^*, \cV]$. We define a $\cV$-natural transformation $$\rho \colon T_{\cH^*}(I) \Rightarrow R_{\cH}(I)$$ as follows. The components are
defined by
$$
\rho \colon T_{\cH^*}(I)(X) = \sum_{Y} \cH(X,Y) \lto \sum_{Y} \cH(Y,X) = R_{\cH}(I)(X)
$$
to be the coproduct of the antipodes $S \colon \cH(X,Y) \to \cH(Y,X)$. It follows that $(I, \rho)$ is a $\chi$-coalgebra in $[\cH^*, \cV]$.

Now, let $N \colon [\cH^*, \cV] \to \cV$ be the functor which maps a $\cV$-functor $F$ to the coend
$$
\int^{X} FX
$$
which is, by definition, the coequaliser
$$
\xymatrix@C=3.5em{
\displaystyle\sum_{X,Y} FY \otimes \cH(X,Y) \cong \sum_{X,Y} \cH(X,Y) \otimes FY \ar@<1ex>[r]^-{\sum \epsilon \otimes 1} \ar@<-1ex>[r]_-{\sum\overline F}& \displaystyle\sum_X FX \ar@{->>}[r] & \displaystyle\int^{X} FX.
}
$$
We define morphisms
$$
\xymatrix{
L \colon \cH(Y,X) \otimes FX \ar[r]^-{\overline F} & FY \ar[rr]^-{u \otimes 1} && \cH(Y,Y) \otimes FY
}
$$
Since we have
$$
\sum_X R_{\cH}(F)(X) = \sum_{X,Y} \cH(Y,X) \otimes FX, \qquad \sum_X T_{\cH^*}(F)(X) = \sum_{X,Y} \cH(X,Y) \otimes FX
$$
the morphisms $L$ define a morphism
$$
\sum_X R_\cH (F) (X) \lto \sum_X T_{\cH^*} (F) (X)
$$
which coequalises the morphisms defining the coend $\displaystyle\int^X R_{\cH}(F)$. Thus, there is an induced morphism
$$
\xymatrix{
\displaystyle \int^X R_{\cH} (F) (X)\ar[rr]^-{\lambda} & & \displaystyle\int^X T_{\cH^*} (F)(X)
}
$$
which defines a natural transformation $\lambda \colon N R_\cH \Rightarrow N T_{\cH^*}$ which turns the triple $(N, \cV, \lambda)$ into a $\chi$-opcoalgebra. Therefore, by Theorem~\ref{dup}, we have:
\begin{cor}\label{endcor}
The simplicial object $$ \rC_{T_{\cH^*} } (N, I)$$
is a duplicial object in $\cV$.
\end{cor}
\begin{exa}
Let $\cV = \set$ with monoidal structure given by the Cartesian product $\times$. A Hopf $\set$-category $\cG$ is precisely a groupoid~\cite[Prop.~2.4]{Hop}. The terminal object $I = \{ * \}$ is  the unit, and the coend functor $N$ is given by the colimit functor
$$
\colim \colon [\cG^*, \set] \to \set.
$$
In this case
$$
\rC_{T_{\cG^*} } (\colim, I) \cong N \cG
$$
where $N\cG$ denotes the nerve of the groupoid $\cG$, cf.~Section~\ref{nerve}.

We can say even more: suppose that we have an adjunction
$$
\xymatrix{
\cG \ar@{}[rr]|-{\perp}\ar@/^0.5pc/[rr]^-F & & \ar@/^0.5pc/[ll]^-U \cC
}$$
where $\cC$ is any category. The presheaf category $[\cC^*, \set]$ becomes a monoidal category with the pointwise Cartesian product, and the comonad $T_{\cC^*}$ is opmonoidal (since every comonad is with respect to this particular monoidal structure) and so there is a distributive law
$$
T_{\cC^*} \circ (T_{\cC^*} (I) \otimes {-}) \Rightarrow (T_{\cC^*}(I) \otimes {-}) \circ T_{\cC^*} .
$$ The functor $$[F^*, \set] \colon [\cC^*, \set] \to [\cG^*, \set]$$ is the functor part of a 1-cell in $\dist$ between the aforementioned distributive laws. However, since $F$ is necessarily full and faithful (see the proof of~Corollary~\ref{motivation}), we have
$$
\colim \circ~[F^*, \set] \cong \colim \colon [\cC^*, \set] \to \set.
$$Therefore we may act on the above duplicial object with the 1-cell $[F^*, \set]$ as in Section~\ref{twistsec} to obtain another one
$$
\rC_{T_{\cC^*}}(\colim, I) \cong N \cC.
$$
The duplicial structure induced on  $N\cC$ is the same as that given by Theorem~\ref{auspara}.
\end{exa}
\begin{exa}
Let $k$ be a commutative ring, and let $\cV = \kmod$, with tensor product $\otimes = \otimes_k$ and unit $k$. Any Hopf algebra $H$ over $k$ can be viewed as a one-object Hopf $\cV$-category $\cH$, and $[\cH, \cV] \cong H\cMod$. The coend functor becomes
$$
{-}\otimes_H k \colon H^*\cMod \to \kmod
$$
and the duplicial $k$-module $$
\rC_{T_{\cH^*}} ({-}\otimes_H k, k)
$$
is precisely the one given in Section~\ref{coeffsforhopf} in the case that $A = M = N = k$, whose ordinary homology is given by $\Tor_{H/ k} (k, k)$.
\end{exa}

\chapter{Unanswered questions}\label{GRANDFINALE}

In this brief final chapter, we state some problems that came up during the preparation of this thesis that remain unsolved.

\begin{ques}
In Chapter~\ref{EXAMPLES} we saw related algebraic examples of duplicial objects arising from Theorem~\ref{dup}. What is missing is a brand new cyclic homology theory. We do however, show that the non-empty list functor $L^+$ becomes a bimonad (Section~\ref{newbimonad}). Is there a functor $N \colon \mathbf{SemiGrp} \to \cY$ for some category $\cY$ that admits the structure of an opcoalgebra over a distributive law, which gives rise to an interesting cyclic homology theory of semigroups?
\end{ques}

\begin{ques}
The duplicial object $\rC_{T_{\cH^*}}(N, I)$ of Corollary~\ref{endcor} can be used to describe duplicial structures on the nerves of categories (Section~\ref{nerve}) as well as on the simplicial object with homology $\Tor_{H/k} (k, k)$ for a Hopf algebra $H$ over a commutative ring $k$ (Section~\ref{coeffsforhopf}). However, it says nothing of more general Hopf algebroids over a noncommutative base algebra $A$. Is there a way to upgrade the construction method of $\rC_{T_{\cH^*}}(N, I)$ so that it admits a special case of the duplicial object in Section~\ref{coeffsforhopf} as an example?
\end{ques}

\begin{ques}
In Theorem~\ref{auspara} we describe duplicial structures on the nerves of categories $\cC$ in terms of
adjunctions
$$
\xymatrix{
\cG \ar@{}[rr]|-{\perp}\ar@/^0.5pc/[rr]^-I & & \ar@/^0.5pc/[ll]^-R \cC
}$$
where $\cG$ is a groupoid. If there exists two groupoids $\cG, \cG'$ equipped with left adjoints into $\cC$, then we have $\cG \simeq \cG'$, as explained in Remark~\ref{catrem}. How are the two duplicial structures on $N\cC$ induced by $\cG$ and $\cG'$ related?
\end{ques}

\begin{ques}
Suppose that we are in the situation of Section~\ref{bimfromhopf}, in the special case that we have a bialgebra $H$ over a commutative ring $k$. There is a distributive law $\chi$ defined by
$$
\xymatrix@R=1em{
H \otimes H \otimes X \ar[r]^-{\chi_X} & H \otimes H \otimes X \\
h \otimes g \otimes x \ar@{|->}[r] & h_{(1)}g \otimes h_{(2)} \otimes x.
}$$
If $H$ is a Hopf algebra, the functor $k \otimes_H{-} \colon H\mbox{-}\mathbf{Mod} \to k\mbox{-}\mathbf{Mod}$ becomes a $\chi$-opcoalgebra, with structure morphism $\lambda$ defined by
$$
\xymatrix@R=1em{
   k \otimes_H (H \otimes X) \ar[r]^-{\lambda_X} & k \otimes_H (H \otimes X) \cong X \\
1 \otimes_H (h \otimes x) \ar@{|->}[r] & S(h)x.
}
$$
Does the converse hold; i.e.~does a $\chi$-opcoalgebra structure on $k \otimes_H{-}$ imply that $H$ must be a Hopf algebra?
\end{ques}

\begin{ques}
Is there an interesting application of the work in Section~\ref{HOCHLAX} other than duplicial objects?
\end{ques}

\begin{ques}
Yetter-Drinfel'd modules seem to be of some importance with regard to dupliciality/cyclicity. For example:
\begin{itemize}
\item One requires a Yetter-Drinfel'd module $N$ in Section~\ref{coeffsforhopf} to define a $\chi$-opcoalgebra.
\item In the language of Section~\ref{coeffsforhopf} it is shown that the duplicial object $\rCC_T({-}\otimes_H N,M)$ is cyclic if $M \otimes_{A^*} N$ is a stable anti Yetter-Drinfel'd module.
\item A monoidal category $\cA$ becomes a 0-cell in $\AmodA$ with the regular actions (c.f.~Section~\ref{HOCHLAX}). We get that $\rH^0(\cA, \cA)$ is isomorphic to the lax centre of the monoidal category $\cA$ (see~\cite{MR2381533,MR1800718}). When $\cA$ is the category of modules over a Hopf algebra, this is the category of Yetter-Drinfel'd modules.
\end{itemize}
Is there a general phenomenon at work which explains these connections?
\end{ques}

\bibliography{Bibliography}
\bibliographystyle{abbrv}

\end{document}